\documentclass[11pt]{amsart}
\usepackage[margin=1in]{geometry}

\usepackage{hyperref}
\usepackage{url}

\usepackage{mathrsfs}
\usepackage{multirow}
\usepackage{amssymb}
\usepackage{graphicx}

\newtheorem{thm}{Theorem}

\newtheorem{lem}[thm]{Lemma}

\newtheorem{prop}[thm]{Proposition}

\newtheorem{cor}[thm]{Corollary}

\newtheorem{defn}[thm]{Definition}

\newtheorem{ex}[thm]{Example}

\newtheorem{rem}[thm]{Remark}
\newtheorem{remark}[thm]{Remark}

\numberwithin{thm}{section}
\numberwithin{equation}{section}
\numberwithin{theorem}{section}

\usepackage{amsbsy}

\usepackage{mdwlist} 
\usepackage{enumerate} 
\usepackage{multicol} 

\usepackage{paralist}
\usepackage{amsmath}
\usepackage{amsthm}
\usepackage{subfigure}
\usepackage{multicol}
\usepackage{caption}
 \usepackage{tikz,tikz-cd}

\newcommand\myemph\emph

\newcommand \taui[1] {\tau_{i_{#1}}}

\newcommand\ga{\gamma}

\newcommand \om {\omega}

\newcommand\puncture{\text{\tiny $P$}}

\newcommand{\Brac}{\text{Brac}} 
\newcommand{\Cheb}{\text{Cheb}} 
\newcommand{\F}{\mathcal{F}} 

\newcommand\notchP{^{(\text{\tiny $P$})}}

\newcommand\punctureP{\text{\tiny{P}}}

\newcommand{\ellOpOrientation}{counterclockwise}
\newcommand\disk{{disk}}
\newcommand\disks{{disks}}

\newcommand\cross{\text{cross}}

\newcommand\signedadjacencymatrixofanidealtriangulation{signed adjacency matrix of an ideal triangulation}

\usepackage{tikz}
\usetikzlibrary{patterns}
\usetikzlibrary{positioning}
\usetikzlibrary{matrix,arrows,calc,shapes}
\usetikzlibrary{arrows,automata}
\usetikzlibrary{decorations.markings}
\usetikzlibrary{matrix}

\definecolor{myblue}{cmyk}{1.00,0.56,0.00,0.34}
\definecolor{mygreen}{cmyk}{0.5,0,0.5,0.5}
\definecolor{myred}{cmyk}{0.00,1.00,0.63,0.00}

\usetikzlibrary{arrows,automata}

\usetikzlibrary{decorations.pathreplacing,decorations.markings}

\tikzset{->-/.style={decoration={
  markings,
  mark=at position #1 with {\arrow{stealth}}},postaction={decorate}}}

\tikzset{
  on each segment/.style={
    decorate,
    decoration={
      show path construction,
      moveto code={},
      lineto code={
        \path [#1]
        (\tikzinputsegmentfirst) -- (\tikzinputsegmentlast);
      },
      curveto code={
        \path [#1] (\tikzinputsegmentfirst)
        .. controls
        (\tikzinputsegmentsupporta) and (\tikzinputsegmentsupportb)
        ..
        (\tikzinputsegmentlast);
      },
      closepath code={
        \path [#1]
        (\tikzinputsegmentfirst) -- (\tikzinputsegmentlast);
      },
    },
  },
  mid arrow/.style={postaction={decorate,decoration={
        markings,
        mark=at position .5 with {\arrow[#1]{stealth}}
      }}},
}
\def\Prufer#1#2#3#4{
	\coordinate (x) at (#1,#2);
	\fill (x) circle (.1);
	\draw[#4] (x) .. controls (#1,#2+2.5) and (#1+.5,#2+3) .. (#1+#3,#2+3);
}

\newcommand\PruferShort[4]{
	\coordinate (x) at (#1,#2);
	\fill (x) circle (.1);
	\draw[#4] (x) .. controls (#1,#2+4) and (#1+.1,#2+4) .. (#1+#3,#2+3.8);
}

\newcommand\PruferShortCurvedLeft[4]{
	\coordinate (x) at (#1,#2);
	\fill (x) circle (.1);
	\draw[#4] (x) .. controls (#1,#2+4) and (#1-.1,#2+4) .. (#1-#3,#2+3.8);
}

\newcommand\MacroTriangle{
\draw[thick]
(0,0) -- (4,0) -- (2,4) -- (0,0);
\draw [fill=olive] (0,0) circle (.6ex);
\draw [fill=olive] (4,0) circle (.6ex);
\draw [fill=olive] (2,4) circle (.6ex);
}
\newcommand\TikzTriangle[1]{
\begin{tikzpicture}[scale={#1} ]
\MacroTriangle
\node at (4.5,0) {}; \node at (-0.5,0) {};
\end{tikzpicture}
}

\newcommand\MacroTwoVertexTriangle{ 
\draw[style = thick]
	(0,0) .. controls (-1.7,1) and (-1.6,3) .. (0,3.5) ..
	controls (1.6,3)  and (1.7,1) .. (0,0);
	\draw [fill=olive] (0,0) circle (.6ex);
	\draw [fill=olive] (0,3.5) circle (.6ex);

\draw[style = thick, rounded corners = 20pt] 
	(0,0) .. controls (-0.8,1) and (-0.5,1.7) .. (0,2.5) ..
	controls (0.5,1.7)  and (0.8,1) .. (0,0);
}
\newcommand\TikzTwoVertexTriangle[1]{ 
\begin{tikzpicture}[scale={#1} ]
\MacroTwoVertexTriangle
\end{tikzpicture}
}

\newcommand\MacroVeryThinNoose{
\draw[style = thick, rounded corners = 10pt] 
	(0,0) .. controls (-0.6,1) and (-0.4,1.7) .. (0,2.2) ..
	controls (0.4,1.7)  and (0.6,1) .. (0,0);
}

\newcommand\MacroOneVertexTriangle{
\draw[style = thick, rounded corners = 30pt]
	(0,0) .. controls (-2.7,1) and (-2.4,2.5) .. (0,3.5) ..
	controls (2.4,2.5)  and (2.7,1) .. (0,0);
	
	\draw [fill=olive] (0,0) circle (.6ex);
	
\begin{scope}[rotate=30]
\MacroVeryThinNoose
\end{scope}

\begin{scope}[rotate=-35]
\MacroVeryThinNoose
\end{scope}
}

\newcommand\TikzOneVertexTriangle[1]{ 
\begin{tikzpicture}[scale={#1} ]
\MacroOneVertexTriangle
\end{tikzpicture}
}

\newcommand\MacroSelfFoldedTriangleNoLabel{
	\draw[thick]
	(0,0) -- (0,2);
	
	\draw [fill=violet] (0,2) circle (1.2*\TikzPunctureActualSize);
	
	\draw[style = thick, rounded corners = 20pt]
	(0,0) .. controls (-1.8,1) and (-1.5,2.5) .. (0,3.5) ..
	controls (1.5,2.5)  and (1.8,1) .. (0,0); 
	\draw [fill=olive] (0,0) circle (.6ex);
}
\newcommand\TikzSelfFoldedTriangle[1]{ 
\begin{tikzpicture}[scale={#1} ]
\MacroSelfFoldedTriangleNoLabel
\draw (-1.6,2) node {$\ell$};
\draw (0.3,1) node {$r$};
\end{tikzpicture}
}

\newcommand\TikzPunctureActualSize{0.1}

\begin{document}

\title{Cluster algebraic interpretation of infinite friezes}

\author[Gunawan]{Emily Gunawan}
\address{University of Connecticut, Storrs, Connecticut USA}
\email{egunawan@math.umn.edu}
\author[Musiker]{Gregg Musiker}
\address{Department of Mathematics, University of Minnesota, Minneapolis, Minnesota, USA}
\email{musiker@math.umn.edu}
\author[Vogel]{Hannah Vogel}
\address{Department of Mathematics and Scientific Computing, University of Graz, Graz, Austria}
\email{hannah.vogel@uni-graz.at}

\date{\today}


\begin{abstract}
Originally studied by Conway and Coxeter, friezes appeared
in various recreational mathematics publications in the 1970s.   More recently, in 2015, Baur, Parsons, and Tschabold 
constructed periodic infinite friezes and related them to matching numbers in the once-punctured disk and annulus.  In this paper, we
study such infinite friezes with an eye towards cluster algebras of type D and affine A, respectively.  By examining infinite friezes 
with Laurent polynomial entries, we discover new symmetries and formulas relating the entries of this frieze to one another. 
Lastly, we also present a correspondence between Broline, Crowe and Isaacs's classical matching tuples and combinatorial interpretations of elements of cluster algebras from
surfaces.
\end{abstract}

\maketitle

\tableofcontents

\section{Introduction}
\label{sec:intro}

A Conway-Coxeter \emph{frieze} $\F=
\{\F_{ij}\}_{i \leq j}$ 
 is an array of rows (arranged and indexed as in~Fig.~\ref{fig:defn_integer_infinite_frieze}) 
such that $\F_{i,i} = 0$ and $\F_{i,i+1} = 1$, and, 
 for every diamond
\[
\begin{array}{ccccccc}
 &c&\\
 a&&b\\
 &d&
\end{array}\]
of entries in the frieze, the equation $ab-cd =1$ is satisfied. 

\begin{figure}[!hbt]
\centering
\begin{tikzpicture}[font=\footnotesize] 
\matrix(m) [matrix of math nodes,row sep={1.35em,between origins},column sep={1.55em,between origins},nodes in empty cells]{
&0&&0&&0&&0&&0&&&&&&\\ 
\node{\cdots};&&1&&1&&1&&1&&1&&\node{\cdots};&&\\[-0.25em]
&&&\F_{-1,1}&&\F_{0,2}&&\F_{1,3}&&\F_{2,4}&&\F_{3,5}&&&\\
&&\node{\cdots};&&\F_{-1,2}&&\F_{0,3}&&\F_{1,4}&&\F_{2,5}&&\F_{3,6}&&\node{\cdots};\\
&&&&&\F_{-1,3}&&\F_{0,4}&&\F_{1,5}&&\F_{2,6}&&\F_{3,7}&\\
&&&&&&&&\node[rotate=-6.5,shift={(-0.034cm,-0.08cm)}] {\ddots};&&&&\node[rotate=-6.5,shift={(-0.034cm,-0.08cm)}]  {\ddots};&&\\
};
\end{tikzpicture}
\caption{Arrangement and indices for frieze entries.}
\label{fig:defn_integer_infinite_frieze}
\newcommand\pentascale{0.31}
\newcommand\myyellow{white}
\newcommand\mygray{gray}
\tikzset{->-/.style={decoration={
  markings,
  mark=at position #1 with {\arrow{stealth}}},postaction={decorate}}}
 \newcommand\mypentagon[1]
 {
 \begin{scope}[scale=\pentascale]
 \coordinate (a) at (0,0);
\foreach \x in {-144,-72,0,72,144} {
   \begin{scope}[rotate=\x]
    \coordinate (\x) at (0,-1.25) {};
   \end{scope}
}
\draw [fill=#1,opacity=0.55] (-144) -- (-72) -- (0) -- (72) -- (144) -- (-144);
\draw [opacity=0.65] (0) -- (144) (-144) -- (0);
\end{scope}
}

 \begin{tikzpicture}
\matrix[outer sep=0pt,column sep=-0.2cm,row sep=-0.1cm,
ampersand replacement=\&]
{
\mypentagon{\mygray}
\draw[thick,->-=0.66] (-72) -- (0); 
\draw (-72) node [blue,fill,circle,inner sep=1pt] {};
\draw (0) node [red,fill,circle,inner sep=1pt] {};
\&
\&
\mypentagon{\mygray}
\draw[thick,->-=0.7] (0) -- (72); 
\draw (0) node [red,fill,circle,inner sep=1pt] {};
\draw (72) node [fill=mygreen,circle,inner sep=1pt] {};
\&
\&
\mypentagon{\mygray}
\draw[thick,->-=0.6] (72) -- (144); 
\draw (72) node [mygreen,fill,circle,inner sep=1pt] {};
 \draw (144) node [violet,fill,circle,inner sep=1pt] {};
\&\&
\mypentagon{\mygray}
\draw[thick,->-=0.65] 
(144) -- (-144);
 \draw (144) node [violet,fill,circle,inner sep=1pt] {};
 \draw (-144) node [magenta,fill,circle,inner sep=1pt] {};
\&
\&
\mypentagon{\myyellow}
\draw[thick,->-=0.7] (-144) -- (-72); 
\draw (-144) node [magenta,fill,circle,inner sep=1pt] {};
\draw (-72) node [blue,fill,circle,inner sep=1pt] {};
\&
\&
\mypentagon{\mygray}
\draw[thick,->-=0.6] (-72) -- (0); 
\draw (-72) node [blue,fill,circle,inner sep=1pt] {};
\draw (0) node [red,fill,circle,inner sep=1pt] {};
\\
\draw (-1,0) node {$\cdots$};
\&
\mypentagon{\mygray} 
\draw[thick,->-=0.6] (-72) -- (72); 
\draw (-72) node [fill,blue,circle,inner sep=1pt] {};
\draw (72) node [mygreen,fill,circle,inner sep=1pt] {};
\&
\&
\mypentagon{\mygray}
\draw[thick,->-=0.6] (0) -- (144); 
\draw (0) node [red,fill,circle,inner sep=1pt] {};
 \draw (144) node [violet,fill,circle,inner sep=1pt] {};
\&
\&
\mypentagon{\mygray}
\draw[thick,->-=0.6] (72) -- (-144); 
\draw (72) node [fill=mygreen,circle,inner sep=1pt] {};
\draw (-144) node [fill=magenta,circle,inner sep=1pt] {};
\&
\&
\mypentagon{\myyellow}
\draw[thick,->-=0.6] (144) -- (-72); 
\draw (144) node [violet,fill,circle,inner sep=1pt] {};
\draw (-72) node [blue,fill,circle,inner sep=1pt] {};
\&
\&
\mypentagon{\myyellow}
\draw[thick,->-=0.6] (-144) -- (0); 
\draw (-144) node [magenta,fill,circle,inner sep=1pt] {};
\draw (0) node [red,fill,circle,inner sep=1pt] {};
\&
\&
\mypentagon{\mygray}
\draw[thick,->-=0.6] (-72) -- (72); 
\draw (-72) node [blue,fill,circle,inner sep=1pt] {};
\draw (72) node [fill=mygreen,circle,inner sep=1pt] {};
 \\
  \mypentagon{\myyellow}
  \draw[thick,->-=0.6] (-144) -- (72); 
  \draw (-144) node [magenta,fill,circle,inner sep=1pt] {};
  \draw (72) node [fill=mygreen,circle,inner sep=1pt] {};
  \&
  \&
\mypentagon{\mygray}
\draw[thick,->-=0.6] (-72) -- (144); 
\draw (-72) node [blue,fill,circle,inner sep=1pt] {};
 \draw (144) node [violet,fill,circle,inner sep=1pt] {};
\&
\&
\mypentagon{\mygray}
\draw[thick,->-=0.6] (0) -- (-144); 
\draw (0) node [red,fill,circle,inner sep=1pt] {};
\draw (-144) node [magenta,fill,circle,inner sep=1pt] {};
\&
\&
\mypentagon{\myyellow}
\draw[thick,->-=0.6] (72) -- (-72); 
\draw (72) node [fill=mygreen,circle,inner sep=1pt] {};
\draw (-72) node [blue,fill,circle,inner sep=1pt] {};
\&
\&
\mypentagon{\myyellow}
\draw[thick,->-=0.6] (144) -- (0); 
\draw (144) node [violet,fill,circle,inner sep=1pt] {};
\draw (0) node [red,fill,circle,inner sep=1pt] {};
\&
\&
\mypentagon{\myyellow}
\draw[thick,->-=0.6] (-144) -- (72); 
\draw (-144) node [magenta,fill,circle,inner sep=1pt] {};
\draw (72) node [fill=mygreen,circle,inner sep=1pt] {}; 
\\
\draw (-1,0) node {$\cdots$};
\&
\mypentagon{\myyellow}
\draw[thick,->-=0.65] (-144) -- (144); 
\draw (-144) node [magenta,fill,circle,inner sep=1pt] {};
\draw (144) node [violet,fill,circle,inner sep=1pt] {};
\&
\&
\mypentagon{\mygray}
\draw[thick,->-=0.65] (-72) -- (-144); 
\draw (-72) node [blue,fill,circle,inner sep=1pt] {};
\draw (-144) node [magenta,fill,circle,inner sep=1pt] {};
\&
\&
\mypentagon{\myyellow}
\draw[thick,->-=0.7] (0) -- (-72); 
\draw (0) node [red,fill,circle,inner sep=1pt] {};
\draw (-72) node [blue,fill,circle,inner sep=1pt] {};
\&
\&
\mypentagon{\myyellow}
\draw[thick,->-=0.6] (72) -- (0); 
\draw (72) node [fill=mygreen,circle,inner sep=1pt] {};
\draw (0) node [red,fill,circle,inner sep=1pt] {};
\&
 \&
\mypentagon{\myyellow}
\draw[thick,->-=0.7] (144) -- (72); 
\draw (144) node [violet,fill,circle,inner sep=1pt] {};
\draw (72) node [fill=mygreen,circle,inner sep=1pt] {};
\&
\&
\mypentagon{\myyellow}
\draw[thick,->-=0.7] (-144) -- (144); 
\draw (-144) node [magenta,fill,circle,inner sep=1pt] {};
\draw (144) node [violet,fill,circle,inner sep=1pt] {};
\\
};
\end{tikzpicture}
\input{friezepentagonintyellow}
\caption{Diagonals of a polygon correspond to
entries of a finite frieze. 
}
\label{fig:friezepentagonyellow}
\def\mylocalscale{0.8}
\def\mylocalblue{blue!30}
\def\mylocalviolet{violet}
\def\mylocalpathcolor{black}
\def\localTemplatePolygon{
 \coordinate (a) at (0,0);
\foreach \x in {-144,-72,0,72,144} {
   \begin{scope}[rotate=\x]
    \coordinate (\x) at (0,-1.25) {}; 
    \end{scope}
}
\draw[opacity=0.65] (-144) -- (-72) -- (0) -- (72) -- (144) -- (-144);
\draw (0) node [below, violet] {   $\mathbf v_1$};
\draw (72) node [right, blue] {    $\mathbf v_2$};
\draw (144) node [right] {   $v_3$};
\draw (-144) node [left] {   $v_4$};
\draw (-72) node [left] {   $v_5$};
\draw[opacity=0.4] (0) -- (144) (-144) -- (0);

\draw [fill=\mylocalblue,opacity=.5] (0) -- (144) -- (72) -- (0);

}

\begin{tikzpicture}[scale=\mylocalscale,font=\small]
 \coordinate (a) at (0,0);
\foreach \x in {-144,-72,0,72,144} {
   \begin{scope}[rotate=\x]
    \node (\x) at (0,-1.25) [fill,circle,inner sep=0.5pt] {};
   \end{scope}
}
\draw[opacity=0.65] (-144) -- (-72) -- (0) -- (72) -- (144) -- (-144);
\draw (0) node [below] {  $v_1$};
\draw (72) node [right] {    $v_2$};
\draw (144) node [right] {  $v_3$};
\draw (-144) node [left] {   $v_4$};
\draw (-72) node [left] {   $v_5$};
\draw[opacity=0.65] (0) -- (144) (-144) -- (0);
\draw[->-=.7,red,very thick, densely dotted] (-72) -- (144) node[left,pos=0.6]{$\gamma$};
\draw (0,-2.5) node {$T$ and \textcolor{red}{$\gamma$}};
\draw (0,-3) node {};
\draw (-0.4,-0.2) node[magenta] {$b$};
\draw (0.4,-0.2) node[black] {$a$};
\end{tikzpicture}
\begin{tikzpicture}[scale=\mylocalscale,font=\small]
\localTemplatePolygon
\draw[fill=\mylocalviolet,opacity=.6] (0) -- (-144) -- (-72) -- (0);
\draw [ultra thick, \mylocalpathcolor, ->-=0.5] (-72) -- (-144) node[pos=0.5,left] {\small $1$};
\draw [ultra thick, \mylocalpathcolor, ->-=0.5] (-144) -- (0) node[pos=0.3,right,magenta] {$b$};
\draw [ultra thick, \mylocalpathcolor, ->-=0.5] (0) -- (144) node[pos=0.7,left,black] {$a$};
\draw (0,-2.3) node []{$(t_1,t_2)=(\textcolor{violet}{A},\textcolor{blue}{C})$};
\draw (0,-3) node {$1\, b^{-1} \, a $};
\draw (-0.8,0) node [black] {$A$};
\draw (0.7,0) node [black] {$C$};
\end{tikzpicture}
\begin{tikzpicture}[scale=\mylocalscale,font=\small]
\localTemplatePolygon
\draw[fill=\mylocalviolet,opacity=.6] (0) -- (144) -- (-144) -- (0);
\draw[ultra thick, \mylocalpathcolor, ->-=0.5] (-72) -- (0) node[pos=0.3,right, below] {\small $1$};
\draw[ultra thick, \mylocalpathcolor, ->-=0.5] (0) -- (-144) node[pos=0.6,left,magenta] {$b$};
\draw[ultra thick, \mylocalpathcolor, ->-=0.5] (-144) -- (144)node[pos=0.6,above] {\small $1$};
\draw (0,-2.3) node []{$(t_1,t_2)=(\textcolor{violet}{B},\textcolor{blue}{C})$};
\draw (0,-3) node {$1\, b^{-1} \, 1 $};

\draw (0,0) node [black] {$B$};
\draw (0.7,0) node [black] {$C$};
\end{tikzpicture}
\caption{The BCI $2$-tuples for $\gamma$ which match vertices $v_1$ \& $v_2$ to their adjacent triangles, and
the corresponding trails  
whose weights sum up to the expansion $x_{\textcolor{red}{\gamma}}=\frac{a}{\textcolor{magenta}{b}} + \frac{1}{\textcolor{magenta}{b}}$.}
\label{fig:bci_trail}
\end{figure}
 
 We say a frieze is \emph{finite} if it is bounded above and below by a row of $1$s.
In the 70s, Conway and Coxeter showed that finite friezes with positive integer entries are in bijection with triangulations of polygons~\cite{Cox71,CC73}.
Given a triangulation $T$ of a polygon, each entry of the second row of the corresponding frieze is the number of triangles adjacent to a vertex.
Broline, Crowe, and Isaacs further studied this in~\cite{BCI74} and 
found that every entry in such a frieze corresponds to a diagonal (see~Fig.~\ref{fig:friezepentagonyellow}).
To any diagonal, they associate a set of vertices $v_{i_1}, \ldots, v_{i_r}$ (those lying to the right) and then match these to a \emph{BCI $r$-tuple} $(t_{1}, \ldots, t_{r})$ of pairwise-distinct triangles in $T$, such that $t_{j}$ is incident to vertex $v_{i_j}$. 
For example, in~Fig.~\ref{fig:bci_trail}, the diagonal from vertex $v_5$ to vertex $v_3$ is associated to the vertices $v_1$ and $v_2$. There are exactly two BCI $2$-tuples corresponding to $v_1$ and $v_2$.
 
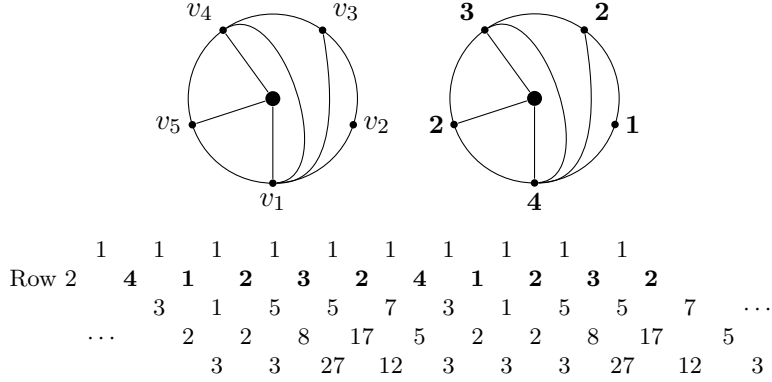
\begin{figure}[!hbt]
\def\scalepuncturedpentagon{0.6}
\begin{tikzpicture}[scale=0.9,font=\normalsize] 
\node (a) at (0,0) [fill,circle,inner sep=2pt] {};
\draw (0,0) circle (1.25cm);
\foreach \x in {-144,-72,0,72,144} {
   \begin{scope}[rotate=\x]
    \node (\x) at (0,-1.25) [fill,circle,inner sep=1pt] {};
   \end{scope}
}
\draw (144) node [above right] {$v_3$};
\draw (72) node [right] {$v_2$};
\draw (0) node [below] {$v_1$};
\draw (-72) node [left] {$v_5$};
\draw (-144) node [above left] {$v_4$};
\draw[] (a) to (-144);
\draw[] (a) to (-72);
\draw[] (a) to (0);
\draw[out=5,in=30] (0) to  
(-144);
\draw[out=0,in=-80] (0) to  (144);
\end{tikzpicture}
\begin{tikzpicture}[scale=0.9,font=\small] 
\node (a) at (0,0) [fill,circle,inner sep=2pt] {};
\draw (0,0) circle (1.25cm);
\foreach \x in {-144,-72,0,72,144} {
   \begin{scope}[rotate=\x]
    \node (\x) at (0,-1.25) [fill,circle,inner sep=1pt] {};
   \end{scope}
}
\draw (0) node [below] {$\mathbf 4$}; 
\draw (72) node [right] {$\mathbf {1}$}; 
\draw (144) node [above right]  {$\mathbf 2$}; 
\draw (-144) node [above left] {$\mathbf 3$}; 
\draw (-72) node [left] {$\mathbf 2$}; 
\draw[] (a) to (-144); 
\draw[] (a) to (-72);
\draw[] (a) to (0);
\draw[out=5,in=30] (0) to  (-144);
\draw[out=0,in=-80] (0) to  (144);
\end{tikzpicture}
\begin{tikzpicture}
 \matrix(m) [matrix of math nodes,row sep={1.0em,between origins},column sep={1.0em,between origins},nodes in empty cells, font=\footnotesize]{
&&1&&1&&1&&1&&1&&1&&1&&1&&1&&1&\\
\text{Row $2$}&&&\mathbf{4}&&\mathbf{1}&&\mathbf{2}&&\mathbf{3}&&\mathbf{2}&&\mathbf{4}&&\mathbf{1}&&\mathbf{2}&&\mathbf{3}&&\mathbf{2}&\\
&&&&3&&1&&5&&5&&7&&3&&1&&5&&5&&7&&\node{\cdots};&\\
&&\node{\cdots};&&&2&&2&&8&&17&&5&&2&&2&&8&&17&&5&\\
&&&&&&3&&3&&27&&12&&3&&3&&3&&27&&12&&3&\\
};
\end{tikzpicture}
\caption{First $5$ rows
of the infinite frieze from a pentagon triangulation.}
\label{fig:puncturedpentagon}
\end{figure}
 
More recently,
Caldero and Chapoton in~\cite{CC06} showed that finite frieze patterns appear in the context of Fomin--Zelevinsky cluster algebras~\cite{FZ02} of type $A$. 
Carroll and Price in~\cite{CP03} gave an expansion formula for cluster variables of type $A$ in terms of BCI tuples (see Fig.~\ref{fig:bci_trail}). 
Enumerating BCI tuples is equivalent to counting perfect matchings in a bipartite graph whose nodes are the triangles and vertices of (resp. a snake graph associated to) a triangulation; see Sec.~2 (resp. Sec.~4) of~\cite{Pro05}.  

A frieze is said to be \emph{infinite} if it is not bounded below by a row of $1$s.  An infinite frieze is said to be \emph{periodic} if there exists an integer $n \geq 1$ such that each row is invariant under horizontal translation by $n$, i.e. $\mathcal{F}_{ij} = \mathcal{F}_{i+n,j+n}$ for all $i \leq j$.
Infinite friezes of positive integers arising from once-punctured {\disks} were introduced in~\cite{Tsc15} by Tschabold. Given an ideal triangulation $T$ (in the sense of~\cite{FST08}) of a once-punctured {\disk} with $n$ marked boundary vertices labeled $v_1,v_ 2, \ldots, v_n$ counterclockwise around the boundary, 
we can count the number of BCI tuples 
in a similar way, see Fig.~\ref{fig:puncturedpentagon}.

In~\cite{BPT16}, Baur, Parsons, and Tschabold went further and gave a complete characterization of infinite friezes of positive integers via triangulations of quotients of an infinite strip in the plane. In this classification, periodic friezes arise from triangulations of the annulus or of the once-punctured disk (which can be thought of as a quotient of the infinite strip), see Fig.~\ref{fig:disc_as_AT}.  
An infinite frieze is said to be of type $D$ or type $\tilde{A}$, if it arises from a once-punctured {\disk} or annulus, respectively. 
Related work on friezes of type D and $\tilde{A}$ 
include~\cite{Sch08Dn, BM09, ARS10,BR10,Smi15,BFPT16,FP16,GS18}.

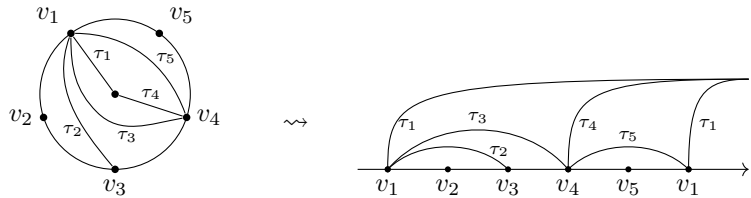
\begin{figure}[!htbp]
\centering
\subfigure{
\begin{tikzpicture}[scale=0.8,font=\small]
\node (a) at (0,0) [fill,circle,inner sep=1pt] {};
\draw (0,0) circle (1.25cm);
\foreach \x in {-144,-72,0,72,144} {
   \begin{scope}[rotate=\x]
    \node (\x) at (0,-1.25) [fill,circle,inner sep=1pt] {};
   \end{scope}
}
\draw (-72) node [left] {$v_2$};
\draw (0) node [below] {$v_3$};
\draw (72) node [right] {$v_4$};
\draw (144) node [above right] {$v_5$};
\draw (-144) node [above left] {$v_1$};

\draw[thin] (a) to (-144);
\draw[thin] (a) to (72);
\draw[thin,out=110,in=0] (72) to  (-144);
\draw[thin,out=-110,in=130] (-144) to (0);
\draw[thin,out=-170,in=-45] (72) to (-0.3,-0.3) ;
\draw[thin,out=135,in=-90] (-0.3,-0.3) to(-144);

\fill(-.2,0.65) node[] {\tiny{$\tau_1$}};
\fill(-.7,-.6) node[] {\tiny{$\tau_2$}};
\fill(.2,-.7) node[] {\tiny{$\tau_3$}};
\fill(.6,0) node[] {\tiny{$\tau_4$}};
\fill(.85,.6) node[] {\tiny{$\tau_5$}};

\fill(3,-0.5) node[] {$\rightsquigarrow$};
\end{tikzpicture}

\quad

		\begin{tikzpicture}[scale = .4]
		\tikzstyle{every node} = [font = \small]
		\foreach \x in {0}
		{
			\foreach \y in {-8}
			{
				\draw[->] (\x-7,\y-3) -- (\x+6,\y-3);
			
				\foreach \t in {-6,-4,-2,0,2,4}
				{
					\fill (\x+\t,\y-3) circle (.1);
				}

				\fill (\x-6,\y-3) node [below] {$v_1$};
				\fill (\x-4,\y-3) node [below] {$v_2$};
				\fill (\x-2,\y-3) node [below] {$v_3$};
				\fill (\x+0,\y-3) node [below] {$v_4$};
				\fill (\x+2,\y-3) node [below] {$v_5$};
				\fill (\x+4,\y-3) node [below] {$v_1$};
				
				\Prufer{\x-6}{\y-3}{12}{}
				\fill (\x-6,\y-1.5) node [right] {\tiny{$\tau_1$}};
				\draw[] (\x-6,\y-3) .. controls (\x-5,\y-2) and (\x-3,\y-2) .. (\x-2,\y-3);
				\fill (\x-2.3,\y-2.9) node [above] {\tiny{$\tau_2$}};
				\draw[] (\x-6,\y-3) .. controls (\x-4.5,\y-1.25) and (\x-1.5,\y-1.25) .. (\x-0,\y-3);
				\fill (\x-3,\y-1.8) node [above] {\tiny{$\tau_3$}};
				\Prufer{\x-0}{\y-3}{6}{}
				\fill (\x-0,\y-1.5) node [right] {\tiny{$\tau_4$}};	
				\draw[] (\x+0,\y-3) .. controls (\x+1,\y-2) and (\x+3,\y-2) .. (\x+4,\y-3);
				\fill (\x+2,\y-2.4) node [above] {\tiny{$\tau_5$}};	
				\Prufer{\x+4}{\y-3}{2}{}
				\fill (\x+4,\y-1.5) node [right] {\tiny{$\tau_1$}};	
			}
		}	
		\end{tikzpicture}
		}	
\caption{Triangulation of a once-punctured pentagon drawn as an asymptotic triangulation 
(see~\cite[Lemma 3.6]{BPT16}).}
\label{fig:disc_as_AT}
\end{figure}

In this paper, we expand on the work of~\cite{Tsc15, BPT16} and 
 study periodic infinite friezes whose entries are Laurent polynomials (as opposed to positive integers). Our first result (Theorem~\ref{thm:Frieze_Laurent}) is a construction of an infinite frieze where the entries of the frieze correspond to Laurent polynomials associated to \emph{generalized peripheral arcs} (Definition~\ref{defn:array_of_peripheral_arcs}).  Here, our construction uses a once-punctured disk or annulus with marked points on its outer boundary.  A \emph{peripheral arc} is a curve that has these aforementioned marked points as its endpoints while possibly wrapping around the inner boundary or the internal puncture.  We extend our construction to \emph{generalized} peripheral arcs by allowing curves with self-intersections.  See Figure \ref{fig:univ_cover}.  Our association of Laurent polynomials to generalized peripheral arcs appeared in more generality in \cite{MW13,MSW13}, where it appeared in the context of cluster algebras from surfaces.  See Section~\ref{sec:cluster_algebras_from_surfaces}.  Note that it follows from this association that any periodic friezes considered in~\cite{Tsc15, BPT16} can be obtained from our construction by specializing all variables of Laurent polynomials to one.

We go on to describe nice symmetries and properties that this frieze pattern satisfies.
In Section~\ref{sec:proof_progression_formulas}, we introduce \emph{complementary arcs}, which are arcs between the same two vertices in a surface, but of alternate direction. Complementary arcs give rise to a special type of symmetry, which we call \emph{complement symmetry}, in an infinite frieze. This complement symmetry reduces to glide-symmetry in a finite frieze pattern because complementary arcs in 
a polygon
give rise to the same entry (matching number/arc/Laurent polynomial) in a finite frieze.
We use these complementary arcs to describe \emph{progressions} of arcs in the frieze   (Theorem~\ref{thm:progression_formula}).
In Section~\ref{sec:brac_and_growth}, we discuss \emph{growth coefficients} (first defined in~\cite{BFPT16}) of the frieze, and show that they are equal to Laurent polynomials corresponding to certain curves called \emph{bracelets} in the surface. 
Bracelets are associated to important cluster algebra elements~\cite{SZ04,DT13,MW13,MSW13,DT14,CS15a,CS15b}.  
In Section \ref{sec:recursive}, we state further algebraic and combinatorial results involving the relationship between complementary arcs. 

Finally, recall from above, the work of Broline, Crowe, and Isaacs \cite{BCI74} provided one of the original combinatorial interpretations for entries of finite friezes outside of the quiddity row, namely BCI tuples.  For infinite friezes, the analogous definition has not previously appeared, although different combinatorial interpretations for peripheral arcs, via T-paths, have appeared (as in~\cite{Sch08Tpath,ST09,GM15}).
Inspired by this, in Appendix~\ref{appendix:BCI_Tpath}, we extend previous unpublished work of Carroll-Price~\cite{CP03} to provide a bijection between BCI tuples~\cite{BCI74} and $T$-paths.  Our bijection yields an expansion formula for cluster variables in terms of BCI tuples (Corollary~\ref{rem:bci_formula}) and 
preserves a natural distributive lattice structure which the $T$-paths are known to have (Proposition~\ref{thm:lattice_tuple1}).

We begin our article with Section~\ref{sec:cluster_algebras_from_surfaces}, which introduces the necessary background material, including  the notation and terminology of arcs, triangulations, and cluster algebras that we will use throughout the article.  We also recall cluster algebra elements associated to generalized arcs and closed loops (with or without self-crossings) via snake graphs and band graphs, as per~\cite{MW13,MSW11,MSW13}.  Then in Section~\ref{sec:background_infinite_frieze}, we recall some facts about infinite friezes, and explain how triangulations (of once-punctured {\disks} and annuli) give rise to infinite friezes.  The remainder of our article provides the statements of our results and their proofs as previewed above. 

\section*{Acknowledgments}
E. Gunawan and G. Musiker were supported by NSF Grants DMS-1148634 and DMS-1362980. 
H. Vogel was supported by the Austrian Science Fund (FWF): projects No.\ P25141-N26 and W1230, and acknowledges support from NAWI Graz. She would also like to thank the University of Minnesota for hosting her during her stay in the Winter of 2016.

We thank Karin Baur for helpful comments, and Manuela Tschabold for allowing us to use some of her tikz figures. Some of the images and Laurent polynomial computation was done with the help of {\sc SageMath}~\cite{sage, combinat} and code written by Ana Garc\'{i}a Elsener and Jorge Nicol\'{a}s L\'{o}pez. We also thank the referees for many useful comments. 

\section{Cluster algebras from surfaces}
\label{sec:cluster_algebras_from_surfaces}

We provide a brief background on cluster algebras arising from marked surfaces following Fomin, Shapiro, and Thurston \cite{FST08}. 

\subsection{Triangulations of marked surfaces}
\label{subsec:triangulations}

\begin{defn}
[marked surface]
Let $S$ be a connected, oriented, Riemann surface with (possibly empty) boundary, and $M$ a non-empty, finite set of marked points in the closure of $S$, such that there is at least one marked point on each boundary component of $S$. Then $(S,M)$ is called a \emph{marked surface}, and the interior marked points of $S$ are called \emph{punctures}. 
\end{defn}

For technical reasons, assume that $(S,M)$ is not the following: a sphere with fewer than four punctures; a monogon with zero or one puncture; or a bigon or triangle without punctures.

\begin{defn}[ordinary arc]
An \emph{ordinary arc $\gamma$ in $(S,M)$} is a curve in $S$, considered up to isotopy, such that: (1) the endpoints of $\gamma$ are in $M$, (2) $\gamma$ does not cross itself (except its endpoints may coincide), (3) the interior of $\gamma$ is disjoint from $M$ and from the boundary of $S$, and (4) $\gamma$ does not cut out an unpunctured monogon or bigon.  
A \emph{boundary edge} is a curve that connects two marked points and lies entirely on the boundary of $S$ without passing through a third marked point.
\end{defn}
We say that two ordinary arcs $\alpha, \beta$ are \emph{compatible} if there exist representatives $\alpha', \beta'$ in their respective isotopy classes such that $\alpha'$ and $\beta'$ do not intersect in the interior of $S$.

\begin{defn}[ideal triangulation]
An \emph{ideal triangulation} is a maximal (by inclusion) collection of distinct, pairwise compatible ordinary arcs.
The ordinary arcs of an ideal triangulation cut the surface into \emph{ideal triangles} (see Fig.~\ref{fig:possible_types_of_ideal_triangles}).
\end{defn}

\begin{rem}[possible types of ideal triangles]
There are two types of ideal triangles in a triangulation: triangles that have three distinct sides (Figs.~\ref{fig:TikzTriangle}, ~\ref{fig:TikzTwoVertexTriangle}, and ~\ref{fig:TikzOneVertexTriangle}), and \emph{self-folded triangles} (Fig.~\ref{fig:TikzSelfFoldedTriangle}). A self-folded triangle consists of an arc $\ell$ (which we will refer to as an \emph{$\ell$-loop}) whose endpoints coincide, along with an arc $r$ (called a \emph{radius}) that goes from the endpoint of $\ell$ to an enclosed puncture. 
\end{rem}

\begin{figure}[h]
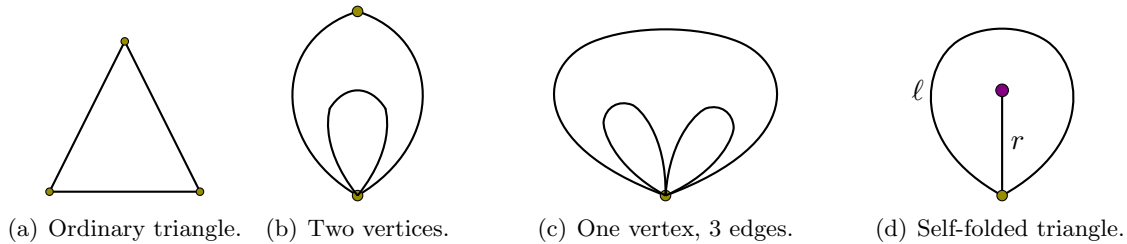

\centering
\mbox{\subfigure[Ordinary triangle.]
{
\TikzTriangle{0.5}
\label{fig:TikzTriangle}
}
\subfigure[Two vertices.]
{
\TikzTwoVertexTriangle{0.7}
\label{fig:TikzTwoVertexTriangle}
}
\subfigure[One vertex, 3 edges.]
{
\makebox[1.1\width]
{
\TikzOneVertexTriangle{0.7}
\label{fig:TikzOneVertexTriangle}
}
}
\subfigure[Self-folded triangle.]
{\quad
\TikzSelfFoldedTriangle{0.7}\quad
\label{fig:TikzSelfFoldedTriangle}
}
}
\caption{Possible types of ideal triangles.}
\label{fig:possible_types_of_ideal_triangles}
\end{figure}

\begin{defn}[peripheral arcs]\label{defn:types_of_arcs}
Let $\ga$ be an ordinary arc of a surface $(S,M)$ with nonempty boundary.
Following \cite{DT13,BD14}, we say that $\ga$ is a \emph{peripheral arc} if: (1) both its endpoints (or its unique endpoint in the case of a loop) are on a single boundary component $Bd$ of $S$, and (2) $\ga$ is isotopic to a concatenation of two or more boundary edges of a boundary component $Bd$.  Examples of non-peripheral arcs are those connecting a marked point on a boundary component to a puncture (or a different boundary component).
\end{defn}

\begin{defn}[flip for ordinary arc]
A \emph{flip} is a move that replaces an ordinary arc $\gamma$ in an ideal triangulation $T$ with the unique arc $\gamma' \neq \gamma$ such that $(T\setminus \gamma) \cup \gamma'$ forms a new ideal triangulation.
\end{defn}

Any two ideal triangulations of a surface are connected to each other by a sequence of flips (see Fig.~\ref{fig:seq_flips} for an example).

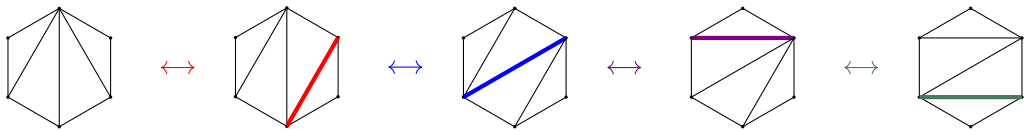
\begin{figure}[h!]
\centering
\def\localpicscale{.45}
\begin{tikzpicture}[scale = \localpicscale]
		\tikzstyle{every node} = [font = \small]
		\foreach \x in {0}
		{
			\foreach \y in {-8}
			{
			\fill (\x,\y+1.75) circle (.05);
			\fill(\x+1.5158,\y+.875) circle (.05);
			\fill (\x+1.5158,\y-.875) circle (.05);
			\fill(\x,\y-1.75) circle (.05);
			\fill(\x-1.5158,\y-.875) circle (.05);
			\fill (\x-1.5158,\y+.875) circle (.05);
		
			\draw [] (\x,\y+1.75)--(\x+1.5158,\y+.875); 
			\draw [] (\x+1.5158,\y+.875) -- (\x+1.5158,\y-.875); 
			\draw [] (\x+1.5158,\y-.875)--(\x,\y-1.75); 
			\draw [] (\x,\y-1.75) -- (\x-1.5158,\y-.875); 
			\draw [] (\x-1.5158,\y-.875)--(\x-1.5158,\y+.875); 
			\draw [] (\x-1.5158,\y+.875) -- (\x,\y+1.75); 
		
			\draw [] (\x,\y+1.75) -- (\x-1.5158,\y-.875);
			\draw [] (\x,\y+1.75) -- (\x,\y-1.75);
			\draw [] (\x,\y+1.75) -- (\x+1.5158,\y-.875);
		
			}
		}
		\draw[red,<->] (3,-8) -- (4,-8);
	\end{tikzpicture}
	\quad
	\begin{tikzpicture}[scale = \localpicscale]
		\tikzstyle{every node} = [font = \small]
		\foreach \x in {0}
		{
			\foreach \y in {-8}
			{
			\fill (\x,\y+1.75) circle (.05);
			\fill(\x+1.5158,\y+.875) circle (.05);
			\fill (\x+1.5158,\y-.875) circle (.05);
			\fill(\x,\y-1.75) circle (.05);
			\fill(\x-1.5158,\y-.875) circle (.05);
			\fill (\x-1.5158,\y+.875) circle (.05);
		
			\draw [] (\x,\y+1.75)--(\x+1.5158,\y+.875); 
			\draw [] (\x+1.5158,\y+.875) -- (\x+1.5158,\y-.875); 
			\draw [] (\x+1.5158,\y-.875)--(\x,\y-1.75); 
			\draw [] (\x,\y-1.75) -- (\x-1.5158,\y-.875); 
			\draw [] (\x-1.5158,\y-.875)--(\x-1.5158,\y+.875); 
			\draw [] (\x-1.5158,\y+.875) -- (\x,\y+1.75); 
		
			\draw [] (\x,\y+1.75) -- (\x-1.5158,\y-.875);
			\draw [] (\x,\y+1.75) -- (\x,\y-1.75);
			\draw [red,ultra thick] (\x,\y-1.75) -- (\x+1.5158,\y+.875);
		
			}
		}
		\draw[blue, <->] (3,-8) -- (4,-8);
	\end{tikzpicture}
		\quad
	\begin{tikzpicture}[scale = \localpicscale]
		\tikzstyle{every node} = [font = \small]
		\foreach \x in {0}
		{
			\foreach \y in {-8}
			{
			\fill (\x,\y+1.75) circle (.05);
			\fill(\x+1.5158,\y+.875) circle (.05);
			\fill (\x+1.5158,\y-.875) circle (.05);
			\fill(\x,\y-1.75) circle (.05);
			\fill(\x-1.5158,\y-.875) circle (.05);
			\fill (\x-1.5158,\y+.875) circle (.05);
		
			\draw [] (\x,\y+1.75)--(\x+1.5158,\y+.875); 
			\draw [] (\x+1.5158,\y+.875) -- (\x+1.5158,\y-.875); 
			\draw [] (\x+1.5158,\y-.875)--(\x,\y-1.75); 
			\draw [] (\x,\y-1.75) -- (\x-1.5158,\y-.875); 
			\draw [] (\x-1.5158,\y-.875)--(\x-1.5158,\y+.875); 
			\draw [] (\x-1.5158,\y+.875) -- (\x,\y+1.75); 
		
			\draw [] (\x,\y+1.75) -- (\x-1.5158,\y-.875);
			\draw [blue, ultra thick] (\x-1.5158,\y-.875) -- (\x+1.5158,\y+.875);
			\draw [] (\x,\y-1.75) -- (\x+1.5158,\y+.875);
		
			}
		}
	\end{tikzpicture}
	\quad 
	\begin{tikzpicture}[scale = \localpicscale]
		\tikzstyle{every node} = [font = \small]
		
		\draw[violet, <->] (-3,-8) -- (-4,-8);
		
		\foreach \x in {0}
		{
			\foreach \y in {-8}
			{
			\fill (\x,\y+1.75) circle (.05);
			\fill(\x+1.5158,\y+.875) circle (.05);
			\fill (\x+1.5158,\y-.875) circle (.05);
			\fill(\x,\y-1.75) circle (.05);
			\fill(\x-1.5158,\y-.875) circle (.05);
			\fill (\x-1.5158,\y+.875) circle (.05);
		
			\draw [] (\x,\y+1.75)--(\x+1.5158,\y+.875); 
			\draw [] (\x+1.5158,\y+.875) -- (\x+1.5158,\y-.875); 
			\draw [] (\x+1.5158,\y-.875)--(\x,\y-1.75); 
			\draw [] (\x,\y-1.75) -- (\x-1.5158,\y-.875); 
			\draw [] (\x-1.5158,\y-.875)--(\x-1.5158,\y+.875); 
			\draw [] (\x-1.5158,\y+.875) -- (\x,\y+1.75); 
		
			\draw [violet, ultra thick] (\x-1.5158,\y+.875) -- (\x+1.5158,\y+.875);
			\draw [] (\x-1.5158,\y-.875) -- (\x+1.5158,\y+.875);
			\draw [] (\x,\y-1.75) -- (\x+1.5158,\y+.875);
		
			}
		}
		\draw[mygreen, <->] (3,-8) -- (4,-8);
	\end{tikzpicture}
		\quad
	\begin{tikzpicture}[scale = \localpicscale]
		\tikzstyle{every node} = [font = \small]
		\foreach \x in {0}
		{
			\foreach \y in {-8}
			{
			\fill (\x,\y+1.75) circle (.05);
			\fill(\x+1.5158,\y+.875) circle (.05);
			\fill (\x+1.5158,\y-.875) circle (.05);
			\fill(\x,\y-1.75) circle (.05);
			\fill(\x-1.5158,\y-.875) circle (.05);
			\fill (\x-1.5158,\y+.875) circle (.05);
		
			\draw [] (\x,\y+1.75)--(\x+1.5158,\y+.875); 
			\draw [] (\x+1.5158,\y+.875) -- (\x+1.5158,\y-.875); 
			\draw [] (\x+1.5158,\y-.875)--(\x,\y-1.75); 
			\draw [] (\x,\y-1.75) -- (\x-1.5158,\y-.875); 
			\draw [] (\x-1.5158,\y-.875)--(\x-1.5158,\y+.875); 
			\draw [] (\x-1.5158,\y+.875) -- (\x,\y+1.75); 
		
			\draw [] (\x-1.5158,\y+.875) -- (\x+1.5158,\y+.875);
			\draw [] (\x-1.5158,\y-.875) -- (\x+1.5158,\y+.875);
			\draw [mygreen, ultra thick] (\x-1.5158,\y-.875) -- (\x+1.5158,\y-.875);
		
			}
		}
	\end{tikzpicture}
\caption{Sequence of flips.}
\label{fig:seq_flips}
\end{figure}

When working with a cluster algebra associated to a surface with punctures, it is not sufficient to work with ordinary arcs and ideal triangulations.
The authors of \cite{FST08} introduced \emph{tagged arcs} and \emph{tagged triangulations}, and showed that they are in bijection with cluster variables and clusters.

\begin{defn}
[tagged arcs] \label{defn:tagged}
A \emph{tagged arc} is obtained by taking an ordinary arc (that is not an $\ell$-loop) $\gamma$ and marking (``tagging'') each end $\gamma$ with one of two options, \emph{plain} or \emph{notched}, such that:
\begin{enumerate}[1)]
\item an endpoint lying on the boundary must be tagged plain, and 
\item both ends of a loop must be tagged the same way.
\end{enumerate}
A notched tagging is usually indicated by a bow tie, and a plain tagging is usually denoted by no marking. Note that a tagged arc never cuts out a once-punctured monogon, i.e., an $\ell$-loop is \emph{not} a tagged arc.
For a list of tagged arcs, see \cite[Remark 7.3]{FST08}.
\end{defn}

Compatibility of two tagged arcs is defined in \cite[Def. 7.4]{FST08}. A maximal (by inclusion) collection of distinct, pairwise compatible tagged arcs is called a \emph{tagged triangulation}. Fig.~\ref{fig:tagged_tri} (center) gives an example of a tagged triangulation.
The \emph{flip} of a tagged arc is defined in \cite[Section 9.3]{FST08}.

\begin{figure}[h!]
\begin{tikzpicture}[scale=1.2]
\node (a) at (0,0) [fill,violet,circle,inner sep=1pt] {};
\draw (0,0) circle (1.25cm);
\foreach \x in {-144,-72,0,72,144} {
   \begin{scope}[rotate=\x]
    \node (\x) at (0,-1.25) [fill,circle,inner sep=1pt] {};
   \end{scope}
}

\draw[thin] (a) to (-144);

\draw[thin,
out=-140,in=-80] (0.2,-0.2) to (-144);
\draw[thin,
out=-30,in=40] (-144) to  (0.2,-0.2);

\draw[thin,out=110,in=0] (72) to  (-144);
\draw[thin,out=-110,in=130] (-144) to (0);
\draw[thin,out=-170,in=-45] (72) to (-0.3,-0.3) ;
\draw[thin,out=135,in=-90] (-0.3,-0.3) to(-144);

\fill(-.3,0.65) node[] {{$\tau_1$}};
\fill(-.7,-.6) node[] {{$\tau_2$}};
\fill(.2,-.7) node[] {{$\tau_3$}};
\fill(.5,0) node[] {{$\tau_4$}};
\fill(.85,.5) node[] {{$\tau_0$}};

\end{tikzpicture}
\quad
\begin{tikzpicture}[scale=1.2]
\node (a) at (0,0) [fill,violet,circle,inner sep=1pt] {};
\draw (0,0) circle (1.25cm);
\foreach \x in {-144,-72,0,72,144} {
   \begin{scope}[rotate=\x]
    \node (\x) at (0,-1.25) [fill,circle,inner sep=1pt] {};
   \end{scope}
}



\draw[thin,out=110,in=0] (72) to  (-144);
\draw[thin,out=-110,in=130] (-144) to (0);
\draw[thin,out=-170,in=-45] (72) to (-0.3,-0.3) ;
\draw[thin,out=135,in=-90] (-0.3,-0.3) to(-144);

\fill(-.5,0.65) node[] {{$\tau_1$}};
\fill(0,0.55) node[] {{$\tau_4$}};

\fill(-.7,-.6) node[] {{$\tau_2$}};
\fill(.2,-.7) node[] {{$\tau_3$}};

\fill(.85,.5) node[] {{$\tau_0$}};

\draw[
out=280,in=160] (-144) to  (a);
\draw[
out=100,in=-20] (a) to  (-144);
\draw[
out=110,in=0] (72) to  (-144);


\begin{scope}[scale=0.5]
\foreach \x in {-1.1} {
    \foreach \y in {-1.9}{
    \path[fill, black] (\x+1,\y+2.5) -- (\x+1.1,\y+2.8) -- (\x+1.3,\y+2.5) -- cycle; 
    \path[fill, black] (\x+1,\y+2.5) -- (\x+0.9,\y+2.2) -- (\x+0.7,\y+2.5) -- cycle; 
    }
}
\end{scope}
\end{tikzpicture}
\quad
\begin{tikzpicture}
\matrix [matrix of math nodes,left delimiter={(},right delimiter={)}]
{
0 & 1 & 0 & -1 & 1 \\
-1 & 0 & 0 & 1 & 0 \\
0 & 0 & 0 & -1 & 0 \\
1 & -1 & 1 & 0 & -1 \\
-1 & 0 & 0 & 1 & 0\\
};
\end{tikzpicture}
\caption{An ideal triangulation \& a corresponding tagged triangulation of the once-punctured $5$-gon, and the corresponding signed adjacency matrix.}
\label{fig:tagged_tri}
\end{figure}
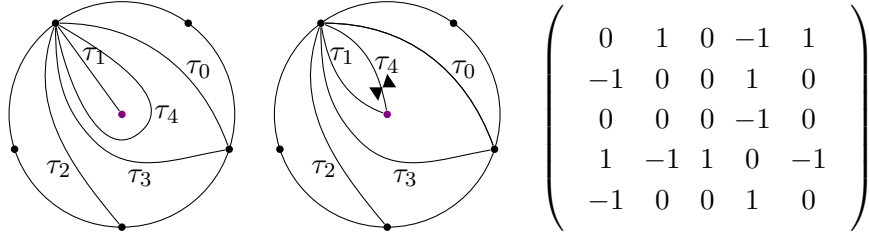

\begin{defn}
[representing ordinary arcs as tagged arcs]
\label{def:representing_ordinary_arcs_as_tagged_arcs}
Any ordinary arc $\gamma$ can be represented by a tagged arc $\iota (\gamma)$ as follows.
Suppose $\gamma$ is an $\ell$-loop (based at marked point $v$) which encloses a puncture $\puncture$. Let $r$ be the unique arc which is compatible with $\gamma$ and which connects~$v$ and~$\puncture$.
Then~$\iota(\gamma)$ is obtained by tagging $r$ plain at $v$ and notched at~$p$.
Otherwise, $\iota(\gamma)$ is simply~$\gamma$ tagged plain at both endpoints.
For example, see Fig.~\ref{fig:tagged_tri}(left \& center).
\end{defn}

\subsection{From surfaces to cluster algebras}
\label{subsec:clusters}

We can associate an exchange matrix {\cite[Def. 4.1 and 9.6]{FST08}}, and hence a cluster algebra, to $(S,M)$.
Note that our convention agrees with \cite{Sch10,MS10} but is opposite of the more recent papers~\cite{MSW11,MW13,MSW13}.
\begin{defn}[{\signedadjacencymatrixofanidealtriangulation}]
\label{defn:signed_adjacency_matrix_ideal}
Let $T$ be an ideal triangulation, and $\tau_1,\tau_2,\ldots,\tau_n$ arcs of $T$. For any non-self folded triangle $\Delta$ in $T$, we define a matrix $B^\Delta=\left( b^\Delta_{ij} \right)_{1\le i\le n, 1\le j\le n}$ as follows:
\begin{itemize}
\item $b_{ij}^\Delta=1$ and $b_{ji}^{\Delta}=-1$ in the following cases:
\begin{itemize}
\item[(a)] $\tau_i$ and $\tau_j$ are sides of $\Delta$ with  $\tau_j$ following $\tau_i$  in the {\ellOpOrientation} order;
\item[(b)] $\tau_j$ is a radial arc in a self-folded triangle enclosed by an $\ell$-loop $\tau_\ell$, and $\tau_i$ and $\tau_\ell$ are sides of $\Delta$ with  $\tau_\ell$ following $\tau_i$  in the {\ellOpOrientation} order;
\item[(c)] $\tau_i$ is a radial arc in a self-folded triangle enclosed by an $\ell$-loop $\tau_\ell$, and $\tau_\ell$ and $\tau_j$ are sides of $\Delta$ with  $\tau_j$ following $\tau_\ell$ in the {\ellOpOrientation} order;
\end{itemize}
\item $b_{ij}^\Delta=0$ otherwise.
\end{itemize}
 
Then define the \emph{signed adjacency matrix} $B_{T}=(b_{ij})_{1\le i\le n, 1\le j\le n}$ of $T$ 
 by
$b_{ij}=\sum_{\Delta} b_{ij}^\Delta$, where the sum is taken over all
triangles in $T$ that are not self-folded. 
\end{defn}

\begin{defn}[signed adjacency matrix of a tagged triangulation]
\label{defn:signed_adjacency_matrix_tagged}
Let $T$ be a tagged triangulation. 
From $T$, we construct a tagged triangulation $\widehat{T}$ as follows: for each puncture such that all endpoints are notched, we change their tags to plain. 
Let  $T^o$ be the ideal triangulation which is represented by $\widehat{T}$.
For each tagged arc in $T$, the corresponding ordinary arc in $T^o$ retains the same label.
The \emph{signed adjacency matrix} $B_T$ of $T$ is defined to be the signed adjacency matrix $B_{T^o}$ of $T^o$ (as in Definition \ref{defn:signed_adjacency_matrix_ideal}).
See Fig.~\ref{fig:tagged_tri}(right) for an example.
\end{defn}

\begin{thm}[\cite{FST08} Theorem 7.11, and \cite{FT12} Theorem 6.1]
\label{thm:FST}
\label{thm:fst_thm}
Let $(S,M)$ be a marked surface, and let $\mathcal{A} = \mathcal{A}(S,M)$ be the coefficient-free cluster algebra associated to the signed adjacency matrix of a tagged triangulation 
(as in Definition \ref{defn:signed_adjacency_matrix_tagged}).
Then the (unlabeled) seeds $\Sigma_{T}$ of $\mathcal{A}$ are in bijection with tagged triangulations $T$ of $(S,M)$, and the cluster variables are in bijection with the tagged arcs of $(S,M)$ (so we can denote each cluster variable by $x_{\gamma}$ or $x(\gamma)$, where $\gamma$ is a tagged arc). Moreover, each seed in $\mathcal{A}$ is uniquely determined by its cluster. Furthermore, if a tagged triangulation $T'$ is obtained from another tagged triangulation $T$ by flipping a tagged arc $\gamma\in T$ and obtaining $\gamma'$, then $\Sigma_{T'}$ is obtained from $\Sigma_{T}$ by the seed mutation replacing $x_{\gamma}$ by $x_{\gamma'}$.
\end{thm}

If $\ell$ is an unnotched $\ell$-loop 
which encloses a radius $r$ and a puncture $\puncture$, then we set $x_\ell = x_r \, x_{r\notchP}$, where $r\notchP$ denote the arc obtained from $r$ by changing its notching at $\punctureP$. 
If $\tau$ is a boundary edge, we set $x_\tau:=1$.

\subsection{Generalized arcs and closed loops}
\label{subsec:background_generalized_arcs_and_loops}

In \cite{MSW11}, the second author, Schiffler, and Williams gave a combinatorial formula for the Laurent expansion of any cluster variable in a cluster algebra associated to a marked surface. Their expansion formula, which is a weighted sum over perfect matchings of a \emph{planar snake graph}, yields a cluster variable for any arc in the surface.
In \cite{MSW13,MW13}, the same authors generalized this construction and associated cluster algebra elements to \emph{generalized arcs}, as well as to \emph{closed loops} (with or without self-crossings). Instead of perfect matchings of a planar graph, the Laurent polynomial associated to a closed curve is a weighted sum over good matchings in a \emph{band graph} on a M\"obius strip or annulus. In coefficient-free settings, these constructions for generalized arcs and loops work even in the existence of punctures.

\begin{defn}
[generalized arcs]\label{defn:generalized_arc}
A \emph{generalized (ordinary) arc} in $(S,M)$ is a curve $\gamma$ in $S$, considered up to isotopy, such that 
\begin{enumerate}
\item the endpoints of $\ga$ are in $M$, 
\item the interior of $\ga$ is disjoint from $M$ and the boundary of $S$, and 
\item $\ga$ does not cut out an unpunctured  bigon or monogon. In other words, $\ga$ is not contractible to a point, and $\ga$ is not isotopic to a boundary edge.
\end{enumerate}
Generalized arcs are allowed to intersect themselves a finite number of times (possibly $0$). We consider these arcs up to isotopy of immersed arcs, that is, allowing Reidemeister moves of types II and III but not of type I.
In particular, an isotopy cannot remove a contractible kink from a generalized arc. If an arc intersects itself, we say that the arc has a \emph{self-crossing}.
\end{defn}

\begin{defn}[generalized peripheral arc]\label{defn:generalized_peripheral_arc}
Suppose $(S,M)$ contains a boundary component $Bd$.
We say that a generalized (ordinary) arc $\ga$ is a \emph{generalized peripheral arc on $Bd$} 
if $\ga$ starts at a marked point on $Bd$, wraps finitely (possibly 0) many times around $Bd$ and then ends at a marked point on $Bd$ (possibly at the same starting point).
Furthermore, as in Definition~\ref{defn:types_of_arcs},
a generalized peripheral arc on $Bd$ is isotopic to a concatenation of two or more boundary edges of $Bd$.
Our convention is to choose the orientation of $\ga$ so that $Bd$ is to the right of $\ga$ when looking from above.
\end{defn}

\begin{rem} \label{rem:gamIJ}
In the case of the once-punctured {\disk} (respectively, annulus), we can draw a generalized arc in the universal cover as in \cite[Sec.~3.3]{BPT16}.  In the universal cover, we identify generalized peripheral arcs $\gamma(i,j)$ on the (lower) boundary with their two endpoints, labeled with $i,j \in \mathbb{Z}$.

\end{rem}

\begin{figure}[h]
	\begin{center}
	
        \begin{tikzpicture}
        \draw (0,0) circle (1.25cm);
        \foreach \x in {-144,-72,0,72,144} {
           \begin{scope}[rotate=\x]
            \node (\x) at (0,-1.25) [fill,circle,inner sep=1pt] {};
           \end{scope}
        }
        
        \filldraw[fill=gray!20](0,0) circle (.25cm);
        \draw (144) node [above right] {$4$};
        \draw (72) node [right] {$3$};
        \draw (0) node [below] {$2$};
        \draw (-72) node [left] {$1$};
        \draw (-144) node [above left] {$5$};
        
        \draw [ultra thick] (-.5,.0) to[out=90,in=125] (.4,.4) to[out=-180+135,in=30] (0);
        \draw [ultra thick] (-.5,.0) to[out=270,in=220] (.4,-.4) to[out=30,in=-55] (144);
        
        \end{tikzpicture}
	
	\begin{tikzpicture}[scale = .53]
		\tikzstyle{every node} = [font = \small]
		\foreach \x in {0}
		{
			\foreach \y in {-8}
			{
				\draw[black,thick,<-] (\x-15,\y+2) -- (\x+12,\y+2);
				\draw[black,thick,->] (\x-15,\y-2) -- (\x+12,\y-2);
				\draw[dashed] (\x-14,\y+2) -- (\x-14,\y-2);
				\draw[dashed] (\x-9,\y+2) -- (\x-9,\y-2);
				\draw[dashed] (\x-4,\y+2) -- (\x-4,\y-2);
				\draw[dashed] (\x+1,\y+2) -- (\x+1,\y-2);
				\draw[dashed] (\x+6,\y+2) -- (\x+6,\y-2);
				\draw[dashed] (\x+11,\y+2) -- (\x+11,\y-2);
				
				\draw (\x-13,\y-2) .. controls (\x-12,\y-.5) and (\x-7,\y-.5) .. (\x-6,\y-2);
				\draw (\x-8,\y-2) .. controls (\x-7,\y-.5) and (\x-2,\y-.5) .. (\x-1,\y-2);
				\draw (\x-3,\y-2) .. controls (\x-2,\y-.5) and (\x+3,\y-.5) .. (\x+4,\y-2);
				\draw (\x+2,\y-2) .. controls (\x+3,\y-.5) and (\x+8,\y-.5) .. (\x+9,\y-2);
				\draw (\x+7,\y-2) .. controls (\x+8,\y-.85) and (\x+10, \y-.85) ..  (\x+11,\y-.95);

				\foreach \t in {1,2,3,4,5,6,7,8,9,10,11,12,13,14,15,16,17,18,19,20,21,22,23,24,25,26}
				{
					\fill (-15+\t,\y-2) circle (.1);
					\fill (-15+\t,\y-2) node [below] {$\tiny{\t}$};
				}
			}
		}
		\end{tikzpicture}
	\end{center}
\caption{Top: The generalized peripheral arc $\gamma(2,9)$. Bottom: Universal cover of the annulus $C_{n,m}$, for $n=5$.}
\label{fig:univ_cover}
\end{figure}
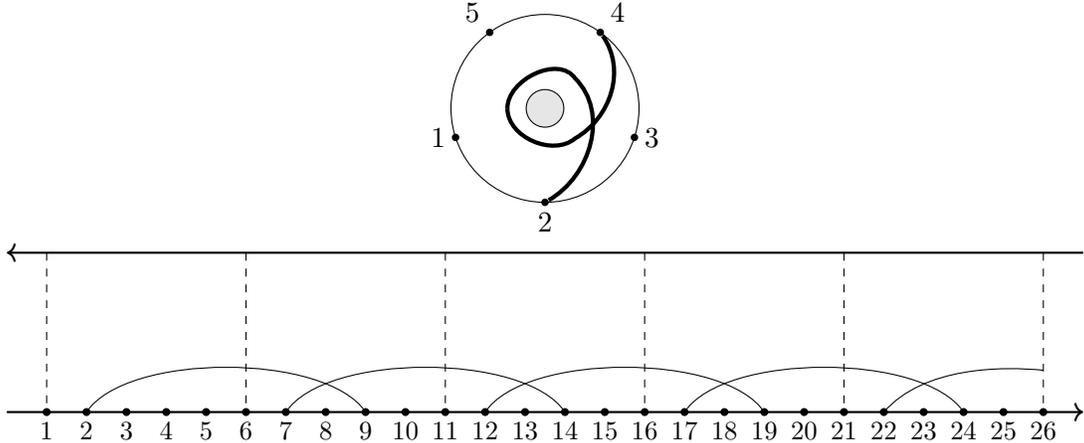

\begin{ex}
In Fig.~\ref{fig:univ_cover}, we draw copies of the arc $\gamma(2,9)$ (top figure) along the lower (outer) boundary in the universal cover of an annulus with $5$ points on the outer boundary (bottom figure).
Note that the arc $\gamma(2,9)$ has a self-crossing in the annulus. 
This can be seen in the universal cover by the arc crossing into another \emph{frame} (denoted by the dashed lines). 
\end{ex}

\begin{defn}[closed loops]
A \emph{closed loop} in $(S,M)$ is a closed curve $\gamma$ in $S$ which is disjoint from the boundary of $S$. Again, we allow closed loops to have a finite number of self-crossings, and we consider closed loops up to isotopy. 
\end{defn}

\begin{defn}[bracelets]
\label{defn:bracelets}
A closed loop obtained by following a (non-contractible, non-self-crossing, kink-free)
loop $k$ times, and thus creating $k-1$ self-crossings, is called a \emph{$k$-bracelet} and is denoted by $Brac_k$. See Fig.~\ref{fig:bracelets}.
\end{defn}

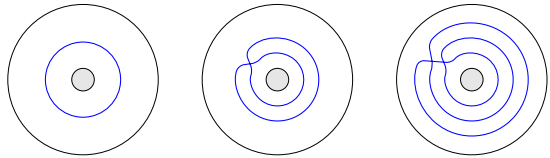
\begin{figure}[h!]
\begin{center}
\vskip -2mm
\scalebox{0.5}{

\def\FigBraceletsScale{1}
\def\MacroFigBraceletsAnnulus{
\draw (0,0) circle (2cm);
\draw[fill=gray!20] (0,0) circle (.3cm);
}

\subfigure{
\begin{tikzpicture}[scale=\FigBraceletsScale]
\MacroFigBraceletsAnnulus
\draw[ thick, blue] (0,0) circle (1cm);
\end{tikzpicture}
}
\quad \quad
\subfigure{

\begin{tikzpicture}[scale=\FigBraceletsScale]
\MacroFigBraceletsAnnulus

\draw[ thick,blue, rounded corners] ([shift=(140:0.7cm)]0,0) arc (140:-200:0.7cm)
-- ([shift=(140:1.1cm)]0,0) 
arc (140:-200:1.1cm) -- cycle;
\end{tikzpicture}

}
\quad \quad
\subfigure{
\begin{tikzpicture}[scale=\FigBraceletsScale]
\MacroFigBraceletsAnnulus

\draw[ thick,blue, rounded corners] ([shift=(140:0.7cm)]0,0) arc (140:-200:0.7cm)
-- ([shift=(140:1.1cm)]0,0)
arc (140:-200:1.1cm)
--  ([shift=(140:1.5cm)]0,0) 
arc (140:-200:1.5cm)
-- cycle; 
\end{tikzpicture}
}

}
\end{center}
\caption{Bracelets $Brac_1$, $Brac_2$, and $Brac_3$.}\label{fig:bracelets}
\end{figure}

\subsection{Laurent polynomials associated to generalized arcs and closed loops}
\label{subsec:laurent_generalized_arcs}

Recall from Theorem  
\ref{thm:FST}
that for unlabeled seeds $\Sigma_T$ of a cluster algebra $\mathcal{A}$, the tagged arcs of $(S,M)$ are in bijection with the cluster variables, and we denote these variables by $x_\tau$ for $\tau$ a tagged arc. 

\begin{defn} [snake graph]
A \emph{snake graph} is a connected sequence of square tiles embedded in the plane.
    To build a snake graph, start with one tile, then glue a new tile so that
    the new tile is glued to the north or the east of the previous tile.
    Let $\ga$ be an ordinary generalized arc which is not an arc of an ideal triangulation $T$.
    A \emph{snake graph} associated to $\ga$ and $T$ is a weighted graph which is defined in \cite{MSW11,MW13}  (see also \cite{CS13}). For example, see Fig.~\ref{fig:generalized_arc_snakegraph}.
\end{defn}

\begin{defn} [band graph]
A \emph{band graph}, which lies on an annulus or a  M\"obius strip, is obtained from identifying two edges of a snake graph.
Let $\zeta$ be a loop. 
A \emph{band graph} associated to $\zeta$ and an ideal triangulation $T$ is as defined in \cite{MW13, MSW13} (see also \cite{CS15a}). 
For example, see Fig.~\ref{fig:band_laurent_example}.
\end{defn}

\begin{defn} [crossing monomial]
If $\gamma$ is a generalized arc or a closed loop, and $\tau_{i_1}, \tau_{i_2}, \ldots, \tau_{i_d}$ is the sequence of arcs in $T$ which $\gamma$ crosses, then the \emph{crossing monomial} of $\gamma$ with respect to $T$ is defined to be (see, for example, \cite[Def. 4.5]{MSW11}) 
\[\cross(T,\gamma) = \prod_{j=1}^d x_{\tau_{i_j}}.\]
\end{defn}

Recall that if $\tau$ is a boundary segment, we let $x_\tau:=1$.

\begin{defn}[weight of a perfect matching] 
A \emph{perfect matching} of a graph $G$ is a subset $P$ of the edges of $G$ such that each vertex of $G$ is incident to exactly one edge of $P$. If $G$ is a snake or band graph, and the edges of a perfect matching $P$ of $G$ are labeled $\tau_{j_1}, \ldots, \tau_{j_r}$, then we define the \emph{weight} $x(P)$ of $P$ to be $x_{\tau_{j_1}} \cdots x_{\tau_{j_r}}$.  See \cite[Def. 3.7]{MW13}.
\end{defn}

\begin{defn}[Laurent polynomial from a generalized arc]
\label{defn:generalized_arc_laurent_polynomial}
Let $(S,M)$ be a surface, $T$ an ideal triangulation, and $\mathcal{A}$ the cluster algebra 
associated to $B_T$. Let $\gamma$ be a generalized arc and let $G_{T,\gamma}$ denote its snake graph. We define a Laurent polynomial  which lies in (the fraction field) of $\mathcal{A}$.
\begin{enumerate}
\item If $\gamma$ cuts out a contractible monogon, then $X_\gamma^T$ is equal to zero.
\item If $\gamma$ has a contractible kink, let $\overline{\gamma}$ denote the corresponding tagged arc with this kink removed, and define $X_\gamma^T := (-1)X_{\overline{\gamma}}^T$.
\item Otherwise, define $$X_\gamma^T := \frac{1}{\cross(T,\gamma)} \sum_P x(P),$$ 
where the sum is over all perfect matchings $P$ of $G_{T,\gamma}$.  
\end{enumerate}
See \cite[Def. 3.12]{MW13}.
\end{defn} 

\begin{thm}[{\cite[Thm 4.10]{MSW11}}]
When $\gamma$ is an arc (with no self-crossings), $X_\gamma^T$ is equal to the Laurent expansion of the cluster variable $x_\gamma \in \mathcal{A}$ with respect to the seed $\Sigma_T$.
\end{thm}

\begin{defn}[good matching of a band graph]
Let $\zeta$ be a closed loop. A perfect matching $P$ of the band graph $\tilde{G}_{T,\zeta}$ is called a \emph{good matching} if there exits at least one tile of $\tilde{G}_{T,\zeta}$ with two of its four edges as part of $P$. For a precise definition, see~\cite[Def. 3.18]{MW13}.
\end{defn}

We can now define a Laurent polynomial $X_\zeta$ for every closed loop $\zeta$.

\begin{defn}[Laurent polynomial from a closed loop]
\label{defn:bracelet_laurent_polynomial}
Let $(S,M)$ be a surface, $T$ an ideal triangulation, and $\mathcal{A}$ the cluster algebra associated to $B_T$.
Let $\zeta$ be a closed loop.
We define a Laurent polynomial $X_\zeta^T$ which lies in (the fraction field) of $\mathcal{A}$.
\begin{enumerate}
\item If $\zeta$ is a contractible loop, then let $X_\zeta^T := -2$.
\item \label{defn:bracelet_laurent_polynomial:item:brac1_equals_2}
 If $\zeta$ is a closed loop without self-crossings enclosing a single puncture $\puncture$,
 then $X_\zeta^T:=2$.
\item If $\zeta$ has a contractible kink, let $\overline{\zeta}$ denote the corresponding closed loop with this kink removed, and define $X_\zeta^T := (-1)X_{\overline{\zeta}}$.
\item Otherwise, let $$X_\zeta^T := \frac{1}{\cross(T,\gamma)} \sum_P x(P) 
,$$ 
where the sum is over all good matchings $P$ of the band graph $\tilde{G}_{T,\zeta}$.
\end{enumerate}
See~\cite[Def. 3.21]{MW13}.
\end{defn}

In our study of infinite friezes, we only consider marked surfaces $(S,M)$ which have nonempty boundary. In this situation, the Laurent polynomials given in Definitions \ref{defn:generalized_arc_laurent_polynomial} and \ref{defn:bracelet_laurent_polynomial} in fact lie in $\mathcal{A}$, due to~\cite[Proposition 4.5]{MSW13}, \cite[Theorem 4.1]{Mul13} and \cite[Theorem 5]{CLS15}.

\begin{figure}[h]
\begin{tikzpicture}[scale=1.4]
\node (a) at (0,0) [fill,violet,circle,inner sep=1pt] {};
\draw (0,0) circle (1.25cm);
\foreach \x in {-144,-72,0,72,144} {
   \begin{scope}[rotate=\x]
    \node (\x) at (0,-1.25) [fill,circle,inner sep=1pt] {};
   \end{scope}
}
\draw (-72) node [left] {\tiny$v_2$};
\draw (0) node [below] {\tiny$v_3$};
\draw (72) node [right] {\tiny$v_4$};
\draw (144) node [above right] {\tiny{$v_5$}};
\draw (-144) node [above left] {\tiny$v_1$};

\draw (15:1.25cm) node [right] {$b_{45}$};
\draw (90:1.25cm) node [above] {$b_{51}$};
\draw (160:1.25cm) node [left] {$b_{12}$};
\draw (230:1.25cm) node [left] {$b_{23}$};
\draw (-55:1.25cm) node [right] {$b_{34}$};

\draw[thin] (a) to (-144);

\draw[thin,
out=-140,in=-80] (0.2,-0.2) to (-144);
\draw[thin,
out=-30,in=40] (-144) to  (0.2,-0.2);

\draw[thin,out=110,in=0] (72) to  (-144);
\draw[thin,out=-110,in=130] (-144) to (0);
\draw[thin,out=-170,in=-45] (72) to (-0.3,-0.3) ;
\draw[thin,out=135,in=-90] (-0.3,-0.3) to(-144);

\draw[red,rounded corners=6pt,densely dotted,->,very thick] (144) to (-0.2,0.1) to (0,-0.2) to (0.2,0.2) to (-0.4,0.3) to (-0.1,-0.3) to (72);

\fill(-.3,0.65) node[] {{$1$}}; 
\fill(-.7,-.6) node[] {{$2$}}; 
\fill(.2,-.7) node[] {{$3$}}; 
\fill(.5,0) node[] {{$4$}}; 
\fill(.85,.5) node[] {{$0$}}; 

\fill(3,0) node[] {$\rightsquigarrow$};
\end{tikzpicture}

\quad

		\begin{tikzpicture}[scale = .4]
		\tikzstyle{every node} = [font = \small]
		\foreach \x in {0,10,20}
		{
			\foreach \y in {-8}
			{
				\draw[] (\x-7,\y-3) -- (\x+6,\y-3);
			
				\foreach \t in {-6,-4,-2,0,2,4}
				{
					\fill (\x+\t,\y-3) circle (.1);
				}

				\fill (\x-6,\y-3) node [below] {\tiny$v_1$};
				\fill (\x-4,\y-3) node [below] {\tiny$v_2$};
				\fill (\x-2,\y-3) node [below] {\tiny$v_3$};
				\fill (\x+0,\y-3) node [below] {\tiny$v_4$};
				\fill (\x+2,\y-3) node [below] {\tiny$v_5$};
				\fill (\x+4,\y-3) node [below] {\tiny$v_1$};
				
				\PruferShort{\x-6}{\y-3}{13}{}
				\fill (\x-5.7,\y-1) node [left] {\tiny{$\tau_1$}};
				\draw[] (\x-6,\y-3) .. controls (\x-5,\y-2) and (\x-3,\y-2) .. (\x-2,\y-3);
				\fill (\x-2.3,\y-2.9) node [above] {\tiny{$\tau_2$}};
                \draw[] (\x-6,\y-3) .. controls (\x-4.5,\y) and (\x+3,\y) .. (\x+4,\y-3);
                \fill (\x-0,\y-0.5) node [right] {\tiny{$\tau_4$}};
				\draw[] (\x-6,\y-3) .. controls (\x-4.5,\y-1.25) and (\x-1.5,\y-1.25) .. (\x-0,\y-3);
				\fill (\x-3,\y-1.8) node [above] {\tiny{$\tau_3$}};

				\draw[] (\x+0,\y-3) .. controls (\x+1,\y-2) and (\x+3,\y-2) .. (\x+4,\y-3);
				\fill (\x+2,\y-2.4) node [above] {\tiny{$\tau_0$}};	
			}
		}	
        \PruferShort{20+4}{-8-3}{1}{}
        \draw[->] (20-7,-8-3) -- (20+6,-8-3);
        \fill (20+4,-8-1) node [right,black] {\tiny{$\tau_1$}};
        
        \foreach \x in {0}
		{
			\foreach \y in {-8}
			{
				\draw[red,densely dotted,->,very thick] (\x+2,\y-3) .. controls (\x+4,\y+1) and (\x+19,\y+1) .. (\x+20,\y-3);
              \fill (\x+17,\y) node[red]{$\gamma$};

                }
         }
		\end{tikzpicture}
\caption{Top: An ideal triangulation $T$ and a generalized arc $\gamma$ of a once-punctured {\disk}. Bottom: $T$ drawn on a strip.}
\label{fig:example_generalized_arc}
\end{figure}

\begin{figure}[h]
\begin{center}
\begin{center}
\begin{tikzpicture}[scale=1.3]

\draw (0,1) rectangle (4,0);
\draw (0,0) -- (0,-1) -- (1,-1)-- (1,1);
\draw (0,1) -- (1,1);
\draw (2,1) -- (2,0);
\draw (3,1) -- (3,0);

\draw (0,-.5) node[left] {$b_{45}$};
\draw (0,.5) node[left] {$0$};
\draw (1,-.65) node[right] {$4$};
\draw (.9,.5) node[right] {$1$};

\draw (1.9,.5) node[right] {$4$};
\draw (2.9,.5) node[right] {$1$};
\draw (4,.5) node[right] {$3$};

\draw (0.5,1) node[above] {$1$};
\draw (0.5,.1) node[below] {$3$};
\draw (0.5,-1.2) node[below] {$b_{51}$};

\draw (1.5,1) node[above] {$1$};
\draw (2.5,1) node[above] {$4$};
\draw (3.5,1) node[above] {$0$};

\draw(1.5,0) node[below] {$4$};
\draw(2.5,0) node[below] {$1$};
\draw(3.5,0) node[below] {$1$};

\draw[red] (0.5,-.55) node[]{$0$};
\draw[red] (0.5,.5) node[]{$4$};
\draw[red] (1.55,.5) node[]{$1$};
\draw[red] (2.55,.5) node[]{$1$};
\draw[red] (3.55,.5) node[]{$4$};

\end{tikzpicture}
\end{center}
\caption{The snake graph corresponding to the generalized arc $\gamma$ of Fig.~\ref{fig:example_generalized_arc}.}
\label{fig:generalized_arc_snakegraph}
\end{center}
\end{figure}

\begin{figure}
\begin{center}
\includegraphics[width=15cm]{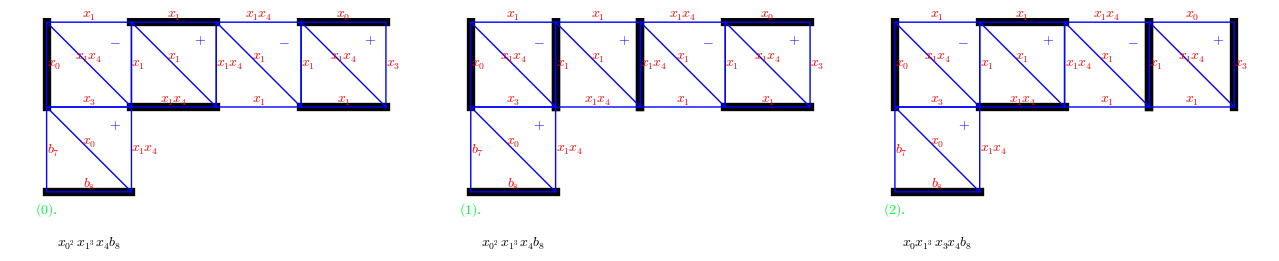}

\includegraphics[width=15cm]{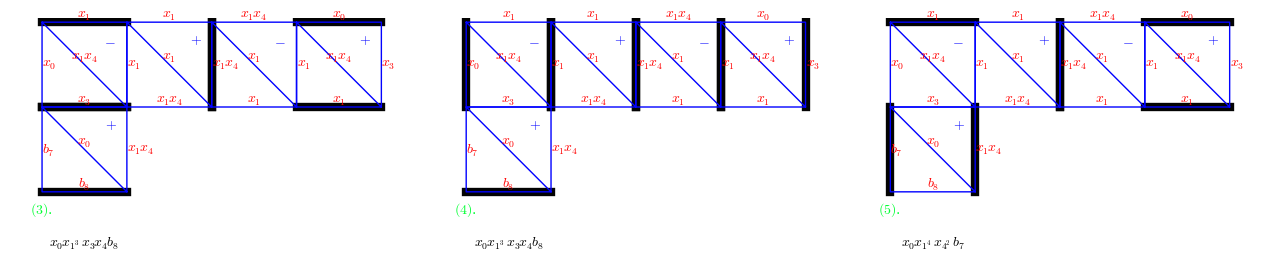}

\includegraphics[width=15cm]{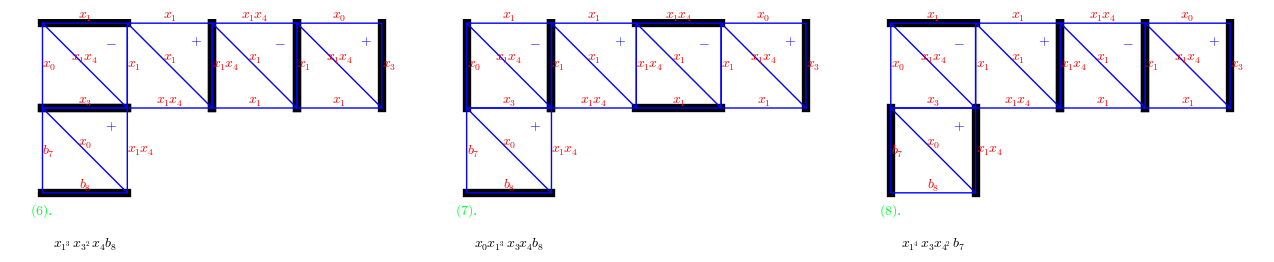}

\includegraphics[width=10cm]{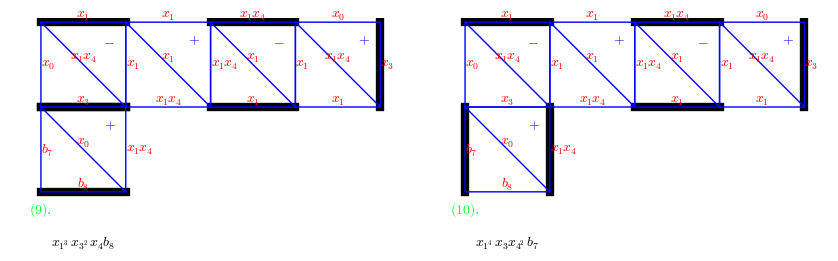}
\caption{The $11$ perfect matchings of the snake graph from Fig.~\ref{fig:generalized_arc_snakegraph}, created using the help of {\sc SageMath}~\cite{sage, combinat}.}
\label{fig:generalized_arc_snakegraph_11_pm}
\end{center}
\end{figure}

\begin{ex}[example of a Laurent expansion corresponding to a generalized arc]
Consider the ideal triangulation $T$ of a once-punctured {\disk} and a generalized arc in
Fig.~\ref{fig:example_generalized_arc}.
We obtain the snake graph ${G}_{T,\gamma}$ in Fig.~\ref{fig:generalized_arc_snakegraph}. 
Following Definition \ref{defn:generalized_arc_laurent_polynomial}, we compute
\[
X_{\gamma}^T = 
\frac{x_{0} x_{1} x_{4} + 
2 x_{1} x_{3} x_{4} + 
2 x_{0}^{2} + 
4 x_{0} x_{3} + 
2 x_{3}^{2}}
{x_{0} x_{1} x_{4}}
\]
by specializing $x(\tau)=1$ for each boundary edge $\tau$.
In particular, the snake graph ${G}_{T,\gamma}$ has $11$ perfect matchings (see Fig.~\ref{fig:generalized_arc_snakegraph_11_pm}).

\end{ex}

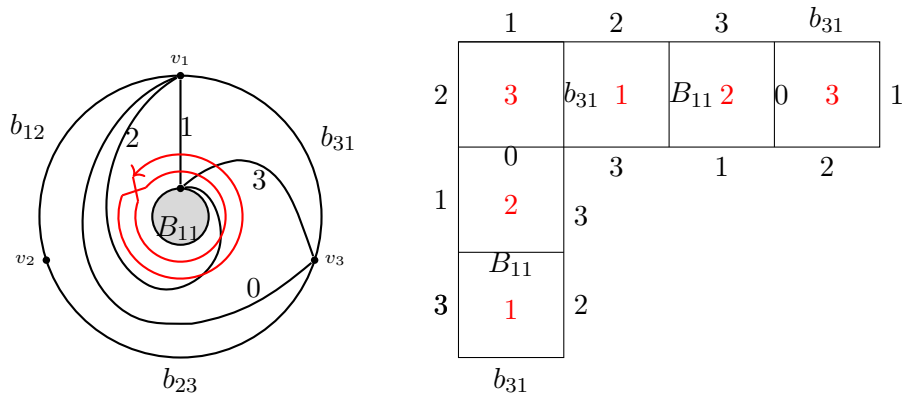
\begin{figure}[h]
\subfigure{
\begin{tikzpicture}[scale=1.5]
\node (a) at (0,0) [fill,circle,inner sep=1pt] {};
\draw[thick] (0,0) circle (1.25cm);
\foreach \x in {-72,180,72} {
   \begin{scope}[rotate=\x]
    \node (\x) at (0,-1.25) [fill,circle,inner sep=1pt] {};
   \end{scope}
}

\filldraw[thick, fill=gray!30](0,0) circle (.25cm);
\node (-180) at (0,.25) [fill,circle,inner sep=1pt] {};

\draw (72) node [right] {\tiny $v_3$};
\draw (180) node [above] {\tiny $v_1$};
\draw (-72) node [left] {\tiny $v_2$};

\draw[smooth, thick] (180) to (-180);
\draw [smooth, thick] (180) to [out=210,in=150] (-.3,-.6) to [out= 330 , in= 260] (.35, -.25) to [out=80,in=10] (-180); 
\draw [smooth, thick] (180) to [out=200,in=140] (-.55,-.8) to [out= 320 , in= 180] (0.1, -.95) to [out=10,in=220] (72); 
\draw [smooth, thick] (-180) to [out=40,in=190] (.5,.5) to [out= 0 , in= 110] (72); 

\draw (-0.02,-0.1) node[]{{$B_{11}$}};

\draw (.7,.5) node[below]{{$3$}};
\draw (.65,-.8) node[above]{{$0$}};
\draw (-0.1,.8) node[right]{{$1$}};
\draw (-.58,.7) node[right]{{$2$}};

\draw (-1.1,.8) node[left]{{$b_{12}$}};
\draw (0,-1.25) node[below]{{$b_{23}$}};
\draw (1.15,.7) node[right]{{$b_{31}$}};

\draw[thick,red] ([shift=(140:.4cm)]0,0) arc (140:-200:.4cm);
\draw[thick,red,<-] ([shift=(140:.55cm)]0,0) arc (140:-200:.55cm);
\draw[thick,red]  ([shift = (-200:.55cm)]0,0) -- ([shift=(140:.4cm)]0,0);
\draw[thick,red] ([shift=(140:.55cm)]0,0) -- ([shift = (-200:.4cm)]0,0);

\end{tikzpicture}
}
\quad
\subfigure{
\begin{tikzpicture}[scale=1.4]

\draw (0,1) rectangle (4,0);
\draw (0,0) -- (0,-2) -- (1,-2)-- (1,1);
\draw (0,-1) -- (1,-1);
\draw (0,1) -- (1,1);
\draw (2,1) -- (2,0);
\draw (3,1) -- (3,0);

\draw (0,-1.5) node[left] {$3$};
\draw (0,-.5) node[left] {$1$};
\draw (0,.5) node[left] {$2$};
\draw (1,-1.5) node[right] {$2$};
\draw (1,-.65) node[right] {$3$};
\draw (0,-1.5) node[left] {$3$};
\draw (.9,.5) node[right] {$b_{31}$};

\draw (1.9,.5) node[right] {$B_{11}$};
\draw (2.9,.5) node[right] {$0$};
\draw (4,.5) node[right] {$1$};

\draw (0.5,1) node[above] {$1$};
\draw (0.5,.1) node[below] {$0$};
\draw (0.5,-.9) node[below] {$B_{11}$};
\draw (0.5,-2) node[below] {$b_{31}$};

\draw (1.5,1) node[above] {$2$};
\draw (2.5,1) node[above] {$3$};
\draw (3.5,1) node[above] {$b_{31}$};

\draw(1.5,0) node[below] {$3$};
\draw(2.5,0) node[below] {$1$};
\draw(3.5,0) node[below] {$2$};

\draw[red] (0.5,-1.55) node[]{$1$};
\draw[red] (0.5,-.55) node[]{$2$};
\draw[red] (0.5,.5) node[]{$3$};
\draw[red] (1.55,.5) node[]{$1$};
\draw[red] (2.55,.5) node[]{$2$};
\draw[red] (3.55,.5) node[]{$3$};

\end{tikzpicture}
}
\caption{Left: An ideal triangulation $T$ and the bracelet $\Brac_2$. Right: Corresponding band graph $\widetilde{G}_{T,\gamma}$.}\label{fig:band_laurent_example}
\end{figure}

\begin{figure}
\begin{center}
\includegraphics[width=15cm]{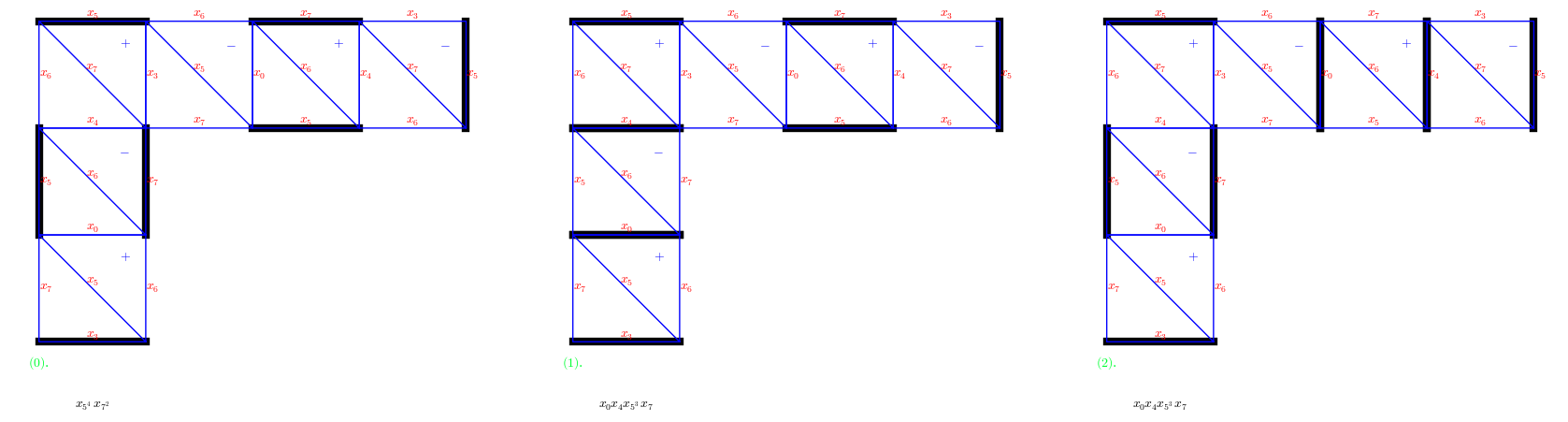}

\includegraphics[width=15cm]{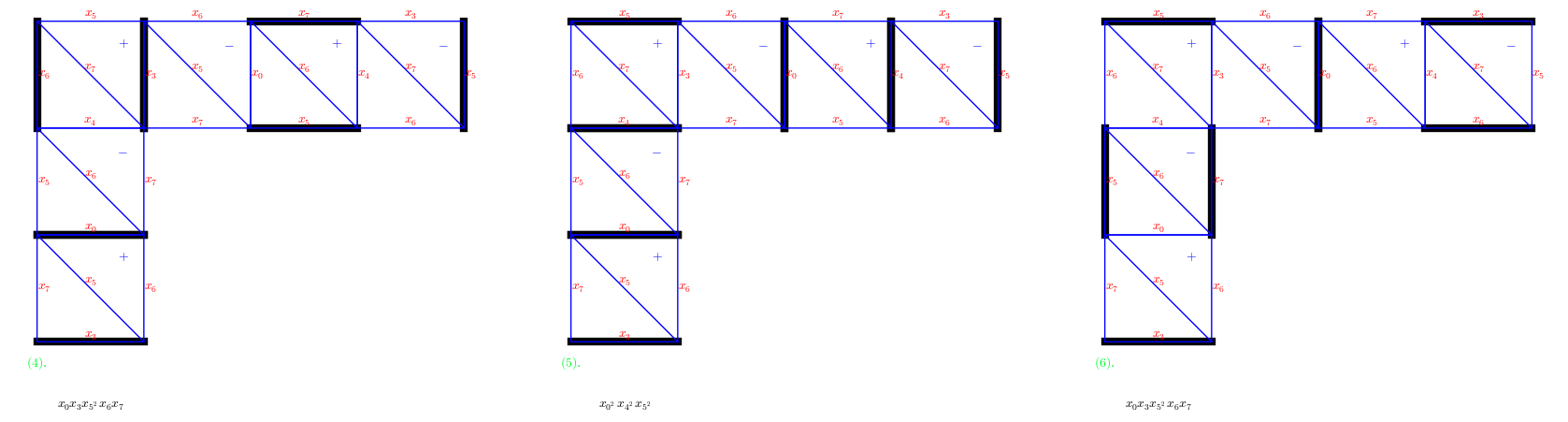}

\includegraphics[width=15cm]{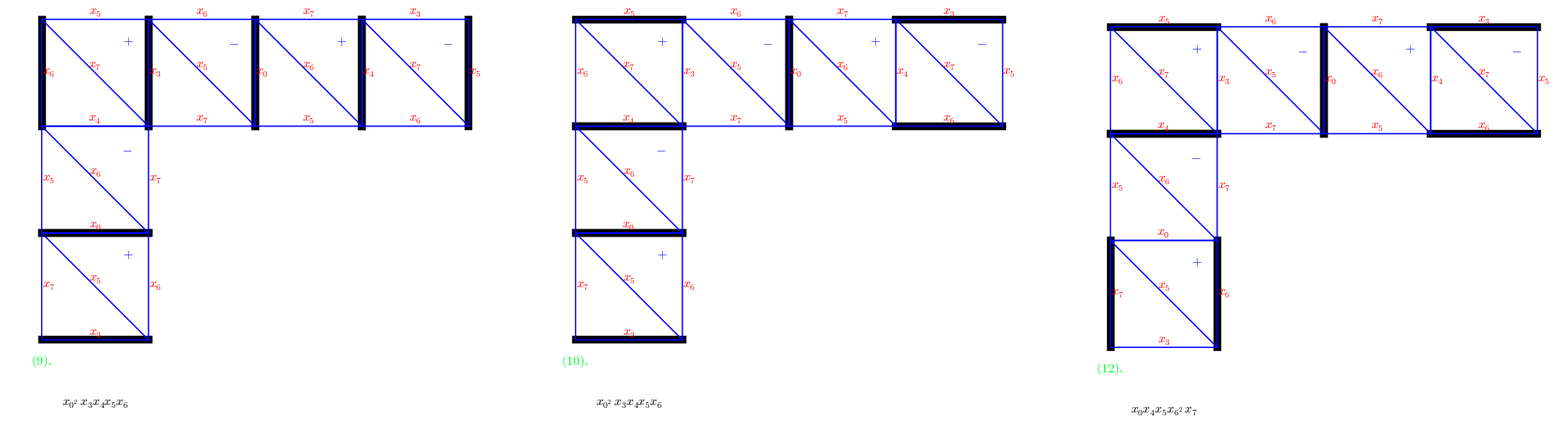}

\includegraphics[width=15cm]{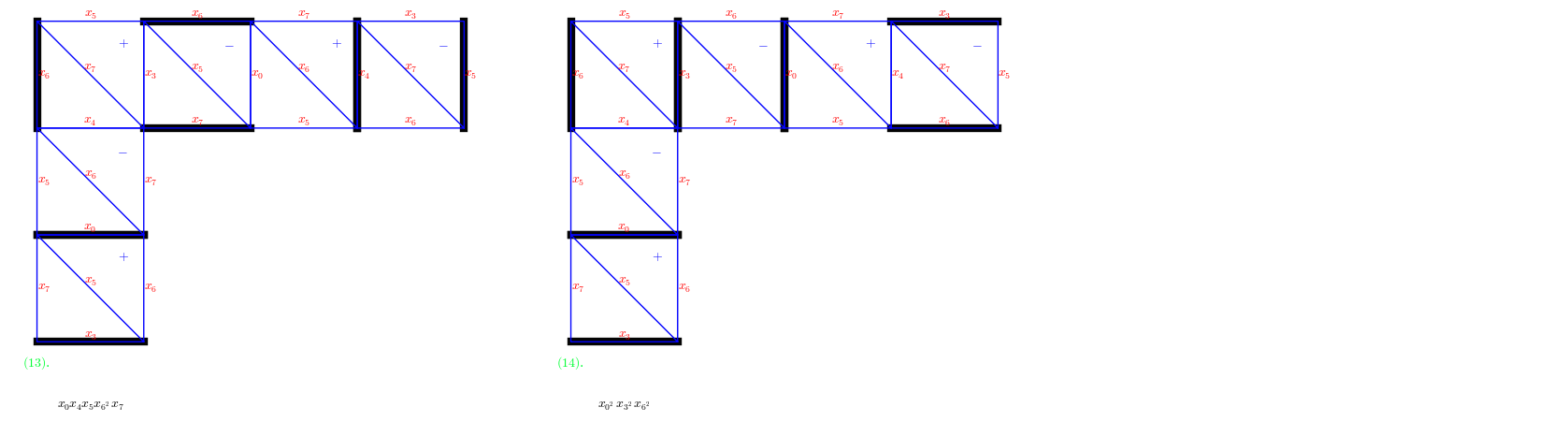}

\includegraphics[width=15cm]{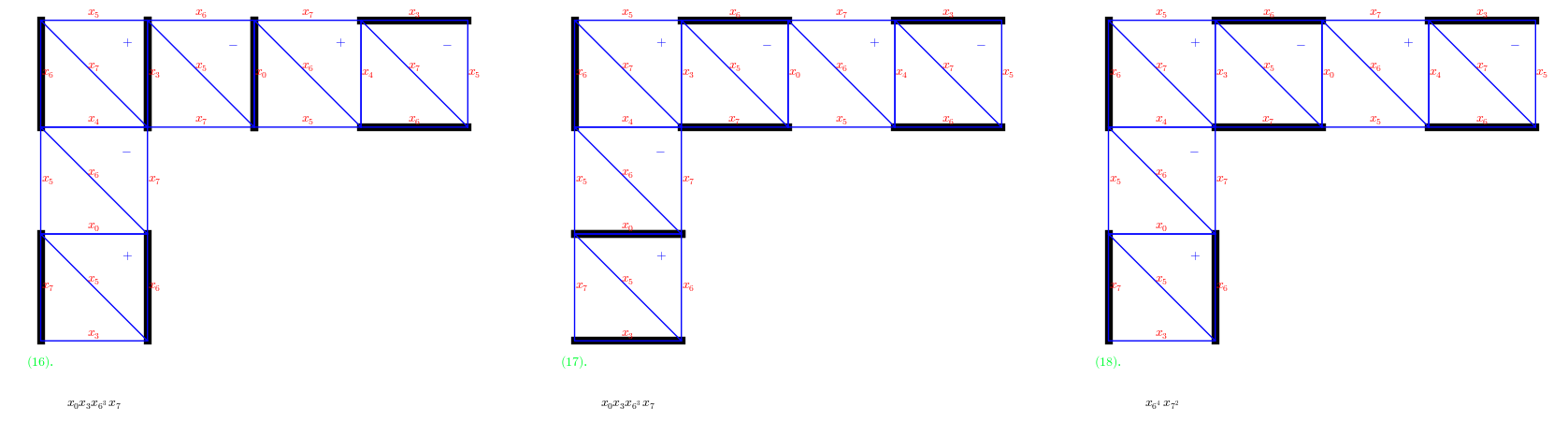}
\caption{The $14$ good matchings of the band graph from Fig.~\ref{fig:band_laurent_example}, created using the help of {\sc SageMath}~\cite{sage, combinat}.}
\label{fig:band_laurent_example_14pm}
\end{center}
\end{figure}

\begin{ex}[Example of the Laurent polynomial corresponding to $\Brac_2$ in an annulus]
Consider the ideal triangulation $T$ of an annulus and the $\Brac_2$ in
Fig.~\ref{fig:band_laurent_example} (left).
We obtain the graph $\widetilde{G}_{T,\Brac_2}$ in Fig.~\ref{fig:band_laurent_example} (right). 
Following Definition
\ref{defn:bracelet_laurent_polynomial}, we compute
\[
X_{\Brac_2}^T 
\hspace{-1mm}
=
\hspace{-1mm}
\frac{x_{1}^{4} x_{3}^{2} 
\hspace{-.5mm}+ 
x_{2}^{4} x_{3}^{2} 
\hspace{-.5mm}+ 
2 x_{0} x_{1}^{3} x_{3} 
\hspace{-.5mm}+ 
2 x_{0} x_{1} x_{2}^{2} x_{3} 
\hspace{-.5mm}+ 
x_{0}^{2} x_{1}^{2} 
\hspace{-.5mm}+ 
2 x_{1}^{2} x_{2} x_{3} 
\hspace{-.5mm}+ 
2 x_{2}^{3} x_{3} 
\hspace{-.5mm}+ 
2 x_{0} x_{1} x_{2} 
\hspace{-.5mm}+ 
x_{2}^{2}}
{x_{1}^{2} x_{2}^{2} x_{3}^{2}}
\]
by specializing $x(\tau)=1$ for each boundary edge $\tau$.
In particular, the band graph $\widetilde{G}_{T,\Brac_2}$ has $14$ good matchings (see Fig.~\ref{fig:band_laurent_example_14pm}).
\end{ex}

For the rest of the paper, we will use the notation $x_\gamma$ or $x(\gamma)$ to denote the cluster algebra element corresponding to $\ga$, where $\gamma$ is a generalized arc or loop.

\section{Infinite friezes}
\label{sec:background_infinite_frieze}
 
Tschabold in~\cite{Tsc15} showed that triangulations of once-punctured {\disks} give rise to certain periodic (positive integral) infinite friezes, and that infinite friezes arising in this way satisfy a certain arithmetic property (see Section~\ref{subsec:arithmetic_progressions}). In~\cite{BPT16}, Baur, Parsons, and Tschabold went further and gave a complete characterization of infinite frieze patterns of positive integers via triangulations of quotients of an infinite strip in the plane.  In this classification, periodic frieze patterns arise from triangulations of the annulus (which can be thought of as a quotient of the infinite strip by translation) or of the once-punctured disk.  
We refer the reader to Lemma 3.6 of \cite{BPT16} for a description on how to draw a triangulation of a once-punctured {\disk} as an asymptotic triangulation in the infinite strip.  See an example in Fig.~\ref{fig:disc_as_AT}.

\begin{defn} 
\label{defn:infinite_frieze}
An \emph{infinite frieze} $\F$ of positive integers is an array $\{m_{ij}\}_{i,j \in \mathbb{Z}, j \geq i}$ with infinitely many rows,  
drawn as in Fig.~\ref{fig:defn_integer_infinite_frieze2}, such that $m_{i,i} = 0$, $m_{i,i+1} = 1$, $m_{ij} \in \mathbb{Z}_{>0}$ for all $i \leq j$, 
where, for every diamond in $\F$ indexed by
$$\begin{array}{ccccccc}
 &({i+1,j})&\\
 ({i,j})&&({i+1,j+1})\\
 &({i,j+1})&
\end{array}$$
the relation $m_{i,j} \, m_{i+1,j+1} - m_{i+1,j} \, m_{i,j+1} = 1$ is satisfied.
We say that the row of all $0$s is the zeroth row.
\end{defn}

\begin{center}
\begin{figure}
\begin{tikzpicture}[font=\normalsize] 
\matrix(m) [matrix of math nodes,row sep={1.75em,between origins},column sep={1.75em,between origins},nodes in empty cells]{
&0&&0&&0&&0&&0&&&&&&\\ 
\node{\cdots};&&1&&1&&1&&1&&1&&\node{\cdots};&&\\[-0.25em]
&&&(-1,1)&&(0,2)&&(1,3)&&(2,4)&&(3,5)&&&\\
&&\node{\cdots};&&(-1,2)&&(0,3)&&(1,4)&&(2,5)&&(3,6)&&\node{\cdots};\\
&&&&&(-1,3)&&(0,4)&&(1,5)&&(2,6)&&(3,7)&\\
&&&&&&&&\node[rotate=-6.5,shift={(-0.034cm,-0.08cm)}] {\ddots};&&&&\node[rotate=-6.5,shift={(-0.034cm,-0.08cm)}]  {\ddots};&&\\
};
\end{tikzpicture}
\caption{Indexing for an infinite frieze $\F$.}
\label{fig:defn_integer_infinite_frieze2}
\end{figure}
\end{center}

We often omit the row of 0s when writing a frieze pattern, as they do not provide any additional information. The first non-trivial row of a frieze (that is, the second row) is called a \emph{quiddity row}. If $\F$ is periodic with period $n$, then we call the $n$-tuple $(a_1,\ldots, a_n)$ a \emph{quiddity sequence}. 

Just as finite friezes of type $A_{n}$ correspond to triangulations of a polygon $P_{n+3}$, triangulations of the once-punctured {\disk} $D_n$ give rise to infinite friezes of positive integers via matching numbers. Given a triangulation $T$ of a once-punctured {\disk} $D_n$, let the \emph{quiddity sequence of $T$} be $q_T = (a_1, \ldots, a_n)$ where $a_i$ is the number of ideal triangles incident to a vertex $i$, such that, if $i$ is adjacent to a self-folded triangle, both the self-folded triangle and the triangle with an $\ell$-loop as one of its three sides, are counted twice. See Fig.  \ref{fig:counting_triangles_punctured_disc} for an example of computing these matching numbers.

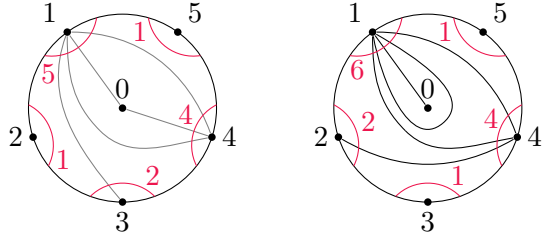
\begin{figure}[h!]
\subfigure{
\begin{tikzpicture}
\node (a) at (0,0) [fill,circle,inner sep=1pt] {};
\draw (0,0) circle (1.25cm);
\foreach \x in {-144,-72,0,72,144} {
   \begin{scope}[rotate=\x]
    \node (\x) at (0,-1.25) [fill,circle,inner sep=1pt] {};
   \end{scope}
}
\draw (a) node [above] {$0$};
\draw (-72) node [left] {$2$};
\draw (0) node [below] {$3$};
\draw (72) node [right] {$4$};
\draw (144) node [above right] {$5$};
\draw (-144) node [above left] {$1$};

\draw[thin,opacity=0.5] (a) to (-144);
\draw[thin,opacity=0.5] (a) to (72);
\draw[thin,opacity=0.5,out=110,in=0] (72) to  (-144);
\draw[thin,opacity=0.5,out=-110,in=130] (-144) to (0);
\draw[thin,opacity=0.5,out=-170,in=-45] (72) to (-0.3,-0.3) ;
\draw[thin,opacity=0.5,out=135,in=-90] (-0.3,-0.3) to(-144);

\foreach \x in {-144,-72,0,72,144} {
   \begin{scope}[rotate=\x]

\draw[thin,myred,out=45,in=135] (-0.428,-1.175) to (0.428,-1.175);
   \end{scope}
}

\draw (144) node [myred,shift={(-0.5cm,-.0cm)}] {$1$};
\draw (72) node [myred,shift={(-0.35cm,0.3cm)}] {$4$};
\draw (0) node [myred,shift={(0.4cm,0.35cm)}] {$2$};
\draw (-72) node [myred,shift={(0.4cm,-0.3cm)}] {$1$};
\draw (-144) node [myred,shift={(-0.25cm,-0.55cm)}] {$5$};
\end{tikzpicture}
}
\quad
\subfigure{
\begin{tikzpicture}
\node (a) at (0,0) [fill,circle,inner sep=1pt] {};
\draw (0,0) circle (1.25cm);
\foreach \x in {-144,-72,0,72,144} {
   \begin{scope}[rotate=\x]
    \node (\x) at (0,-1.25) [fill,circle,inner sep=1pt] {};
   \end{scope}
}
\draw (a) node [above] {$0$};
\draw (144) node [above right] {$5$};
\draw (72) node [right] {$4$};
\draw (0) node [below] {$3$};
\draw (-72) node [left] {$2$};
\draw (-144) node [above left] {$1$};
\draw
(a) to (-144);
\draw[
out=-140,in=-80] (0.2,-0.2) to (-144);
\draw[
out=-30,in=40] (-144) to  (0.2,-0.2);
\draw[
out=110,in=0] (72) to  (-144);
\draw[
out=-30,in=-150] (-72) to (72);
\draw[
out=-170,in=-45] (72) to (-0.3,-0.3) ;
\draw[
out=135,in=-90] (-0.3,-0.3) to(-144);

\foreach \x in {-144,-72,0,72,144} {
   \begin{scope}[rotate=\x]

\draw[thin,myred,out=45,in=135] (-0.428,-1.175) to (0.428,-1.175);
   \end{scope}
}

\draw (144) node [myred,shift={(-0.5cm,0cm)}] {$1$};
\draw (72) node [myred,shift={(-0.35cm,0.25cm)}] {$4$};
\draw (0) node [myred,shift={(0.4cm,0.35cm)}] {$1$};
\draw (-72) node [myred,shift={(0.4cm,0.2cm)}] {$2$};
\draw (-144) node [myred,shift={(-0.2cm,-0.5cm)}] {$6$};
\end{tikzpicture}
}
\caption{Counting triangles adjacent to boundary points in a triangulation of a once-punctured {\disk}.
Left: Triangulation without a self-folded triangle.
Right: Triangulation with a self-folded triangle.
Note that each of the two triangles with an $\ell$-loop as a side is counted twice for vertex $1$.}
\label{fig:counting_triangles_punctured_disc}
\end{figure}

Now that we have a way to read off the quiddity sequence from a triangulation, we can construct an infinite frieze pattern. The frieze pattern coming from the quiddity sequence of Fig.~\ref{fig:counting_triangles_punctured_disc} (left) is given in Fig.~\ref{fig:frieze_from_D5}. We write the $(k\,n)$-th rows (for $k \in \mathbb{N}$) of the frieze in bold characters.

\begin{figure}[h!]
\begin{tikzpicture}
 \matrix(m) [matrix of math nodes,row sep={1.5em,between origins},column sep={1.5em,between origins},nodes in empty cells]{
&&&&&&&&&&&&&&&&&&\\
0&&0&&0&&0&&0&&0&&0&&0&&\\
&1&&1&&1&&1&&1&&1&&1&&1&\\
&&5&&1&&2&&4&&1&&5&&1&&2&\\
&\node{\cdots};&&4&&1&&7&&3&&4&&4&&1&&7&\node{\cdots};&\\
&&&&3&&3&&5&&11&&3&&3&&3&&5&\\
&&&&&\mathbf{8}&&\mathbf{2}&&\mathbf{18}&&\mathbf{8}&&\mathbf{2}&&\mathbf{8}&&\mathbf{2}&&\mathbf{18}&\\
&&&&&&5&&7&&13&&5&&5&&5&&7&&13&\\
&&&&&&&&&\node[rotate=-6.5,shift={(-0.034cm,-0.08cm)}]  {\ddots};&&&&&&&&\node[rotate=-6.5,shift={(-0.034cm,-0.08cm)}]  {\ddots};&\\
};

\draw[opacity=0,rounded corners, fill=myblue,fill opacity=0.35] 
($(m-1-4.south west)+(-0.15cm,0cm)$) -- 
($(m-8-10.south west)+(0.25cm,0cm)$)--($(m-8-20.south west)+(0.25cm,0cm)$) -- ($(m-1-14.south west)+(-0.15cm,0cm)$) -- cycle;

\end{tikzpicture}
\caption{Infinite frieze with quiddity sequence $(2,4,1,5,1)$, with shaded fundamental region.}
\label{fig:frieze_from_D5}
\end{figure}
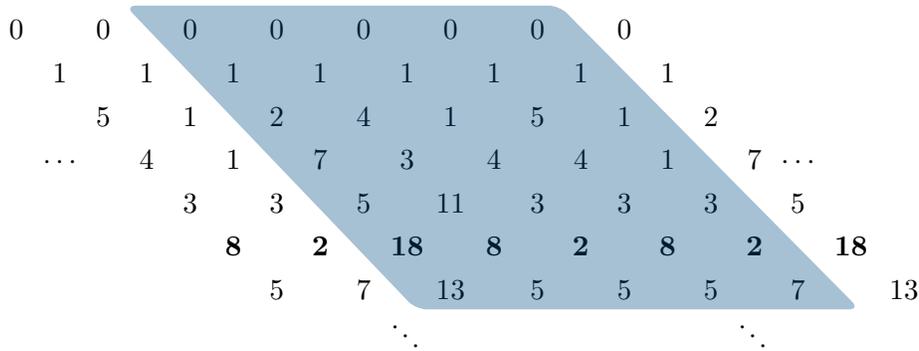

\begin{thm}[{\cite[Thm 3.6]{Tsc15}}] \label{thm:Tsc15_thm36}
Let $T$ be a triangulation of a once-punctured {\disk} $D_n$. 
Then the quiddity sequence $q_T = (a_1, \ldots, a_n)$ of $T$ is a quiddity sequence of an infinite frieze $\F_T$ of period $n$.
\end{thm}

\begin{remark}
In~\cite{BM09}, Baur and Marsh provide a construction of \emph{finite} periodic frieze patterns from once-punctured disks.  The entries of their friezes are in bijection with indecomposables of a type $D_n$ quiver and have the same form as the associated Auslander-Reiten quiver (e.g. see \cite{Sch08Dn}).  The friezes of Baur-Marsh match up with the first level (see Section \ref{subsec:growth_coeff}) of the infinite friezes studied in this paper except that the two rows of theirs associated to the leaves of the fork in the $D_n$ Dynkin diagram must be multiplied together point-wise to obtain the $n$th row in the infinite frieze.  
\end{remark}

Consider a triangulation of the annulus $C_{n,m}$, where the outer boundary component has $n$ marked points, and the inner boundary component has $m$ marked points.  Each of the outer boundary and inner boundary gives a quiddity sequence, and thus an infinite frieze. Unless otherwise stated, we consider the quiddity sequence coming from the outer boundary of an annulus.

\begin{defn}[{\cite[Def. 3.8]{BPT16}}]
Let $(a_1,\ldots, a_n)$ be the quiddity sequence of a periodic frieze. We say that $(a_1,\ldots, a_n)$ can be \emph{realized in an annulus (resp. a once-punctured disk)} if there is some $m \geq 1$ and a triangulation $T$ of $C_{n,m}$ (resp. a triangulation $T$ of $D_n$) such that $(a_1,\ldots, a_n)$ is the quiddity sequence of $T$.
\end{defn}

Let $\mathcal{U}$ be the universal cover as described in Section 3.3 of \cite{BPT16}, which is an infinite strip. Then every quiddity sequence can be realized in $\mathcal{U}$, and every triangulation of $\mathcal{U}$ gives rise to an infinite frieze (Theorem 5.2, \cite{BPT16}). Every entry of these infinite friezes can still be described in terms of matching numbers. 
We can even consider non-periodic infinite friezes, which can be obtained via non-periodic triangulations of an infinite strip (without any quotienting). 

\begin{ex}
Let $T$ be the following triangulation of $\mathcal{U}$, and $\gamma(i,j)$ be the arc from $i$ to $j$ on the outer boundary $\partial$:
\begin{center}
\begin{tikzpicture}[scale = .5]
		\tikzstyle{every node} = [font = \tiny]
		\foreach \x in {6}
		{
			\foreach \y in {-8}
			{
				\draw[<-] (\x-7,\y+4) -- (\x+7,\y+4);
				\fill (\x-8,\y+4) node {\small{$\partial'$}};
				\draw[->] (\x-7,\y-0) -- (\x+7,\y-0);
				\fill (\x-8,\y-0) node {\small{$\partial$}};

				\draw[] (\x-0,\y+4) -- (\x-4,\y-0);
				\draw[] (\x+6,\y+4) -- (\x+2,\y-0);
				\draw[] (\x-4,\y+0) -- (\x+4,\y-0);
				
				\draw[] (\x-6,\y+0) -- (\x-6,\y+4);
				\draw[] (\x-6,\y+0) -- (\x-4,\y+4);
				\draw[] (\x+4,\y+0) -- (\x+6,\y+4);
				\draw[] (\x+6,\y+0) -- (\x+6,\y+4);
				
				\draw[] (\x-4,\y+4) -- (\x-4,\y-0);
				\draw[](\x-2,\y+4) -- (\x-4,\y-0);
				
				\draw[](\x+0,\y+4) -- (\x-2,\y-0);
				\draw[](\x+0,\y+4) -- (\x+2,\y-0);

				\draw[] (\x-0,\y+4) .. controls (\x+1,\y+3.25) and (\x+3,\y+3.25) .. (\x+4,\y+4);
				\draw[] (\x-0,\y+4) .. controls (\x+2,\y+2.5) and (\x+4,\y+2.5) .. (\x+6,\y+4);
				\draw[] (\x-2,\y+0) .. controls (\x-1,\y+.75) and (\x+1,\y+.75) .. (\x+2,\y+0);
				\draw[thick,dashed] (\x-4,\y+0) .. controls (\x-2,\y+2) and (\x+2,\y+2) .. (\x+4,\y+0);
			
				\foreach \t in {-6,-4,-2,0,2,4,6}
				{
					\fill (\x+\t,\y-0) circle (.1);
					\fill (\x+\t,\y+4) circle (.1);
				}
					
				\fill (\x-4,\y-0) node [below] {$i$};
				\fill (\x+4,\y-0) node [below] {$j$};
				\fill (\x+0,\y+1.5) node [above] {$\gamma$};	
				\fill (\x-6,\y+2) node [left] {$\cdots$};	
				\fill (\x+6,\y+2) node [right] {$\cdots$};	
				
			}
			
		}
\end{tikzpicture}
\end{center}
Then the arc $\gamma(i,j)$, corresponds to the $(i,j)$-th entry in the infinite frieze pattern arising from this triangulation.
\end{ex}

\section{Infinite friezes of cluster algebra elements}
\label{sec:proof_frieze_laurent}

Our first result is the construction of infinite frieze patterns consisting of certain elements of a cluster algebra. Let $T$ be an ideal triangulation of a once-punctured {\disk} or an annulus, and let $\mathcal{A} = \mathcal{A}(B_T)$ be the coefficient-free cluster algebra 
associated to the signed adjacency matrix $B_T$. 
Let $Bd$ be a boundary component with $n$ marked points, where $n\geq 2$. 

\begin{defn}[Array of Laurent polynomials from generalized peripheral arcs]
\label{defn:array_of_peripheral_arcs}
We construct an array $\mathcal{F}_{Bd}$ corresponding to the set of all generalized arcs that are peripheral on $Bd$ as follows.
The entries of $\mathcal{F}_{Bd}$ are indexed by $(i,j)$, $i \leq j \in \mathbb{Z}$ such that our labeling convention is consistent with 
Definition \ref{defn:infinite_frieze}.

Set the entry at $(i,i)$ to be $0$. Now for every entry at $(i,j)$, $i<j$, we consider the (generalized) peripheral arc $\gamma(i,j)$ defined by taking the appropriate arc attached to the bottom boundary of the infinite strip 
and projecting this down to the once-punctured {\disk} or annulus.  Here, $i, j \in \mathbb{Z}$ so the marked points on the boundary $Bd$ are labeled as $(i\mod n)$ and $(j\mod n)$.  We let the entry at $(i,j)$ of $\mathcal{F}_{Bd}$ be the Laurent polynomial $x(\gamma(i,j))$ corresponding to the generalized arc $\gamma(i,j)$ (see Definition \ref{defn:generalized_arc_laurent_polynomial}).
Note that $\gamma(i,i+1)$ is a boundary edge, so $x(\gamma(i,i+1))=1$ by definition.
\end{defn}

\begin{thm}
The array ${F}_{Bd}$ Laurent polynomials corresponding to generalized peripheral arcs on $Bd$ form an infinite frieze pattern.
\label{thm:Frieze_Laurent}
\end{thm}

Before we prove the theorem, we recall \emph{skein relations}, and the related terminology.

\begin{defn}
A \emph{multicurve} is a finite multiset of generalized arcs and closed loops such that there are only a finite number of pairwise crossings among the collection. A multicurve is said to be \emph{simple} if there are no pairwise crossings among the collection, and no self-crossings.
\end{defn}

If a multicurve is not simple, there are two ways to resolve a crossing so that we obtain a multicurve that no longer contains that crossing. This process is known as \emph{smoothing}:

\begin{defn}
\label{defn:smoothing}
Let $\gamma, \gamma_1$, and $\gamma_2$ be generalized arcs or closed loops such that we have one of the following two cases:
\begin{enumerate}
\item $\gamma_1$ crosses $\gamma_2$ at a point $c$, or
\item $\gamma$ has a self-crossing at a point $c$.
\end{enumerate}
Then we let $C$ be the multicurve $\{ \gamma_1, \gamma_2\}$ or $\{\gamma\}$ depending on which of the two cases we are in. 
We define the \emph{smoothing of $C$ at the point $c$} to be the pair of configurations $C_+$ 
and $C_-$.
The multicurve 
$C_+$ (respectively, $C_-$) is the same as $C$ except for the local change that replaces the (self-)crossing {\Large $\times$} with the pair of segments {\Large $~_\cap^{\cup}$} (resp., {\large $\supset \subset$}).
See Fig.~\ref{fig:skein}.
\end{defn}

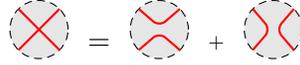
\begin{figure}[h!]
\newcommand\MacroDashedCircle{
\draw[fill=gray!20,densely dashed] circle(\LocalRadius); 
}
\newcommand\TikzSkeinRelation[1]{
\newcommand\LocalRadius{5.5pt}
\begin{tikzpicture}[scale=#1]
\MacroDashedCircle
\draw[red,thick] (-135:\LocalRadius) -- (45:\LocalRadius);
\draw[red,thick] (135:\LocalRadius) -- (-45:\LocalRadius);
\end{tikzpicture}
\begin{tikzpicture}[scale=#1]
\node (0,0) {$=$};
\end{tikzpicture}
\begin{tikzpicture}[scale=#1]
\MacroDashedCircle
\draw[red,thick,rounded corners=3pt] (135:\LocalRadius) --  (90:0.1*\LocalRadius) -- (45:\LocalRadius);
\draw[red,thick,rounded corners=3pt] (-135:\LocalRadius) -- (-90:0.1*\LocalRadius) -- (-45:\LocalRadius);
\end{tikzpicture}
\begin{tikzpicture}[scale=#1]
\node (0,0) {\tiny $+$};
\end{tikzpicture}
\begin{tikzpicture}[scale=#1]
\MacroDashedCircle
\draw[red, thick,rounded corners=3pt] (135:\LocalRadius) --  (180:0.1*\LocalRadius) -- (-135:\LocalRadius);
\draw[red, thick,rounded corners=3pt]  (45:\LocalRadius) -- (0:0.1*\LocalRadius) -- (-45:\LocalRadius);
\end{tikzpicture}
}
\begin{center}
\TikzSkeinRelation{2}
\caption{Skein relation for ordinary arcs}
\label{fig:skein}
\end{center}
\end{figure}

\begin{thm}[{\cite[Props. 6.4, 6.5, 6.6 and Cor. 6.18]{MW13}}]
\label{thm:MW_skein_relations}
Let $T$ be a triangulation of a marked surface, with or without punctures.
Let $C, C_+, C_-$ be as in Definition \ref{defn:smoothing}. Then we have the following identity in $\mathcal{A}(B_T)$: $$x_C = x_{C_+} + x_{C_-},$$ 
\end{thm}

We now look at an example of resolving a crossing using skein relations.

\begin{ex}
Consider the generalized arc $\gamma(2,9)$ in the annulus following the notation of Remark \ref{rem:gamIJ}. Using skein relations, we get that the Laurent polynomial $x(\gamma(2,9))$ corresponding to $\gamma(2,9)$ is the sum 
\[x(\gamma(4,7)) \, + \, x(\gamma(2,4)) \, x(\Brac_1).\]
See Fig.~\ref{fig:frieze:resolving_crossing}.

\begin{figure}[h!]
\begin{center}
\scalebox{.75}{\subfigure{
\begin{tikzpicture}
\draw (0,0) circle (1.25cm);
\foreach \x in {-144,-72,0,72,144} {
   \begin{scope}[rotate=\x]
    \node (\x) at (0,-1.25) [fill,circle,inner sep=1pt] {};
   \end{scope}
}

\filldraw[fill=gray!20](0,0) circle (.25cm);
\draw (144) node [above right] {$4$};
\draw (72) node [right] {$3$};
\draw (0) node [below] {$2$};
\draw (-72) node [left] {$1$};
\draw (-144) node [above left] {$5$};

\draw [ultra thick] (-.5,.0) to[out=90,in=125] (.4,.4) to[out=-180+135,in=30] (0);
\draw [ultra thick] (-.5,.0) to[out=270,in=220] (.4,-.4) to[out=30,in=-55] (144);

\draw[thin, red] (.62,-.2) circle (.2cm);

\draw (2,0) node[] {$=$}; 
\end{tikzpicture}

\begin{tikzpicture}

\draw (0,0) circle (1.25cm);
\foreach \x in {-144,-72,0,72,144} {
   \begin{scope}[rotate=\x]
    \node (\x) at (0,-1.25) [fill,circle,inner sep=1pt] {};
   \end{scope}
}

\filldraw[fill=gray!20](0,0) circle (.25cm);
\draw (144) node [above right] {$4$};
\draw (72) node [right] {$3$};
\draw (0) node [below] {$2$};
\draw (-72) node [left] {$1$};
\draw (-144) node [above left] {$5$};

\draw[ultra thick] (144) .. controls (-.8,.6) and (-.8,-.4) .. (0);

\draw (2,0) node[] {$+$}; 
\end{tikzpicture}

\begin{tikzpicture}
\draw (0,0) circle (1.25cm);
\foreach \x in {-144,-72,0,72,144} {
   \begin{scope}[rotate=\x]
    \node (\x) at (0,-1.25) [fill,circle,inner sep=1pt] {};
   \end{scope}
}

\filldraw[fill=gray!20](0,0) circle (.25cm);
\draw (144) node [above right] {$4$};
\draw (72) node [right] {$3$};
\draw (0) node [below] {$2$};
\draw (-72) node [left] {$1$};
\draw (-144) node [above left] {$5$};

\draw[ultra thick,blue] (0,0) circle (.5cm);
\draw[ultra thick] (0) .. controls (.8,-.6) and (.8,-.2) .. (144);

\end{tikzpicture}
}}
\caption{Resolving a self-crossing.}
\label{fig:frieze:resolving_crossing}
\end{center}
\end{figure}
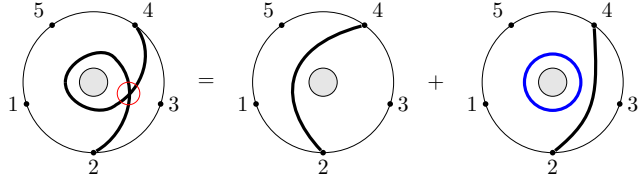
\end{ex}

The proof that the generalized (ordinary) arcs form an infinite frieze pattern now follows easily from the skein relation.

\begin{proof}[Proof of Theorem \ref{thm:Frieze_Laurent}]
%

To show that $\mathcal{F}_{Bd}$ is a frieze pattern, we need to check that for every diamond 
$$
\begin{array}{ccccccc}
 &b&\\
 a&&d\\
 &c&
\end{array}$$
in $\mathcal{F}_{Bd}$, the equation $ad-bc =1$ is satisfied.

From the labeling convention, every diamond in $\mathcal{F}_{Bd}$ has indices of the form
$$
\begin{array}{ccccccc}
 &(i+1,i+m)&\\
 (i,i+m)&&(i+1,i+m+1)\\
 &(i,i+m+1)&
\end{array}$$
where $m \in \mathbb{Z}_{\geq 1}$. We want to show that \[x(\gamma(i,i+m))\, x(\gamma(i+1,i+m+1)) = 1 + x(\gamma(i+1,i+m)) \, x(\gamma(i,i+m+1)).\]

Consider the arcs $\gamma(i,i+m), \gamma(i+1,i+m+1)$ drawn in the universal cover of our surface (Fig.~\ref{fig:skeinrelationsproof}, top).
The arcs $\gamma(i,i+m), \gamma(i+1,i+m+1)$ have exactly one crossing point. Using the skein relations (Theorem \ref{thm:MW_skein_relations}), we have that 
\begin{align*}
x(\gamma(i,i+m))\,x(\gamma(i+1,i+m+1))
= & x(\gamma(i,i+1))\,x(\gamma(i+m,i+m+1))\\ + & x(\gamma(i+1,i+m))\,x(\gamma(i,i+m+1)) \\
= & 1 + x(\gamma(i+1,i+m))\,x(\gamma(i,i+m+1)).
\end{align*}

This holds for every diamond in our pattern, and thus we have constructed a frieze pattern of Laurent polynomials corresponding to the set of all generalized peripheral arcs of $Bd$.
\end{proof}

\newcommand\skeinrelationscale{0.30}
\begin{figure}
	\begin{center}
	\begin{tikzpicture}[scale = \skeinrelationscale]
		\tikzstyle{every node} = [font = \small]
		\foreach \x in {0}
		{
			\foreach \y in {-8}
			{

				\draw[black,thick,->] (\x-7,\y-2) -- (\x+6,\y-2);
				
				\draw (\x-6,\y-2) .. controls (\x-4,\y-0) and (\x+1,\y-0) .. (\x+3,\y-2);
				\draw (\x-4,\y-2) .. controls (\x-2,\y-0) and (\x+3,\y-0) .. (\x+5,\y-2);

				\foreach \t in {9,11,18,20}
				{
					\fill (-15+\t,\y-2) circle (.15);

				}
				\fill (\x-6,\y-2) node[below]{\tiny $i$};
				\fill (\x-4,\y-2) node[below]{\tiny $i+1$};
				\fill (\x+2,\y-2) node[below]{\tiny $i+m$};
				\fill (\x+6,\y-2) node[below]{\tiny $i+m+1$};

				\draw(\x+10, \y-1) node[] {$=$};
			}
		}
		\end{tikzpicture}\\

\vspace{5mm}
\begin{tikzpicture}[scale = \skeinrelationscale]
		\tikzstyle{every node} = [font = \small]
		\foreach \x in {0}
		{
			\foreach \y in {-8}
			{
				\draw[black,thick,->] (\x-7,\y-2) -- (\x+6,\y-2);

				\draw (\x-6,\y-2) .. controls (\x-5.5,\y-1.5) and (\x-4.5,\y-1.5) .. (\x-4,\y-2);
				\draw (\x+3,\y-2) .. controls (\x+3.5,\y-1.5) and (\x+4.5,\y-1.5) .. (\x+5,\y-2);

				\foreach \t in {9,11,18,20}
				{
					\fill (-15+\t,\y-2) circle (.15);

				}
				\fill (\x-6,\y-2) node[below]{\tiny $i$};
				\fill (\x-4,\y-2) node[below]{\tiny $i+1$};
				\fill (\x+2,\y-2) node[below]{\tiny $i+m$};
				\fill (\x+6,\y-2) node[below]{\tiny $i+m+1$};
		\draw (\x+10, \y-1) node[] {$+$};
			}
		}
		\end{tikzpicture}
\quad
    \begin{tikzpicture}[scale = \skeinrelationscale]
		\tikzstyle{every node} = [font = \small]
		\foreach \x in {0}
		{
			\foreach \y in {-8}
			{
				\draw[black,thick,->] (\x-7,\y-2) -- (\x+6,\y-2);

				\draw (\x-6,\y-2) .. controls (\x-5,\y+1) and (\x+4,\y+1) .. (\x+5,\y-2);
				\draw (\x-4,\y-2) .. controls (\x-3,\y-1) and (\x+2,\y-1) .. (\x+3,\y-2);

				\foreach \t in {9,11,18,20}
				{
					\fill (-15+\t,\y-2) circle (.15);

				}
				\fill (\x-6,\y-2) node[below]{\tiny $i$};
				\fill (\x-4,\y-2) node[below]{\tiny $i+1$};
				\fill (\x+2,\y-2) node[below]{\tiny $i+m$};
				\fill (\x+6,\y-2) node[below]{\tiny $i+m+1$};
			}
		}
    \end{tikzpicture}
\end{center}
\caption{Applying skein relations to prove Theorem \ref{thm:Frieze_Laurent}}
\label{fig:skeinrelationsproof}
\end{figure}
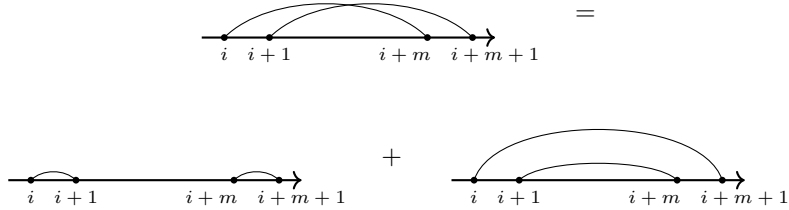

\section{Progression formulas}
\label{sec:proof_progression_formulas}

In Section \ref{sec:proof_frieze_laurent}, we constructed infinite frieze patterns of Laurent polynomial entries. In this section, we present formulas governing relations among these entries. These generalize the relations given in~\cite{BFPT16}, in the sense that~\cite[Thm. 2.5]{BFPT16} is a special case of the formulas.

\subsection{Complementary arcs}
\label{subsec:comp_arcs}

For $1 \leq i, j \leq n$ and $k \geq 1$, we let $\gamma_k(i,j)$ denote the generalized peripheral arc in $C_{n,m}$ or $D_n$ that lifts to the covering by the strip as follows (using the notation of Remark \ref{rem:gamIJ}):
$$\gamma_k(i,j) = \begin{cases} \gamma\left(i,j+(k-1)n\right) & \mathrm{~if~} i < j \\ 
\gamma\left(i,j + kn\right) & \mathrm{~if~} i \geq j \end{cases}.$$
That is, $\gamma_k(i,j)$ is the generalized peripheral arc that starts at the marked point $i$ and finishes at the marked point $j$ (possibly with $i=j$) with $(k-1)$ self-crossings such that the boundary Bd is to the right of the curve as we trace it.  

\begin{defn}[complementary arc] Using the above shorthand notation, we define the arc complementary to $\gamma_k = \gamma_k(i,j)$ as 
$$\gamma_k(i,j)^C =  \begin{cases} \gamma\left(j, i+kn \right) & \mathrm{~if~} i < j \\ 
\gamma\left(j, i + (k-1)n\right) & \mathrm{~if~} i \geq j \end{cases}.$$
\end{defn}

\begin{rem}
When $i\not = j$, the complementary arc $\gamma_k^C$ to $\gamma_k=\gamma_k(i,j)$ can be described as the generalized arc (i.e. up to homotopy) starting at the marked point $j$ and finishing at the marked point $i$ and retaining $(k-1)$ self-crossings while following the orientation of the surface.  See Fig.~\ref{fig:complementary_arcs}.  In this case, $(\gamma_k^C)^C = \gamma_k$.  On the other hand, when $i=j$, observe that complementation is non-involutive and simply decreases the number of self-intersections by one. 
For example, we think of the empty arc as the complementary arc of the once-punctured monogon from $i$ to $i$, and of the once-punctured monogon from $i$ to $i$ as the complementary arc of the loop from $i$ to $i$ which goes around the boundary $Bd$ twice. 
\end{rem}

\begin{figure}[h]
\begin{center}
\newcommand\pentascale{0.6} 
\def\FigScale{0.42} 
\tikzset{->-/.style={decoration={
  markings,
  mark=at position #1 with {\arrow{stealth}}},postaction={decorate}}}
 \newcommand\plaindisk[1]
 {
 \begin{scope}[scale=\pentascale]
 \coordinate (a) at (0,0);
\draw[] (0,0) circle (1.25cm);
\node (a) at (0,0) [opacity=0.65,fill,circle,inner sep=1.5pt] {};

\foreach \x in {-144,-72,0,72,144} {
   \begin{scope}[rotate=\x]
   \coordinate (\x) at (0,-1.25);
    \node (\x) at (0,-1.25) [opacity=0.5,fill,circle,inner sep=0.9pt] {};
   \end{scope}
}

\end{scope}
}

 \begin{tikzpicture}
\plaindisk{}
\draw[thick,->-=.7, bend right=30] (-72) to (72); 
\draw (-72) node [fill,blue,circle,inner sep=1.3pt] {};
\draw (72) node [black,fill=mygreen,circle,inner sep=1.3pt] {};
\draw (-72) node [left,blue] {$i$};
\draw (72) node [right,mygreen] {$j$};
\end{tikzpicture}
\begin{tikzpicture}
\end{tikzpicture}
\hskip -2mm
\begin{tikzpicture}
\plaindisk{}
\draw[thick,->-=.7,bend right=90] (72) to (-72); 
\draw (72) node [black,fill=mygreen,circle,inner sep=1.3pt] {};
\draw (-72) node [blue,fill,circle,inner sep=1.3pt] {};
\draw (-72) node [left,blue] {$i$};
\draw (72) node [right,mygreen] {$j$};
\end{tikzpicture}
\quad\quad\quad
\def\MacroFigSlideBraceletsAnnulus{
\draw (0,0) circle (2cm);
\node (a) at (0,0) [opacity=0.65,fill,circle,inner sep=2pt] {};

\draw (120:2cm) node [fill,blue,circle,inner sep=1.3pt] {};

\draw (-180:2cm) node [black,fill=mygreen,circle,inner sep=1.3pt] {};

\draw[] (120:2cm) node[left] {$i$};
\draw[] (-180:2cm) node[left] {$j$};
}
\begin{tikzpicture}[scale=\FigScale] 
\MacroFigSlideBraceletsAnnulus

\draw[thick, black, rounded corners,->-=0.25] ([shift=(120:2cm)]0,0) -- ([shift = (-170:1.1cm)]0,0) 
arc (-170:120:1.1cm);

\draw[thick, black, rounded corners] ([shift=(120:1.1cm)]0,0) -- ([shift = (-180:1.5cm)]0,0) 
arc (-180:140:1.5cm);

\draw[thick, black, ->-=.7] ([shift=(140:1.5cm)]0,0) -- ([shift = (-180:2cm)]0,0);
\end{tikzpicture}
\quad
\begin{tikzpicture}[scale=\FigScale] 
\MacroFigSlideBraceletsAnnulus

\draw[thick,black, rounded corners, ,->-=.1] ([shift=(-180:2cm)]0,0) 
-- ([shift = (-220:1.5cm)]0,0) 
-- ([shift = (-200:0.7cm)]0,0) 
arc (-200:140:0.7cm);

\draw[ thick,black, rounded corners] ([shift=(140:0.7cm)]0,0) -- ([shift = (-170:1.1cm)]0,0) 
arc (-170:120:1.1cm);

\draw[thick,black, rounded corners,->-=0.9] ([shift=(120:1.1cm)]0,0) -- ([shift = (-180:1.5cm)]0,0) arc (-180:110:1.5cm) -- ([shift = (120:2cm)]0,0);
\end{tikzpicture}
\caption{Examples of involutive complementary arcs $\gamma_1$, $\gamma_1^C$ and $\gamma_3$, $\gamma_3^C$.}\label{fig:complementary_arcs}
\end{center}
\end{figure}
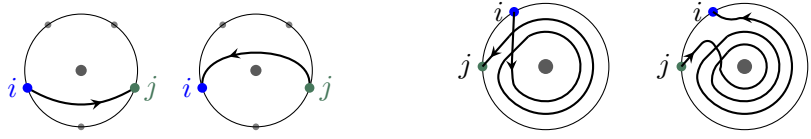

\begin{rem}
It is well-known that a finite frieze pattern has glide-symmetry. In an infinite frieze pattern $\mathcal F$, we do not have glide-symmetry, but instead we have what we call \emph{complement-symmetry} along with the translation-symmetry. 
See Fig.~\ref{fig:complement_symmetry} for an example. Note that each $kn$-th row ($k \in \mathbb{Z}_{\geq 1}$) of $\mathcal F$ corresponds to the generalized arcs of the form $\gamma_k(i,i)$ for which complementation is not involutive.
\input{folder_frieze/fig_complement_symmetry}
\end{rem}

\begin{thm}[progression formulas]
\label{thm:progression_formula}
Let $\ga_1$ be a peripheral arc or a boundary edge of $(S,M)$ starting and finishing at points $i$ and $j$. 
For $k \geq 2$ and $1\leq m \leq k-1$, we have
\[
x(\gamma_k) = x(\gamma_m) \, x(Brac_{k-m}) + x(\gamma_{k-2m+1}^C).
\]
For $r \geq 0$, $\gamma_{-r}^C$ is defined to be the curve $\gamma_{r+1}$ with a kink, so that $x(\gamma_{-r}^C) = - x(\gamma_{r+1})$.
\end{thm}

\begin{rem}\label{rem:m1}
Special cases of above theorem are when $m=1$ (see Fig.~\ref{fig:frieze:gamma4})
and $m=k-1$.
We have
\begin{gather}
\label{eqn:m_1}
x(\gamma_k) = x(\gamma_1) \, x(Brac_{k-1}) + x(\gamma_{k-1}^C),\\
\label{eqn:m_k_min_1}x(\gamma_k) = x(\gamma_{k-1}) \, x(Brac_{1}) - x(\gamma_{k-2})
\end{gather}
Compare (\ref{eqn:m_k_min_1}) with~\cite[Thm. 2.5]{BFPT16}.  In particular, we use the definition $x(\gamma_{-k+3}^C) = - x(\gamma_{k-2})$ in (\ref{eqn:m_k_min_1}).

Furthermore, if $\ga_1$ is the boundary edge between $i$ and $i+1$, then $x(\ga_1)=1$, so, due to (\ref{eqn:m_1}), we have
\begin{gather}
\label{eqn:m_1_growth_coefficient}
 x(Brac_{k}) = x(\gamma_{k+1}) - x(\gamma_{k}^C).
\end{gather}
\end{rem}

\subsection{Proof of Theorem \ref{thm:progression_formula}}
\label{subsec:prog_thm_proof}

Let $\gamma_k:=\gamma_k(i,j)$.
We draw $\gamma_k$ so that it first closely follows the other boundary (or the puncture) and then spirals out.

\begin{figure}\begin{center}
\def\mycolor{black}
\tikzset{->-/.style={decoration={
  markings,
  mark=at position #1 with {\arrow{stealth}}},postaction={decorate}}}
\def\FigScale{0.46}
\def\MacroFigSlideBraceletsAnnulus{
\draw (0,0) circle (2cm);
\draw[fill=gray!20] (0,0) circle (.4cm);
\draw (120:2cm) node [fill,blue,circle,inner sep=1.3pt] {};
\draw (-180:2cm) node [black,fill=mygreen,circle,inner sep=1.3pt] {};
\draw[] (120:2cm) node[left] {$i$};
\draw[] (-180:2cm) node[left] {$j$};
}
\begin{tikzpicture}[scale=\FigScale] 
\MacroFigSlideBraceletsAnnulus
\draw[thick, \mycolor, rounded corners, ,->-=.08] ([shift=(120:2cm)]0,0) -- ([shift = (-200:0.7cm)]0,0)
arc (-200:140:0.7cm);

\draw[thick, \mycolor, rounded corners] ([shift=(140:0.7cm)]0,0) -- ([shift = (-170:1.1cm)]0,0) 
arc (-170:120:1.1cm);

\draw[thick, \mycolor, rounded corners] ([shift=(120:1.1cm)]0,0) -- ([shift = (-180:1.5cm)]0,0) 
arc (-180:140:1.5cm);

\draw[thick, \mycolor, ->-=.7] ([shift=(140:1.5cm)]0,0) -- ([shift = (-180:2cm)]0,0);

\draw[thin, red] (125:1.5cm) circle (.25cm);

\draw (3,0) node[] {$=$}; 
\end{tikzpicture}
\begin{tikzpicture}[scale=\FigScale] 
\MacroFigSlideBraceletsAnnulus
\draw[ thick,\mycolor, rounded corners,->-=0.7] ([shift=(120:2cm)]0,0) -- ([shift = (140:1.7cm)]0,0) -- ([shift = (-180:2cm)]0,0); 

\draw[thick,blue, rounded corners] 
 ([shift=(140:0.7cm)]0,0) -- ([shift = (-180:1.5cm)]0,0)   arc (-180:140:1.5cm) -- ([shift = (-180:1.1cm)]0,0);
  
  \draw[thick,blue, rounded corners] 
 ([shift = (-180:1.1cm)]0,0)
  arc (-180:140:1.1cm) -- ([shift = (-180:0.7cm)]0,0);
  
\draw[ thick,blue] ([shift=(-180:0.7cm)]0,0) arc (-180:140:0.7cm);
\draw (3,0) node[] {$+$}; 
\end{tikzpicture}
\begin{tikzpicture}[scale=\FigScale] 
\MacroFigSlideBraceletsAnnulus

\draw[thick,\mycolor, rounded corners, ,->-=.1] ([shift=(-180:2cm)]0,0) 
-- ([shift = (-220:1.5cm)]0,0) 
-- ([shift = (-200:0.7cm)]0,0) 
arc (-200:140:0.7cm);

\draw[ thick,\mycolor, rounded corners] ([shift=(140:0.7cm)]0,0) -- ([shift = (-170:1.1cm)]0,0) 
arc (-170:120:1.1cm);

\draw[thick,\mycolor, rounded corners,->-=0.9] ([shift=(120:1.1cm)]0,0) -- ([shift = (-180:1.5cm)]0,0) arc (-180:110:1.5cm) -- ([shift = (120:2cm)]0,0);
\end{tikzpicture}

\caption{Case $m=1:$ By the progression formula (Theorem \ref{thm:progression_formula}), we have $x(\gamma_4 (i,j) ) = x(\gamma_1 (i,j) ) \  x(Brac_{3}) + x(\gamma_{3}(i,j)^C)$.}
\label{fig:frieze:gamma4}
\end{center}\end{figure}
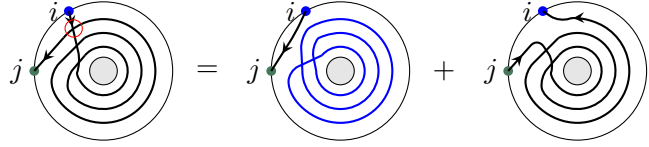

\begin{figure}
\begin{tikzpicture}[scale = 0.34]
\tikzstyle{every node} = [font = \small]
	\foreach \x in {0}
	{
	\foreach \y in {-8}
		{
	
		\draw[black,thick,->] (\x-22,\y-2) -- (\x+20,\y-2);
		
			\foreach \t in {-20,-18,-16,-14,-12,-10, -8, -6, -4, -2, 0, 2, 4, 6, 8, 10,12,14,16,18}
				{
					\fill (\t,\y-2) circle (.1);

				}
				\foreach \t in {-18,-14,-10, -6, -2, 2, 6,10,14,18}
				{
					\fill (\t,\y-2) node[below]{$j$};

				}
				\foreach \t in {-20,-16,-12, -8, -4, 0, 4, 8,12,16}
				{
					\fill (\t,\y-2) node[below]{$i$};

				}	
				\foreach \t in {-21,-17,-13,-9, -5, -1, 3, 7, 11,15,19}
				{
					\draw[dashed] (\t,\y-2)--(\t, \y+8);

				}
			
				\fill (+17,\y+8) node[above]{$Reg_{k-1}$};
				\fill (+1,\y+8) node[above]{$\cdots$};
				\fill (-11,\y+8) node[above]{$Reg_{2}$};
				\fill (-15,\y+8) node[above]{$Reg_{1}$};
				\fill (-19,\y+8) node[above]{$Reg_{0}$};
			
				
				\draw[]  plot[rounded corners] coordinates {(-22,\y+7.5) (14,\y-2) };
				\draw[]  plot[rounded corners] coordinates {(-22,\y+6.5) (10,\y-2) };
				\draw[]  plot[rounded corners] coordinates {(-22,\y+5.5) (6,\y-2) };
				\draw[]  plot[rounded corners] coordinates {(-22,\y+4.5) (2,\y-2) };
				\draw[]  plot[rounded corners] coordinates {(-22,\y+3.5) (-2,\y-2) };
				\draw[]  plot[rounded corners] coordinates {(-22,\y+2.5) (-6,\y-2) };
				\draw[]  plot[rounded corners] coordinates {(-22,\y+1.5) (-10,\y-2) };
				\draw[]  plot[rounded corners] coordinates {(-22,\y+0.5) (-14,\y-2) };
				\draw[]  plot[rounded corners] coordinates {(-22,\y-0.5) (-18,\y-2) };
				
				 \draw[very thick]  plot[rounded corners] coordinates {(-20,\y-2) (-20,\y+8) (18,\y-2)};
				 \draw[]  plot[rounded corners] coordinates {(-16,\y-2) (-16,\y+8) (19,\y-1)};
				 \draw[]  plot[rounded corners] coordinates {(-12,\y-2) (-12,\y+8)(19,\y+0) };
				\draw[]  plot[rounded corners] coordinates {(-8,\y-2) (-8,\y+8)(19,\y+1) };
				\draw[]  plot[rounded corners] coordinates {(-4,\y-2) (-4,\y+8)(19,\y+2) };
				\draw[]  plot[rounded corners] coordinates {(0,\y-2) (0,\y+8)(19,\y+3) };
				\draw[]  plot[rounded corners] coordinates {(4,\y-2) (4,\y+8)(19,\y+4) };
				\draw[]  plot[rounded corners] coordinates {(8,\y-2) (8,\y+8)(19,\y+5) };
				\draw[]  plot[rounded corners] coordinates {(12,\y-2) (12,\y+8)(19,\y+6) };
				\draw[]  plot[rounded corners] coordinates {(16,\y-2) (16,\y+8)(19,\y+7) };

\draw[red] (-20,\y+2) circle (.3);
\draw[red] (-19.9,\y+2.15) node[right] {$c_0$};
}
}
\end{tikzpicture}
\caption{Lift of $\gamma_k$ for $k=10$, $m=4$ drawn on the strip.}
\label{fig:progression_unresolved}

\begin{tikzpicture}[scale = 0.47]
\tikzstyle{every node} = [font = \small]
	\foreach \x in {0}
	{
	\foreach \y in {-8}
		{
	
		\draw[black,thick,->] (\x-15,\y-2) -- (\x+15,\y-2);
		
			\foreach \t in {-14,-12,-10, -8, -6, -4, -2, 0, 2, 4, 6, 8, 10,12,14}
				{
					\fill (\t,\y-2) circle (.1);

				}
				\foreach \t in {-14,-10, -6, -2, 2, 6,10,14}
				{
					\fill (\t,\y-2) node[below]{$j$};

				}
				\foreach \t in {-12, -8, -4, 0, 4, 8,12}
				{
					\fill (\t,\y-2) node[below]{$i$};

				}	
				\foreach \t in {-13,-9, -5, -1, 3, 7, 11}
				{
					\draw[dashed] (\t,\y-2)--(\t, \y+8);

				}
			
				\fill (+13,\y+8) node[above]{$Reg_{k-1}$};
				\fill (+9,\y+8) node[above]{$\cdots$};
				\fill (+5,\y+8) node[above]{$Reg_{4}$};
				\fill (+1,\y+8) node[above]{$Reg_{3}$};
				\fill (-3,\y+8) node[above]{$Reg_{2}$};
				\fill (-7,\y+8) node[above]{$Reg_{1}$};
				\fill (-11,\y+8) node[above]{$Reg_{0}$};

				\draw[very thick]  plot[rounded corners] coordinates {(-14,\y+2)(-12,\y+1.5)(-12,\y+7.5)(12,\y+1.5)(12,\y+7.5)(14,\y+7)};
				\draw  plot[rounded corners] coordinates {(-14,\y+3)(-8,\y+1.5)(-8,\y+7.5)(14,\y+2)};
				\draw  plot[rounded corners] coordinates {(-14,\y+4)(-4,\y+1.5)(-4,\y+7.5)(14,\y+3)};
				\draw  plot[rounded corners] coordinates {(-14,\y+5)(0,\y+1.5)(0,\y+7.5)(14,\y+4)};
				\draw  plot[rounded corners] coordinates {(-14,\y+6)(4,\y+1.5)(4,\y+7.5)(14,\y+5)};
				\draw  plot[rounded corners] coordinates {(-14,\y+7)(8,\y+1.5)(8,\y+7.5)(14,\y+6)};

				\draw  plot[rounded corners] coordinates {(-14,\y-1)(-10,\y-2)};
				\draw  plot[rounded corners] coordinates {(-14,\y-0)(-6,\y-2)};
				\draw  plot[rounded corners] coordinates {(-14,\y+1)(-2,\y-2)};

				\draw[very thick]  plot[rounded corners] coordinates {(-12,\y-2)(-12, \y+1.5)(2,\y-2)};
				\draw  plot[rounded corners] coordinates {(-8,\y-2)(-8, \y+1.5)(6,\y-2)};
				\draw  plot[rounded corners] coordinates {(-4,\y-2)(-4, \y+1.5)(10,\y-2)};
				\draw  plot[rounded corners] coordinates {(0,\y-2)(0, \y+1.5)(14,\y-2)};
				\draw  plot[rounded corners] coordinates {(4,\y-2)(4, \y+1.5)(14,\y-1)};
				\draw  plot[rounded corners] coordinates {(8,\y-2)(8, \y+1.5)(14,\y-0)};
				\draw  plot[rounded corners] coordinates {(12,\y-2)(12, \y+1.5)(14,\y+1)};

\draw[red] (-12,\y+1.5) circle (.3);
\draw[red] (-11.9,\y+1.6) node[right] {$c_0$};
}
}
\end{tikzpicture}
\caption{Lifts of $\gamma_{m}$ and $Brac_{k-m}$ for $k=10$, $m=4$ drawn on the strip.}
\label{fig:progression_bracelet}

\begin{tikzpicture}[scale = 0.47]
\tikzstyle{every node} = [font = \small]
	\foreach \x in {0}
	{
	\foreach \y in {-8}
		{
	
		\draw[black,thick,->] (\x-15,\y-2) -- (\x+15,\y-2);
		
			\foreach \t in {-14,-12,-10, -8, -6, -4, -2, 0, 2, 4, 6, 8, 10,12,14}
				{
					\fill (\t,\y-2) circle (.1);

				}
				\foreach \t in {-14,-10, -6, -2, 2, 6,10,14}
				{
					\fill (\t,\y-2) node[below]{$j$};

				}
				\foreach \t in {-12, -8, -4, 0, 4, 8,12}
				{
					\fill (\t,\y-2) node[below]{$i$};

				}	
				\foreach \t in {-13,-9, -5, -1, 3, 7, 11}
				{
					\draw[dashed] (\t,\y-2)--(\t, \y+8);

				}
			
				\fill (+13,\y+8) node[above]{$Reg_{0}$};
				\fill (+9,\y+8) node[above]{$Reg_{-1}$};
				\fill (+5,\y+8) node[above]{$Reg_{-2}$};
				\fill (+1,\y+8) node[above]{$Reg_{-3}$};
				\fill (-3,\y+8) node[above]{$Reg_{-4}$};
				\fill (-7,\y+8) node[above]{$\cdots$};
				\fill (-11,\y+8) node[above]{$Reg_{-(k-m)}$};
			
		 \draw[very thick]  plot[rounded corners] coordinates {(12,\y-2) (12,\y+1.5) (-12,\y+7.5) (-12,\y+1.5) (2,\y-2)};
				 \draw  plot[rounded corners] coordinates {(8,\y-2) (8,\y+1.5) (-14,\y+7)};
				 \draw  plot[rounded corners] coordinates {(4,\y-2) (4,\y+1.5) (-14,\y+6)};
				 \draw  plot[rounded corners] coordinates {(0,\y-2) (0,\y+1.5) (-14,\y+5)};
				 \draw  plot[rounded corners] coordinates {(-4,\y-2) (-4,\y+1.5) (-14,\y+4)};
				 \draw  plot[rounded corners] coordinates {(-8,\y-2) (-8,\y+1.5) (-14,\y+3)};
				 \draw  plot[rounded corners] coordinates {(-12,\y-2) (-12,\y+1.5) (-14,\y+2)};
				
				\draw  plot[rounded corners] coordinates {(-10,\y-2)(-14,\y-1)};
				 \draw  plot[rounded corners] coordinates {(-6,\y-2)(-14,\y+0)};
				 \draw  plot[rounded corners] coordinates {(-2,\y-2)(-14,\y+1)};
				 \draw  plot[rounded corners] coordinates {(14,\y+1) (12,\y+1.5) (12,\y+7.5) (14,\y+7)};
				 \draw  plot[rounded corners] coordinates {(14,\y+2) (-8,\y+7.5) (-8,\y+1.5) (6,\y-2)};
				 \draw  plot[rounded corners] coordinates {(14,\y+3) (-4,\y+7.5) (-4,\y+1.5) (10,\y-2)};
				 \draw  plot[rounded corners] coordinates {(14,\y+4) (0,\y+7.5) (0,\y+1.5) (14,\y-2)};
				 \draw  plot[rounded corners] coordinates {(14,\y+0) (8,\y+1.5) (8,\y+7.5) (14,\y+6)};
				\draw  plot[rounded corners] coordinates {(14,\y-1) (4,\y+1.5) (4,\y+7.5) (14,\y+5)};

\draw[red] (12,\y+1.5) circle (.3);
\draw[red] (12.1,\y+1.6) node[right] {$c_0$};
}
}
\end{tikzpicture}
\caption{Lift of $\gamma_{k-2m+1}^C$ for $k=10$, $m=4$ drawn on the strip.}
\label{fig:progression}

\begin{tikzpicture}[scale = 0.34]
\tikzstyle{every node} = [font = \small]
	\foreach \x in {0}
	{
	\foreach \y in {-8}
		{
	
		\draw[black,thick,->] (\x-22,\y-2) -- (\x+20,\y-2);
		
			\foreach \t in {-20,-18,-16,-14,-12,-10, -8, -6, -4, -2, 0, 2, 4, 6, 8, 10,12,14,16,18}
				{
					\fill (\t,\y-2) circle (.1);

				}
				\foreach \t in {-18,-14,-10, -6, -2, 2, 6,10,14,18}
				{
					\fill (\t,\y-2) node[below]{$j$};

				}
				\foreach \t in {-20,-16,-12, -8, -4, 0, 4, 8,12,16}
				{
					\fill (\t,\y-2) node[below]{$i$};

				}	
				\foreach \t in {-21,-17,-13,-9, -5, -1, 3, 7, 11,15,19}
				{
					\draw[dashed] (\t,\y-2)--(\t, \y+8);

				}
			
				\fill (+17,\y+8) node[above]{$\cdots$};
				\fill (+9,\y+8) node[above]{$Reg_{-k+2m-1}$};
				\fill (+1,\y+8) node[above]{$\cdots$};
				\fill (-7,\y+8) node[above]{$Reg_{1}$};
				\fill (-11,\y+8) node[above]{$Reg_{0}$};
				\fill (-15,\y+8) node[above]{$Reg_{-1}$};
				\fill (-19,\y+8) node[above]{$Reg_{-2}$};

				\draw[very thick]  plot[rounded corners] coordinates {(-12,\y-2) (-12,\y+6)(-20,\y+8)(-20,\y+6)(10,\y-2)};
				\draw plot[rounded corners] coordinates {(-8,\y-2) (-8,\y+6)(-16,\y+8)(-16,\y+6)(14,\y-2)};
				\draw plot[rounded corners] coordinates {(-4,\y-2) (-4,\y+6)(-12,\y+8)(-12,\y+6)(18,\y-2)};
				\draw plot[rounded corners] coordinates {(0,\y-2) (0,\y+6)(-8,\y+8)(-8,\y+6)(19,\y-1.2)};
				\draw plot[rounded corners] coordinates {(4,\y-2) (4,\y+6)(-4,\y+8)(-4,\y+6)(19,\y-.13)};
				\draw plot[rounded corners] coordinates {(8,\y-2) (8,\y+6)(-0,\y+8)(-0,\y+6)(19,\y+.94)};
				\draw plot[rounded corners] coordinates {(12,\y-2) (12,\y+6)(4,\y+8)(4,\y+6)(19,\y+2.01)};
				\draw plot[rounded corners] coordinates {(16,\y-2) (16,\y+6)(8,\y+8)(8,\y+6)(19,\y+3.08)};
				
				\draw plot[rounded corners] coordinates {(19,\y+6.29)(12,\y+8)(12,\y+6)(19,\y+4.15)};
				\draw plot[rounded corners] coordinates {(19,\y+7.36)(16,\y+8)(16,\y+6)(19,\y+5.22)};
				
				\draw plot[rounded corners] coordinates {(-16,\y-2)(-16,\y+6)(-21,\y+7.36)};
				\draw plot[rounded corners] coordinates {(-20,\y-2)(-20,\y+6)(-21,\y+6.29)};
				\draw plot[rounded corners] coordinates {(-21,\y+5.22)(6,\y-2)};
				\draw plot[rounded corners] coordinates {(-21,\y+4.15)(2,\y-2)};
				\draw plot[rounded corners] coordinates {(-21,\y+3.08)(-2,\y-2)};
				\draw plot[rounded corners] coordinates {(-21,\y+2.01)(-6,\y-2)};
				\draw plot[rounded corners] coordinates {(-21,\y+.94)(-10,\y-2)};
				\draw plot[rounded corners] coordinates {(-21,\y-.13)(-14,\y-2)};
				\draw plot[rounded corners] coordinates {(-21,\y-1.2)(-18,\y-2)};

\draw[red] (-12,\y+6) circle (.3);
\draw[red] (-11.9,\y+6.15) node[right] {$c_0$};
}
}
\end{tikzpicture}
\caption{Lift of $\gamma_{k-2m+1}^C$ for $k=10$, $m=8$ drawn on the strip.}
\label{fig:progression_kink}
\end{figure}

In the covering via the infinite horizontal strip, we draw the lower boundary $Bd$ so that $i$ is drawn to the left of $j$ in each frame.
We draw each representative of $\gamma_k$ as follows.
We start at a frame $Reg_0$.
Starting from a vertex labeled $i$, our pencil goes north, passing through all of the $(k-1)$ crossings.  When we get to the very north, we turn southeast, and finish at a vertex labeled $j$, which is located in the frame $k-1$ frames (respectively, $k$ frames) east of $Reg_0$ if $i\neq j$ (respectively, if $i=j$). See Fig.~\ref{fig:progression_unresolved}. 

We order the crossings of $\gamma_k$ so that the first crossing is the one closest to $Bd$ and the $(k-1)$-th crossing is the one furthest away from $Bd$.
In each frame,
consider the $m$-th crossing $c$ of $\gamma_k$.
Denote the segments meeting at $c$ by $north_{c}$, $south_{c}$, $east_{c}$, and $west_{c}$, so that $north_c$ is the segment drawn north of $c$, $east_{c}$ is the segment drawn east of $c$, et cetera.

If we resolve all representatives of the $m$-th crossing $c$ by glueing $north_c$ with $west_c$ as well as glueing $south_c$ with $east_c$, we get two curves, $\gamma_{m}$ and $Brac_{k-m}$ (see Fig.~\ref{fig:progression_bracelet}). This explains the first summand of Theorem \ref{thm:progression_formula}. 
The second summand (see Figs. \ref{fig:progression} and \ref{fig:progression_kink}) of Theorem \ref{thm:progression_formula} is explained by the following Lemma.

\begin{lem}\label{lem:kminus2mplus1}
Suppose we have the same setup as above for $\gamma_k$.
Suppose we resolve all representatives of the $m$-th crossing $c$ by glueing $north_c$ with $east_c$ as well as glueing $south_c$ with $west_c$. Then we get one curve, $\gamma_{k-2m+1}^C$. See Figs. \ref{fig:progression} and \ref{fig:progression_kink}.
\end{lem}
\begin{proof}[Proof of Lemma \ref{lem:kminus2mplus1}]

Our pencil starts at $Reg_0$ at the starting vertex $i$ and heads north. 
Let $c_0$ denote the $m$-th crossing of $\gamma_k$ at frame $Reg_0$.
When we get to $c_0$, since the segment $south_{c_0}$ is glued to  $west_{c_0}$, we pivot west of $c_0$.
As we trace $west_{c_0}$ with our pencil, we pass through $(k-m)$ other frames west of $Reg_0$, denoted $Reg_{-1}$, $\dots$, $Reg_{-(k-m)}$.
When we get to $Reg_{-(k-m)}$, the curve bends south (because this is how we've chosen to draw $\gamma_k$).
As our pencil traces south, we hit the $m$-th crossing $c'$ in  $Reg_{-(k-m)}$.
Because we have glued $north_{c'}$ with $east_{c'}$, the curve bends east at $c'$.

Since this crossing is the $m$-th crossing closest to $Bd$, 
there are $m-1$ (possibly $m-1=0$) other crossings beneath it closer to $Bd$.
Therefore, our pencil will end (at a representative of the vertex $j$) when we get to the frame that is $m-1$ frames (respectively, $m$ frames) away east of $Reg_{-(k-m)}$ if $i \neq j$ (respectively, if $i=j$).

We consider the two possibilities: $k-m>m-1$ or $k-m \leq m-1$.  
First, assume $k-m>m-1$.
If $i\neq j$, we end at the frame $Reg_{-(k-m)+(m-1)}=Reg_{-(k-2m+1)}$, which we have passed earlier.
Hence we have traced the curve $\gamma_{k-2m+1}^C$ (see Fig.~\ref{fig:progression}).
If $i=j$, we end at the frame $Reg_{-(k-m)+m}=Reg_{-(k-2m)}$, which we have passed earlier.
Since we started at $Reg_0$, and since $i=j$, we have traced the curve $\gamma_{k-2m}$. (In particular, if $k=2m$, then we have traced a curve that is contractible to the point $i$).
By definition, this curve is $\gamma_{k-2m+1}^C$.

If $k-m \leq m-1$, we pass $Reg_0$, crossing our pencil mark exactly once before continuing to another frame east of $Reg_0$. 
If $i \neq j$, we end our drawing at a frame that is $m-1-(k-m)=-k+2m-1$ frames away east of $Reg_0$. (In the case $k-m = m-1$, this quantity is zero and indeed 
we end at the original frame $Reg_0$, and we cross our pencil mark exactly once before ending at $j$.)
Denote this frame $Reg_{(-k+2m-1)}$.  Hence we have traced the curve $\gamma_{(-k+2m)}$ with a kink (note that $(-k +2m \geq 1)$ in this case).
By definition, this curve is $\gamma_{k-2m+1}^C$ (see Fig.~\ref{fig:progression_kink}).
If $i = j$, we end our drawing at a frame that is $m-(k-m)=-k+2m$ frames away east of $Reg_0$. Denote this frame $Reg_{(-k+2m)}$. 
Hence we have traced the curve $\gamma_{(-k+2m)}$ with a kink.
By definition, this curve is $\gamma_{k-2m+1}^C$.
\end{proof}

\section{Bracelets and growth coefficients}
\label{sec:brac_and_growth}

In~\cite[Thm. 2.2]{BFPT16}, the authors show that in an n-periodic infinite frieze of positive integers, the difference between the entries in rows $(nk+1)$ \& $(nk-1)$ and the same column is a constant.
These differences are also constant in our infinite friezes of Laurent polynomials, and we give geometric interpretations to these differences. Following \cite[Def. 2.3]{BFPT16}, we refer to these constants as \emph{growth coefficients}.

\subsection{Growth coefficients}
\label{subsec:growth_coeff}

\def\mylevel{k}
\def\myremainder{j}

We say that \emph{level $\mylevel$} of a frieze consists of the entries of the frieze indexed by $(i,i+(\mylevel-1)n+\myremainder)$ where $\myremainder=1,\dots,n$, that is, the entries in the $(\mylevel-1)n+\myremainder$-th row of the frieze.
Compare the following proposition with \cite[Thm. 2.2]{BFPT16}.

\begin{prop}
\label{prop:see_BFPT16_thm2_2}
Let $\mathcal{F}=\{\mathcal{F}_{i,j}\}$ 
be an infinite periodic frieze pattern as described in Section \ref{sec:proof_frieze_laurent}.
For each $k \geq 1$, we have 
\begin{gather*}
 x(Brac_{k})  = \mathcal{F}_{i,i+1+kn} - \mathcal{F}_{i+1,i+kn}
\end{gather*}
for all $i\in\mathbb{Z}$.
\end{prop}

\begin{proof} 
Let $Bd$ be the boundary corresponding to $\mathcal{F}$, with
$n$ marked points. 
Let $i\in \mathbb{Z}$ and
suppose that $\ga_1$ is the boundary edge from $i$ to $i+1$ (taken modulo $n$).
Since $x(\ga_1)=1$, 
due to (\ref{eqn:m_1_growth_coefficient}), we have
\[ x(Brac_{k}) = x(\gamma_{k+1}) - x(\gamma_{k}^C).\]
Observe that 
$\ga_{k+1}$ 
corresponds to the entry at position $(i,i+1+kn)$ and $\ga_{k}^C$ corresponds to the entry at position $(i+1,i+kn)$. 
Hence \[ x(\gamma_{k+1}) - x(\gamma_{k}^C) 
 = \mathcal{F}_{i,i+1+kn} - \mathcal{F}_{i+1,i+kn}.\]
\end{proof}

\begin{defn}  \label{defn:growthcoefficient} 
Let $\mathcal{F}=\{\mathcal{F}_{i,j}\}$
be an infinite periodic frieze pattern as described in Section \ref{sec:proof_frieze_laurent}. Let $n$ be the number of marked points on the outer boundary of the associated triangulated surface. 
For $k \geq 0$, the \emph{$k$th growth coefficient} for $\mathcal{F}$
is given by $s_0:=2$, and $s_k:= \mathcal{F}_{i,i+1+kn} - \mathcal{F}_{i+1,i+kn}$, otherwise. 
\end{defn}

Note that $s_k$ measures the difference between entries in the first row of the $(k+1)$st level and the penultimate row of the $k$th level. 

\begin{rem}\label{rem:sk_is_brac}
Per Proposition \ref{prop:see_BFPT16_thm2_2}, $s_k=x(\Brac_k)$ whenever $k\geq 1$, so we can use the two terms interchangeably.
\end{rem}

To see that $s_0 =2$ makes sense in the frieze, see Fig.~\ref{fig:growth_frieze}. We write in the row of 0s and then add a row of -1s above the row of 0s. Then $s_0 = 1 - (-1) =2$.

\begin{figure}[h!]
\begin{center}
\begin{tikzpicture}[scale = .79]
 \matrix(m) [matrix of math nodes,row sep={1.4em,between origins},column sep={1.35em,between origins},nodes in empty cells]{
 -1&&-1&&-1&&-1&&-1&&-1&&&&&&&&&&&\\
&\mathbf{0}&&\mathbf{0}&&\mathbf{0}&&\mathbf{0}&&\mathbf{0}&&\mathbf{0}&&&&&&&&&&&\\
&&1&&1&&1&&1&&1&&1&&&&&&&&&&\\
&&&1&&2&&6&&1&&2&&6&&&&&&&&&\\
&&&&\mathbf{1}&&\mathbf{11}&&\mathbf{5}&&\mathbf{1}&&\mathbf{11}&&\mathbf{5}&&&&&&&&\\
&&&&&5&&9&&4&&5&&9&&4&&&&&&&\\
&&&&&&4&&7&&19&&4&&7&&19&&&&&&\\
&&&&&&&\mathbf{3}&&\mathbf{33}&&\mathbf{15}&&\mathbf{3}&&\mathbf{33}&&\mathbf{15}&&&&&\\
&&&&&&&&14&&26&&11&&14&&26&&11&&&&\\
&&&&&&&&&11&&19&&51&&11&&19&&51&&&\\
&&&&&&&&&&\mathbf{8}&&\mathbf{88}&&\mathbf{40}&&\mathbf{8}&&\mathbf{88}&&\mathbf{40}&&\\
&&&&&&&&&&&37&&69&&29&&37&&69&&29&\\
&&&&&&&&&&&&29&&50&&134&&29&&50&&134\\
&&&&&&&&&&&&&&&\node[rotate=-6.5,shift={(-0.034cm,-0.08cm)}]  {\ddots};&&&&&&\node[rotate=-6.5,shift={(-0.034cm,-0.08cm)}]  {\ddots};\\
};

\draw[opacity=0,rounded corners,fill=myblue,fill opacity=0.15] (m-1-1.south west) -- (m-13-13.south west) -- (m-13-19.south west) -- (m-1-7.south west) -- cycle;

\draw ($(m-3-3)+(0,0.6125cm)$) node[ellipse, minimum height=2cm,minimum width=0.65cm,draw,semithick,opacity=0.5] {};
\draw ($(m-6-6)+(0,0.6125cm)$) node[ellipse, minimum height=2cm,minimum width=0.65cm,draw,semithick,opacity=0.5] {};
\draw ($(m-9-9)+(0,0.6125cm)$) node[ellipse, minimum height=2cm,minimum width=0.65cm,draw,semithick,opacity=0.5] {};
\draw ($(m-12-12)+(0,0.6125cm)$) node[ellipse, minimum height=2cm,minimum width=0.65cm,draw,semithick,opacity=0.5] {};

\draw (m-2-1) node[shift={(-1.5cm,0cm)}]{$s_0=2$};
\draw (m-5-4) node[shift={(-1.5cm,0cm)}]{$s_1=3$};
\draw (m-8-7) node[shift={(-1.5cm,0cm)}]{$s_2=7$};
\draw (m-11-10) node[shift={(-1.5cm,0cm)}]{$s_3=18$};


\end{tikzpicture}
\end{center}
\caption{Growth coefficients in the frieze with quiddity sequence $(1,2,6)$.}
\label{fig:growth_frieze}
\end{figure}
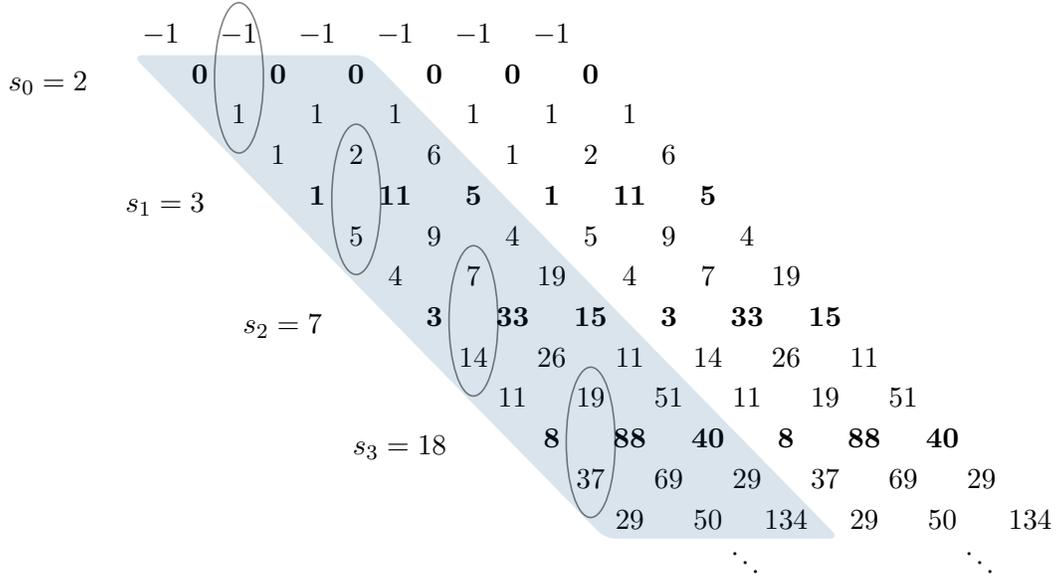

Given a triangulation of an annulus, we get two different quiddity sequences $q$ and $\overline{q}$ from the outer and inner boundaries, respectively. 
We see that $s_q = s_{\overline{q}}$ since $\Brac_k$ is defined independently of the choice of the boundary of an annulus.
This agrees with \cite[Thm. 3.4]{BFPT16}. 

\begin{rem}
\label{rem:compute_entires_with_progression_formulas}
The progression formulas (Theorem \ref{thm:progression_formula}) give us a way to compute entries on lower levels using the growth coefficients and entries on previous levels. Define $m(\gamma)$ to be the integer obtained from $x(\gamma)$ by specializing all the $x_{\tau_i}$ to $1$.
We demonstrate (\ref{eqn:m_1}) on the
frieze pattern of positive integers in Fig.~\ref{fig:reflection_symmetry}.
Consider the underlined entry 5 (in the dotted circle) on the 
 first column, which corresponds to the $\gamma_2$ for some boundary edge $\gamma_1$. This 5 is equal to $s_1  \, m(\gamma_1) + m(\gamma_1^C) = 3(1) + 2$. Similarly, we compute the underlined entry $19 =: m(\gamma_3)$ which is equal to $s_2 \, m(\gamma_1)+m(\gamma_2^C) = 7(2)+5$. We can do this for every entry in a frieze pattern.
\end{rem}

\begin{figure}[h]
\begin{tikzpicture}[scale = .79]
 \matrix(m) [matrix of math nodes,row sep={1.3em,between origins},column sep={1.25em,between origins},nodes in empty cells]{
 &&&&&&&&&&&&&&&&&&&&&\\
&\mathbf{0}&&\mathbf{0}&&\mathbf{0}&&\mathbf{0}&&\mathbf{0}&&\mathbf{0}&&\mathbf{0}&&\mathbf{0}&&&&&&&\\
&&1&&1&&1&&1&&1&&1&&1&&1&&&&&&\\
&&&1&&2&&6&&1&&2&&6&&1&&2&&&&&\\
&&&&\mathbf{1}&&\mathbf{11}&&\mathbf{5}&&\mathbf{1}&&\mathbf{11}&&\mathbf{5}&&\mathbf{1}&&\mathbf{11}&&&&\\
&&&&&\underline{5}&&9&&4&&5&&9&&4&&5&&9&&&\\
&&&&&&4&&7&&19&&4&&7&&19&&4&&7&&&&\\
&&&&&&&\mathbf{3}&&\mathbf{33}&&\mathbf{15}&&\mathbf{3}&&\mathbf{33}&&\mathbf{15}&&\mathbf{3}&&\mathbf{33}&&&&\\
&&&&&&&&14&&26&&11&&14&&26&&11&&14&&26&&&&\\
&&&&&&&&&11&&19&&51&&11&&\underline{19}&&51&&11&&19&&\\
&&&&&&&&&&\mathbf{8}&&\mathbf{88}&&\mathbf{40}&&\mathbf{8}&&\mathbf{88}&&\mathbf{40}&&\mathbf{8}&&\mathbf{88}&&&\\
&&&&&&&&&&&37&&69&&29&&37&&69&&29&&37&&&\\
&&&&&&&&&&&&29&&50&&134&&29&&50&&134&&29&&\\
&&&&&&&&&&&&&&&\node[rotate=-6.5,shift={(-0.034cm,-0.08cm)}]  {\ddots};&&&&&&\node[rotate=-6.5,shift={(-0.034cm,-0.08cm)}]  {\ddots};\\
};


\draw ($(m-6-6)+(0,0.6125cm)$) node[ellipse, minimum height=2cm,minimum width=0.7cm,draw,semithick,opacity=0.5] {};
\draw[densely dotted, fill=mygreen,fill opacity=0.25] (m-6-6) circle (0.32cm);
\draw[densely dotted, fill=mygreen,fill opacity=0.25] (m-4-6) circle (0.32cm);
\draw[densely dotted, fill=mygreen,fill opacity=0.25] (m-3-3) circle (0.32cm);

\draw ($(m-9-18)+(0,0.8125cm)$) node[ellipse, minimum height=3.5cm,minimum width=0.75cm,draw,semithick,opacity=0.5] {};
\draw[dashed, fill=myred,fill opacity=0.07] (m-6-18) circle (0.31cm);
\draw[dashed,fill=myred,fill opacity=0.07] (m-10-18) circle (0.31cm);
\draw[dashed,fill=myred,fill opacity=0.07] (m-4-12) circle (0.31cm);

\draw (m-2-1) node[shift={(-1.5cm,0cm)}]{$s_0=2$};
\draw (m-5-4) node[shift={(-1.5cm,0cm)}]{$s_1=3$};
\draw[densely dotted, fill=mygreen, fill opacity=0.25] (m-5-2)+(-0.0cm,0) circle (0.34cm);
\draw (m-8-7) node[shift={(-1.5cm,0cm)}]{$s_2=7$};
\draw[dashed, fill=myred, fill opacity=0.07] (m-8-5)+(-0.0cm,0) circle (0.34cm);
\draw (m-11-10) node[shift={(-1.5cm,0cm)}]{$s_3=18$};
\end{tikzpicture}
\caption{Computing entries in an infinite frieze pattern with quiddity sequence $(1,2,6)$.}
\label{fig:reflection_symmetry}
\end{figure}
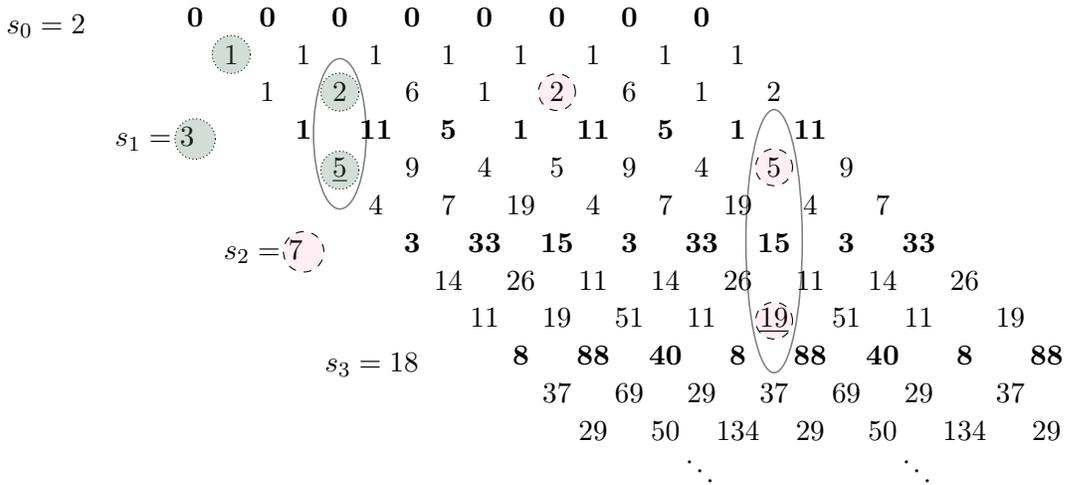

\subsection{Chebyshev polynomials}
\label{subsec:chebyshev}
We now recall some basic facts about Chebyshev polynomials~\cite[Sec. 2.5]{MSW13}.

\begin{defn}[{\cite[Def. 2.33]{MSW13}}]
\label{defn:normalized_chebyshev}
Let $T_k$ denote the $k$-th normalized Chebyshev polynomial defined by
\[
T_k\left(t+\frac{1}{t}\right) = t^k + \frac{1}{t^k} 
\]
\end{defn}

\begin{prop}[{\cite[Prop. 2.34]{MSW13}}]
\label{prop:MSW_chebyshev_recurrence}
The normalized Chebyshev polynomials $T_k(x)$ defined above can also be uniquely determined by the initial conditions $T_0(x) = 2$, $T_1(x) = x$, and the recurrence
\[
T_k(x) = x T_{k-1}(x) - T_{k-2}(x).
\]
Note that $T_k(x)$'s can also be written as $2\Cheb_k(x/2)$, where $\Cheb_k(x)$ denotes the usual Chebyshev polynomial of the first kind, which satisfies $\Cheb_k(\cos x)=\cos (kx)$.
\end{prop}

Table \ref{tabu:Cheb_poly} shows the first few normalized Chebyshev polynomials.  
The elements associated to the bracelets (Definitions~\ref{defn:bracelets} and \ref{defn:bracelet_laurent_polynomial}) satisfy the normalized Chebyshev polynomials. 

\begin{table}
\begin{center}
$\begin{tabular}{l c l}
  $T_0(x)$ &$=$ &$2$  \\
  $T_1(x)$ &$=$ &$x$ \\
  $T_2(x)$ &$=$& $x^2 -2$ \\
  $T_3(x)$ &$=$ &$x^3 - 3x$  \\
  $T_4(x)$ &$=$ &$x^4 - 4x^2 + 2$  \\
  $T_5(x)$ &$=$ &$x^5 - 5x^3 + 5x^2$  \\
  $T_6(x)$ &$=$ &$x^6 - 6x^4 + 9x^2 - 2$  \\
\end{tabular}$
\end{center}
\caption{The normalized Chebyshev polynomials $T_k(x)$ for small $k$.}
\label{tabu:Cheb_poly}
\end{table}

\begin{prop}[{\cite[Prop. 4.2]{MSW13}}]
\label{prop:MSW13_prop42}
We have \[x_{\Brac_k} = T_k(x_{\Brac_1}).\]
\end{prop}

The recurrence implied by Propositions \ref{prop:MSW_chebyshev_recurrence} and \ref{prop:MSW13_prop42} 
agrees with that of growth coefficients, $s_{k+2}=s_1s_{k+1}-s_k$ for $k\ge 0$, from
{\cite[Prop. 2.10]{BFPT16}}.

\begin{rem}
\label{rem:brac_equals_2}
Suppose $\Brac_1$ is a closed loop without self-crossings enclosing a single puncture.
Since $x_{\Brac_1}=2$ (per Definition~\ref{defn:bracelet_laurent_polynomial}(\ref{defn:bracelet_laurent_polynomial:item:brac1_equals_2})),
 $T_k(2)=2$ for all $k\geq 0$ (per Proposition \ref{prop:MSW_chebyshev_recurrence}), and
$x_{\Brac_k} = T_k(x_{\Brac_1})$ (per Proposition \ref{prop:MSW13_prop42}), we have 
\begin{gather*}
x_{\Brac_k} = 2 \text{ for all $k\geq 0$}.
\end{gather*}
In particular, if $(S,M)$ is a once-punctured \disk, all bracelets are associated to $2$, thought of as a scalar in the cluster algebra 
$\mathcal{A} = \mathcal{A}(S,M)$.

\end{rem}

\section{Recursive relationships}
\label{sec:recursive}

\subsection{Differences from complement symmetry}
\label{subsec:diff_comp}

We consider the difference between frieze entries associated to complementary arcs of two marked points. For the once-punctured {\disk}, this difference is constant across all levels and is determined by the two marked points. For the annulus, this difference is determined by the level as well as the end points. Recall from Definition \ref{defn:growthcoefficient}, that $s_k := {\mathcal{F}}_{i,i+1+kn} - {\mathcal{F}}_{i+1, i+kn}$ is the $k$th growth coefficient of a frieze $\mathcal{F}=\{{\mathcal{F}}_{i,j}\}_{i\leq j}$.

\begin{prop}\label{prop:complementary_diff}
Let $\mathcal F$ be a frieze pattern coming from a triangulation of a once-punctured {\disk} or annulus.
Let $\gamma_1=\gamma$ be an ordinary arc from $i$ to $j$ (possibly $i=j$) or a boundary edge from $i$ to $i+1$ (i.e. $\gamma$ is the generalized peripheral arc $ \gamma_1(i,j)$ as defined in Sec. \ref{subsec:comp_arcs}). 
Define $c_{k,\gamma} :=  x\left(\gamma_k\right) - x\left(\gamma^C_k \right)$. 
We write $c_k:=c_{k,\gamma}$, since $\gamma$ is understood. Then we have the following relations for $k > 1$:

\smallskip

\begin{inparaenum}[$(1)$]
\item\label{prop:relationvaluesA1}$c_k = (s_{k-1}-s_{k-2}) c_1 + c_{k-2}$, \quad where we define $c_0 = c_1$;

\item 
\label{prop:relationvaluesA2}
\begin{gather*}
c_k = 
\begin{cases}
\displaystyle{
c_1 \left( 1 +  \sum_{i=0}^{k-1}(-1)^{i+1}s_i  \right)}
\text{ for $k$ even, and}\\
\displaystyle{
c_1 \left( 1 +  \sum_{i=1}^{k-1}(-1)^{i}s_i  \right)}
\text{ for $k$ odd,}
\end{cases}
\end{gather*}

\end{inparaenum}

where $c_1 =  x(\gamma_1) - x(\gamma^C_1) $ is computed from the triangulation or the frieze.
\end{prop}

Note that in the case of the once-punctured \disk, since $s_k = 2$ for all $k$, the formula reduces to $c_k = c_1$ for all $k > 1$. 
Note also that, if $i=j$, then $c_1=x(\gamma_1)$.

\begin{proof}
First, we show that $c_2=(s_1-s_0) c_1 + c_0$.
Per (\ref{eqn:m_1}), we have $\gamma_2 = s_{1} \, x(\gamma_1) + x(\gamma^C_{1})$.
We also have $\gamma^C_2 = s_{1} \, x(\gamma^C_1) + x(\gamma_{1})$ due to (\ref{eqn:m_1}) if $i\neq j$ and due to the fact that $\gamma^C_2=\gamma_1$ and $x(\gamma^C_1)=0$ if $i=j$.
Subtracting the two equations gives us
\begin{align*}
c_2 = x(\gamma_2) - x(\gamma^C_2) &= s_1 \left(x(\gamma_1) - x(\gamma^C_1) \right) - \left( x(\gamma_1) - x(\gamma^C_1)\right) \\
& = s_1 c_1 - c_1 \\
& = (s_1 - 2) c_1 + c_1\\
& =(s_1 - s_0) c_1 + c_0,
\end{align*}
where the last equality is due to the fact that $s_0=2$ and $c_0=c_1$.

If $k\geq 3$, per (\ref{eqn:m_1}), we have
$\gamma_k = s_{k-1} \, x(\gamma_1) + x(\gamma^C_{k-1})$ and $x(\gamma^C_{k-1}) = s_{k-2} \, x(\gamma^C_1) + x(\gamma_{k-2})$. Note that the second equation holds even for the case where $i=j$, due to the fact that $\gamma^C_{k-1}=\gamma_{k-2}$ and $x(\gamma^C_1)=0$ if $i=j$. 
So $x(\gamma_k) = s_{k-1} \, x(\gamma_1) + s_{k-2}\, x(\gamma^C_1) + x(\gamma_{k-2})$. We get a similar equation $x(\gamma^C_k) = s_{k-1} \, x(\gamma^C_1) + s_{k-2} \, x(\gamma_1) + x(\gamma^C_{k-2})$. Subtracting the two gives us (\ref{prop:relationvaluesA1}). 

Part (\ref{prop:relationvaluesA2}) is proved by induction. For $k =1$, we get that $c_1 = c_1$, and for $k =2$, $c_2 =  (s_1 -s_0) \, c_1 + c_0 = (s_1 -s_0) \, c_1 + c_1= c_1 \, \left( (-1)^2s_1 + (-1)^1s_0 + 1 \right) = c_1 \left( \sum_{i=0}^{1}(-1)^{i+1}s_i + 1 \right)$ by (\ref{prop:relationvaluesA1}). Assume (\ref{prop:relationvaluesA2}) holds for some $n$. If $n$ is even, then $n+1$ is odd, and 
\begin{align*}
c_{n+1} &= (s_{n} - s_{n-1}) c_1 + c_{n-1} \text{ by (\ref{prop:relationvaluesA1})}\\
&= c_1 \left( (-1)^{n} \, s_{n} + (-1)^{n-1} s_{n-1} \right) + c_1 \left( \sum_{i=1}^{n-2}(-1)^{i}s_i + 1 \right) \\
&=  c_1 \left(1 +  \sum_{i=1}^{n}(-1)^{i}s_i  \right).
\end{align*} 

Similarly, if $n$ is odd, then $n+1$ is even, and we have 
\begin{align*}
c_{n+1} &= (s_{n} - s_{n-1})c_1 + c_{n-1} \text{ by (\ref{prop:relationvaluesA1})}\\
&= c_1 \left( (-1)^{n+1}s_{n} + (-1)^{n} s_{n-1} \right) + c_1 \left( \sum_{i=0}^{n-2}(-1)^{i+1}s_i + 1 \right)\\ 
&=  c_1 \left( \sum_{i=0}^{n}(-1)^{i+1}s_i + 1 \right).
\end{align*}
Thus the result holds true for all $n$. 
\end{proof}

\subsection{Arithmetic progressions}
\label{subsec:arithmetic_progressions}

Friezes coming from triangulations of once-punctured {\disks} satisfy a beautiful arithmetic property. Consider the frieze ${\mathcal F}_T$ in Fig.~\ref{fig:arithmeticIntegerFrieze},
where $T$ is a triangulation of $D_5$; when jumping $5$ steps along any diagonal in ${\mathcal F}_T$, we get a sequence of numbers that with a \emph{common difference}. Thus these sequences of numbers form an increasing arithmetic progression (see Def. 3.10 of \cite{Tsc15}). We call such friezes \emph{n-arithmetic} when the common difference jumps every $n$ steps.  The 
dotted
and 
dashed
circles in Fig.~\ref{fig:arithmeticIntegerFrieze} show examples of two 5-arithmetic progressions in the frieze.

\begin{prop}[{\cite[Prop. 3.11]{Tsc15}}]
\label{thm:arithmetic_Tsc15}
Every $n$-periodic infinite frieze $\F_T$ associated to a triangulation $T$ of $D_n$ is $n$-arithmetic.
\end{prop}

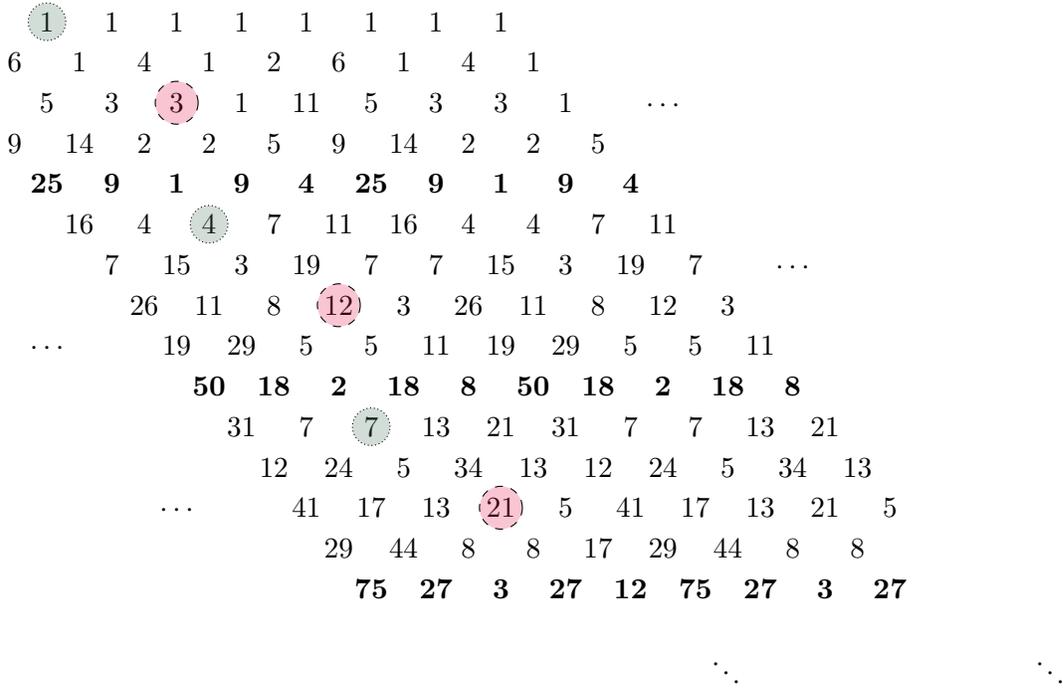
\begin{figure}[h!]
\centering
\begin{tikzpicture}[scale=.5][font=\normalsize] 
 \matrix(m) [matrix of math nodes,row sep={1.4em,between origins},column sep={1.12em,between origins},nodes in empty cells]{
&&&&&&&&&&&&&&&&&&&&&&&&&&&&&&&&&&&&&\\
&1&&1&&1&&1&&1&&1&&1&&1&&&&&&&&&&&&&\\
6&&1&&4&&1&&2&&6&&1&&4&&1&&&&&&&&&&&&\\
&5&&3&&3&&1&&11&&5&&3&&3&&1&&&\node{\cdots};&&&\\
9&&14&&2&&2&&5&&9&&14&&2&&2&&5&&&&&&&&&&\\
&\textbf{25}&&\textbf{9}&&\textbf{1}&&\textbf{9}&&\textbf{4}&&\textbf{25}&&\textbf{9}&&\textbf{1}&&\textbf{9}&&\textbf{4}&&&&&&&&\\
&&16&&4&&4&&7&&11&&16&&4&&4&&7&&11&&&&&&&&\\
&&&7&&15&&3&&19&&7&&7&&15&&3&&19&&7&&&\node{\cdots};&\\
&&&&26&&11&&8&&12&&3&&26&&11&&8&&12&&3&&&&&\\
&\node{\cdots};&&&&19&&29&&5&&5&&11&&19&&29&&5&&5&&11&&&&&\\
&&&&&&\textbf{50}&&\textbf{18}&&\textbf{2}&&\textbf{18}&&\textbf{8}&&\textbf{50}&&\textbf{18}&&\textbf{2}&&\textbf{18}&&\textbf{8}&&&&\\
&&&&&&&31&&7&&7&&13&&21&&31&&7&&7&&13&&21&&&\\
&&&&&&&&12&&24&&5&&34&&13&&12&&24&&5&&34&&13&\\
&&&&&\node{\cdots};&&&&41&&17&&13&&21&&5&&41&&17&&13&&21&&5&\\
&&&&&&&&&&29&&44&&8&&8&&17&&29&&44&&8&&8&&\\
&&&&&&&&&&&\textbf{75}&&\textbf{27}&&\textbf{3}&&\textbf{27}&&\textbf{12}&&\textbf{75}&&\textbf{27}&&\textbf{3}&&\textbf{27}&\\
&&&&&&&&&&&&&&&&&&&&&&&&&&&&&&&&&&&&&\\
&&&&&&&&&&&&&&&&&&&&&&\node[rotate=-6.5,shift={(-0.034cm,-0.08cm)}]  {\ddots};&&&&&&&&&&\node[rotate=-6.5,shift={(-0.034cm,-0.08cm)}]  {\ddots};&&&&&\\
};






\draw[dashed, fill=myred, fill opacity=0.25] (m-4-6) circle (0.57cm);
\draw[dashed, fill=myred, fill opacity=0.25] (m-9-11) circle (0.57cm);
\draw[dashed, fill=myred, fill opacity=0.25] (m-14-16) circle (0.57cm);

\draw[densely dotted, fill=mygreen, fill opacity=0.25] (m-2-2) circle (0.5cm);
\draw[densely dotted, fill=mygreen, fill opacity=0.25] (m-7-7) circle (0.5cm);
\draw[densely dotted, fill=mygreen, fill opacity=0.25]  (m-12-12) circle (0.5cm);
\end{tikzpicture}
\caption{Arithmetic progression in a frieze. 
The dotted and dashed circles
show two different arithmetic progressions. 
The dotted circles have a common difference of $3$, 
the dashed circles of $9$.}
\label{fig:arithmeticIntegerFrieze}
\end{figure}

The following is a geometric explanation for Proposition \ref{thm:arithmetic_Tsc15}.

\begin{prop}
[{Corollary of Theorem \ref{thm:progression_formula}}]
\label{prop:arithmetic_progression}
Suppose $(S,M)$ is a once-punctured \disk. 
Let $\gamma_1=\gamma$ be an ordinary arc from $i$ to $j$ (possibly $i=j$) or a boundary edge from $i$ to $i+1$.  (Again, $\gamma$ is the generalized peripheral arc $ \gamma_1(i,j)$ as defined in Sec. \ref{subsec:comp_arcs}.)
Then, for $k \geq 2$, we have
\[
\displaystyle{
x(\gamma_k) = x(\gamma_{k-1}) + \left(~ x(\gamma_1) + x(\gamma_1^C) ~\right)
}
\]
\end{prop}

\begin{proof}[Proof of Proposition \ref{prop:arithmetic_progression}]
We prove this by induction on $k$.

First, we have 
\begin{align*}
x(\gamma_2) &= x(\Brac_1) ~ x(\gamma_1) + x(\gamma_1^C)
\text{ by (\ref{eqn:m_1})} \\
&=2~x(\gamma_1) + x(\gamma_1^C) \text{ by Remark \ref{rem:brac_equals_2}}
\\
&=x(\gamma_1) + \left(~ x(\gamma_1) + x(\gamma_1^C)  ~\right).
\end{align*}
Similarly, $x(\gamma_2^C)=x(\gamma_1^C) + \left(~ x(\gamma_1) + x(\gamma_1^C) ~\right)$.  (Note that, if $i=j$, we have $x(\gamma_2^C) = x(\gamma_1)$ and $x(\gamma_1^C) = 0$, and so these expressions for $x(\gamma_2)$ and $x(\gamma_2^C)$ are still valid in this case.)

Next, suppose by induction we have the equality
$x(\gamma_k)=x(\gamma_{k-1}) + \left(~ x(\gamma_1) + x(\gamma_1^C) ~\right)$
as well as 
$x(\gamma_k^C)=x(\gamma_{k-1}^C) + \left(~ x(\gamma_1) + x(\gamma_1^C) ~\right)$, where again we let $x(\gamma_1^C) = 0$ if $i=j$.
Then
\begin{align*}
x(\gamma_{k+1}) &= 2~x(\gamma_1) + x(\gamma_k^C) \text{ by (\ref{eqn:m_1})
and Remark  \ref{rem:brac_equals_2}}\\
&= 2~x(\gamma_1) + x(\gamma_{k-1}^C) + \left(~  x(\gamma_1) + x(\gamma_1^C) ~\right) \text{  by the inductive hypothesis}
\\
&= \hspace{2.6em} x(\gamma_k) \hspace{2.6em} + \left(~ x(\gamma_1) + x(\gamma_1^C) ~\right) 
\text{ by (\ref{eqn:m_1}) 
and Remark \ref{rem:brac_equals_2}.}
\end{align*}
\end{proof}

\newcommand\bcifill{gray!70}

\appendix
\section{A bijection between BCI tuples and T-paths}
\label{appendix:BCI_Tpath}

It is known that a combinatorial $T$-path formula can be used to compute the Laurent expansion for an arc in a polygon \cite{Sch08Tpath} (see also \cite{ST09,GM15} for the case of a general surface).
On the other hand, every entry in a finite (respectively, infinite) integral frieze counts the number of matchings, $m_{ij}$, between vertices of a polygon (respectively, once-punctured {\disk} or annulus) and triangles of the associated triangulation \cite[Section 2]{BCI74} (respectively, \cite[Section 4.5]{Tsc15} and \cite[Section 5]{BPT16}). 
We give a bijection between these classical matching tuples and the more recent $T$-paths.
In the case of polygons, such a bijection was constructed by Carroll and Price \cite{CP03,Pro05} in unpublished work. This gives a combinatorial formula (Corollary \ref{rem:bci_formula})
for computing Laurent expansion formulas for cluster variables via these matching tuples in the case of a once-punctured {\disk} or an annulus.

Throughout this section, let $T$ be a triangulation of an annulus or a once-punctured {\disk} $(S,M)$.
Let $\gamma$ be a generalized peripheral arc (allowing self-crossings) on a boundary component $Bd$ of $(S,M)$. 

Choose the usual orientation so that $Bd$ is to the right of $\gamma$.
For the purpose of computing the Laurent polynomial expansion of $x(\gamma)$,
we will work with a finite polygon cover $\widetilde{S}_\gamma$ of $T$ containing a lift of $\gamma$.
By abuse of notation, we also denote the lift of $\gamma$ in $\widetilde{S}_\gamma$  by $\gamma$.
Let $s$ and $t$ be the starting and ending points of $\gamma$, respectively.
Let $R_1$, $R_2$, $\dots$, $R_r$ be the boundary vertices to the right of $\gamma$, not including $s$ and $t$, and let $L_1$, $\dots$, $L_l$ be the boundary vertices to the left of $\gamma$.
These vertices are ordered so that $s$, $R_1$, $R_2$, $\dots$, $R_r$, $t$, $L_l$, $\dots$, $L_2$, $L_1$ go counterclockwise around the polygon cover $\widetilde{S}_\gamma$. 
See Figures
\ref{fig:Sgamma_rectangle} and
\ref{fig:example_finite_cover_strip}.
We say that the $R_i$ are the right vertices and the $L_i$ are the left vertices.
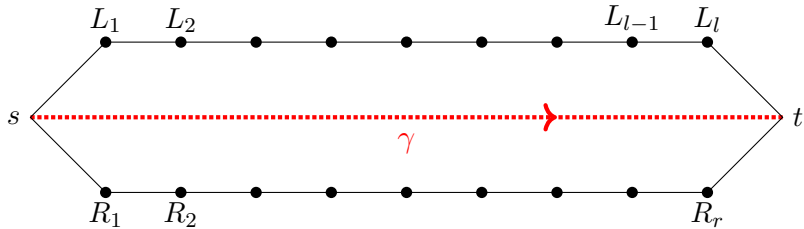
\begin{figure}[h!]
\begin{center}
\begin{tikzpicture}

\foreach \t in {-4,-3,-2,-1,0,1,2,3,4}{
	\fill (\t, -1) circle (0.075);
	\fill (\t, 1) circle (0.075);
}

\draw (-4,-1) node [below] {$R_1$};
\draw (-4,1) node [above] {$L_1$};
\draw (-3,-1) node [below] {$R_2$};
\draw (-3,1) node [above] {$L_2$};
\draw (3,1) node [above] {$L_{l-1}$};
\draw (4,-1) node [below] {$R_{r}$};
\draw (4,1) node [above] {$L_{l}$};
\draw(-5,0) node [left] {$s$};
\draw (5,0) node [right] {$t$};

\draw[ultra thick, densely dotted, red, ->] (-5,0) -- (2,0);
\draw[ultra thick, densely dotted, red] (5,0) -- (2,0);
\draw[red] (0,-.35) node[] {\large{$\gamma$}};

\draw (-4,1) -- (4,1);
\draw (-4,-1) -- (4,-1);
\draw (-4,-1) -- (-5,0)  -- (-4,1);
\draw (4,-1)  -- (5,0)  -- (4,1);

\end{tikzpicture}
\caption{A generic polygon cover for a generalized arc $\gamma$.}
\label{fig:Sgamma_rectangle}
\end{center}
\end{figure}
\subsection{BCI tuples and $T$-paths}

\newcommand\bcituple{\text{B}}
\newcommand\bcitrail{\text{TR }}

\begin{defn}[{\cite[Section 2]{BCI74}}]
\label{defn:bcituple}
A \emph{BCI tuple 
for $\gamma$} is an $r$-tuple $(t_1,\dots,t_r)$ such that:
\begin{enumerate}[(\bcituple 1)]
\item \label{bci:1}
the $i$-th entry $t_i$ is a triangle of $\widetilde{S_\gamma}$ having $R_i$ as a vertex. (We say that the vertex $R_i$ is matched to the triangle in the $i$-th entry of the tuple, and write $\triangle(R_i):=t_i$).
 \item
the entries are pairwise distinct.
\end{enumerate}
\end{defn}

Let $\Delta_0$, $\dots$, $\Delta_d$ be the triangles of $\widetilde{S}_\gamma$ which are crossed by $\gamma$, in order.

\begin{ex}
\label{ex:11_bci_tuples}
The following are the BCI tuples for the generalized arc $\gamma$ from Fig.~\ref{fig:example_bcituples}.
Note that $A,B,C,D,E$, and $F$ are the  triangles which are \emph{not} crossed by $\gamma$ but which are adjacent to at least one of the vertices $R_1,\dots,R_8$ (located to the right of $\gamma$).
\begin{center}
\begin{multicols}{2}
\begin{enumerate}[i.]
\item {$\left(\Delta_3, A, B, C, D, \Delta_4, E, F\right)$} 
\item {$\left(\Delta_2, A, B, C, D, \Delta_5, E, F\right)$} 
\item {$\left(\Delta_2, A, B, C, D, \Delta_4, E, F\right)$} 
\item {$\left(\Delta_2, A, B, C, D, \Delta_3, E, F\right)$} 
\item {$\left(\Delta_1, A, B, C, D, \Delta_5, E, F\right)$} 
\item {$\left(\Delta_1, A, B, C, D, \Delta_4, E, F\right)$} 
\item {$\left(\Delta_1, A, B, C, D, \Delta_3, E, F\right)$} 
\item {$\left(\Delta_0, A, B, C, D, \Delta_5, E, F\right)$} 
\item {$\left(\Delta_0, A, B, C, D, \Delta_4, E, F\right)$} 
\item {$\left(\Delta_0, A, B, C, D, \Delta_3, E, F\right)$} 
\item {$\left(\Delta_3, A, B, C, D, \Delta_5, E, F\right)$} 
\end{enumerate}
\end{multicols}
\end{center}
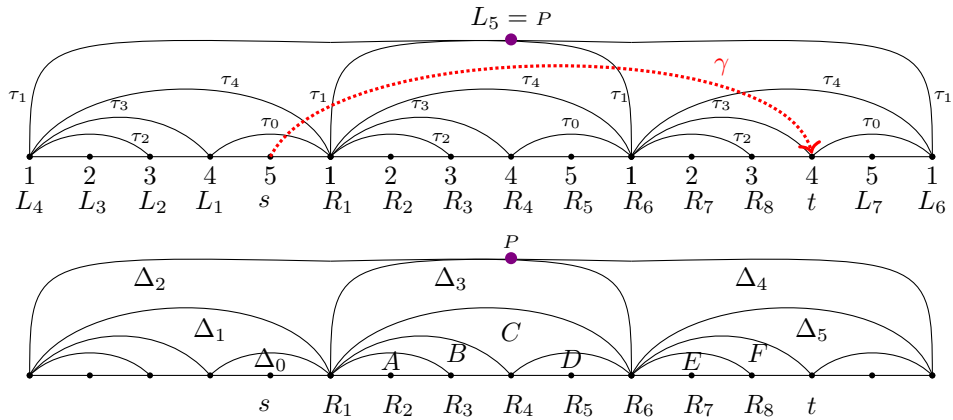
\begin{figure}[h!]
\begin{tikzpicture}[scale = .4]
		\tikzstyle{every node} = [font = \small]
		
		\foreach \x in {0,10}
        {
        	\foreach \y in {-8}
            {
            \PruferShort{\x-6}{\y-3}{10}{}
            }
        }
        
        \PruferShortCurvedLeft{20-6}{-8-3}{10}{}
        
        \foreach \x in {0,10,20}
		{
			\foreach \y in {-8}
			{
				\draw[] (\x-6,\y-3) -- (\x+4,\y-3);
			
				\foreach \t in {-6,-4,-2,0,2,4}
				{
					\fill (\x+\t,\y-3) circle (.1);
				}

				\fill (\x-6,\y-3) node [below] {$1$};
				\fill (\x-4,\y-3) node [below] {$2$};
				\fill (\x-2,\y-3) node [below] {$3$};
				\fill (\x+0,\y-3) node [below] {$4$};
				\fill (\x+2,\y-3) node [below] {$5$};
				\fill (\x+4,\y-3) node [below] {$1$};

				\fill (\x-5.7,\y-1) node [left] {\tiny{$\tau_1$}};
				\draw[] (\x-6,\y-3) .. controls (\x-5,\y-2) and (\x-3,\y-2) .. (\x-2,\y-3);
				\fill (\x-2.3,\y-2.9) node [above] {\tiny{$\tau_2$}};
                \draw[] (\x-6,\y-3) .. controls (\x-4.5,\y) and (\x+3,\y) .. (\x+4,\y-3);
                \fill (\x-0,\y-0.5) node [right] {\tiny{$\tau_4$}};
				\draw[] (\x-6,\y-3) .. controls (\x-4.5,\y-1.25) and (\x-1.5,\y-1.25) .. (\x-0,\y-3);
				\fill (\x-3,\y-1.8) node [above] {\tiny{$\tau_3$}};

				\draw[] (\x+0,\y-3) .. controls (\x+1,\y-2) and (\x+3,\y-2) .. (\x+4,\y-3);
				\fill (\x+2,\y-2.4) node [above] {\tiny{$\tau_0$}};	
			}
		}	
        \PruferShortCurvedLeft{20+4}{-8-3}{10}{}

        \fill (20+3.7,-8-1) node [right,black] {\tiny{$\tau_1$}};
        
        \foreach \x in {0}
		{
			\foreach \y in {-8}
			{
				\draw[red,densely dotted,->,very thick] (\x+2,\y-3) .. controls (\x+4,\y+1) and (\x+19,\y+1) .. (\x+20,\y-2.9);
              \fill (\x+17,\y) node[red]{$\gamma$};

                }
         }
         
         \fill[violet] (0+10,-8+0.9) circle (.2);
         \fill (0+10,-8+0.9) node [above] {$L_5=\puncture$};
         
\foreach \x in {1,2,3,4,5,6,7,8}
{
\fill (2*\x+2.25,-8-4.5) node {$R_\x$};
}
\fill (0+1.8,-8-4.5) node {$s$};
\fill (0,-8-4.5) node {$L_1$};
\fill (0-1.9,-8-4.5) node {$L_2$};
\fill (0-3.9,-8-4.5) node {$L_3$};
\fill (-6,-8-4.5) node {$L_4$};
\fill (20,-8-4.5) node {$t$};
\fill (24,-8-4.5) node {$L_6$};
\fill (24-2.1,-8-4.5) node {$L_7$};
         
\end{tikzpicture}

\begin{tikzpicture}[scale = .4]
		\tikzstyle{every node} = [font = \small]
		
		\foreach \x in {0,10}
        {
        	\foreach \y in {-8}
            {
            \PruferShort{\x-6}{\y-3}{10}{}
            }
        }
        
        \PruferShortCurvedLeft{20-6}{-8-3}{10}{}
        
        \foreach \x in {0,10,20}
		{
			\foreach \y in {-8}
			{
				\draw[] (\x-6,\y-3) -- (\x+4,\y-3);
			
				\foreach \t in {-6,-4,-2,0,2,4}
				{
					\fill (\x+\t,\y-3) circle (.1);
				}

				\draw[] (\x-6,\y-3) .. controls (\x-5,\y-2) and (\x-3,\y-2) .. (\x-2,\y-3);

                \draw[] (\x-6,\y-3) .. controls (\x-4.5,\y) and (\x+3,\y) .. (\x+4,\y-3);

				\draw[] (\x-6,\y-3) .. controls (\x-4.5,\y-1.25) and (\x-1.5,\y-1.25) .. (\x-0,\y-3);

				\draw[] (\x+0,\y-3) .. controls (\x+1,\y-2) and (\x+3,\y-2) .. (\x+4,\y-3);	
			}
		}	
        \PruferShortCurvedLeft{20+4}{-8-3}{10}{}
         
\fill[violet] (0+10,-8+0.9) circle (.2);
\fill (0+10,-8+0.9) node [above] {$\puncture$};
         
\fill (0+2,-8-3.2) node [above] {$\Delta_0$};
\fill (10+2,-8-3.1) node [above] {$D$};

\fill (5+1,-8-3.18) node [above] {$A$};
\fill (15+1,-8-3.18) node [above] {$E$};

\fill (5+3.2,-8-2.9) node [above] {$B$};
\fill (15+3.2,-8-2.9) node [above] {$F$};

\fill (0-2,-8-0.5) node [above] {$\Delta_2$};
\fill (-2+10,-8-0.5) node [above] {$\Delta_3$};
\fill (-2+20,-8-0.5) node [above] {$\Delta_4$};  

\fill (0,-8-2.2) node [above] {$\Delta_1$};
\fill (0+10,-8-2.2) node [above] {$C$}; 
\fill (0+20,-8-2.2) node [above] {$\Delta_5$};  
         
\foreach \x in {1,2,3,4,5,6,7,8}
{
\fill (2*\x+2.25,-8-4) node {$R_\x$};
}
\fill (0+1.8,-8-4) node {$s$};
\fill (20,-8-4) node {$t$};
\end{tikzpicture} 
\caption{Finite polygon cover of $T$ containing copies of triangles crossed by a lift of $\gamma$.}
 \label{fig:example_finite_cover_strip}
\label{fig:example_bcituples}
\end{figure}
\end{ex}

\begin{defn}[BCI trail]
\label{defn:bcitrail}
Let $b$ be a BCI tuple for $\gamma$. 
We define 
$\operatorname{trail}(b):=\alpha=(\alpha_1$, $\dots$, $\alpha_{\text{length}(\alpha)})$ to be a walk from the beginning to the ending point of $\gamma$ along edges of the triangulation $\widetilde{S}_\gamma$ such that:
\begin{enumerate}[(\bcitrail i)]
\item \label{defn:bcitrail:matched}
the triangles in $b$ (called \emph{matched triangles}) are to the right of $\alpha$, and
\item \label{defn:bcitrail:unmatched}
the triangles not in $b$ (called \emph{unmatched triangles}) are to the left of $\alpha$.
\end{enumerate}

For example, Fig.~\ref{fig:bcitrail_big} shows $\operatorname{trail}(b) = (b_{40}$, $\tau_0$, $\tau_1$, $\tau_1$, $\tau_3)$ for a BCI tuple $b=(\Delta_0$, $A$, $B$, $C$, $D$, $\Delta_3$, $E$, $F)$ from Example \ref{ex:11_bci_tuples}.

\begin{figure}[h!]
\begin{center}
\begin{tikzpicture}

\draw [fill=\bcifill] (-5,0) -- (-4,1) -- (-4,-1) -- cycle;
\draw [fill=\bcifill] (-4,-1) -- (0.65,1) -- (1,-1) --cycle;
\draw [fill=\bcifill] (1,-1)  -- (4,0) -- (3,-1) -- cycle;

\foreach \t in {-4,-3,-2,-1,0,1,2,3}{
	\fill (\t, -1) circle (0.075);
}

\foreach \t in {-4,-2.75,-1.65,-.5,.65,1.75,3}{
	\fill (\t, 1) circle (0.075);
}

\draw (-4,-1) node [below] {$R_1$};
\draw (-3,-1) node [below] {$R_2$};
\draw (-2,-1) node [below] {$R_3$};
\draw (-1,-1) node [below] {$R_4$};
\draw (-0,-1) node [below] {$R_5$};
\draw (1,-1) node [below] {$R_6$};
\draw (2,-1) node [below] {$R_{7}$};
\draw (3,-1) node [below] {$R_{8}$};

\draw(-5,0) node [left] {$s$};
\draw (4,0) node [right] {$t$};

\draw (-4,1) node [above] {$L_1$};
\draw (-2.75,1) node [above] {$L_2$};
\draw (-1.65,1) node [above] {$L_3$};
\draw (-0.5,1) node [above] {$L_4$};
\draw (0.65,1) node [above] {$L_5$};
\draw (1.75,1) node [above] {$L_{6}$};
\draw (3,1) node [above] {$L_{7}$};


\draw (-4,1) .. controls (-3.5,0) and (-2.5,0) .. (-.5,1);
\draw (-3.5,0.65) node[] {$\tau_3$};
\draw (-2.75,1) .. controls (-2.25,.5) and (-1.75,.5) .. (-.5,1);
\draw (-2.75,.7) node[] {$\tau_2$};

\draw (-4,-1) .. controls (-1,-0) and (-0,0) .. (1,-1);
\draw (-.5,-.1) node[] {$\tau_4$};
\draw (-4,-1) .. controls (-1.75,-.45) and (-1.5,-.45) .. (-1,-1);
\draw (-1.5,-.5) node[] {$\tau_3$};
\draw (-4,-1) .. controls (-2.25,-.75) .. (-2,-1);
\draw (-2,-.75) node[] {$\tau_2$};
\draw (-1,-1) .. controls (-.25,-.45) and (.25,-.45) .. (1,-1);
\draw (-.5,-.5) node[] {$\tau_5$};

\draw (1,-1) .. controls (2.25,-.75) and (2.5,-.75) .. (3,-1);
\draw (2.75, -.75) node[] {$\tau_2$};

\draw[] (-4,1) -- (3,1);
\draw[] (-4,-1) -- (3,-1);
\draw[] (-4,-1) -- (-5,0);
\draw[] (3,-1)  -- (4,0)  -- (3,1);

\draw[very thick,red,postaction={on each segment={mid arrow=red}}] (-5,0) -- (-4,1);
\draw[red] (-4.5,0.65) node[left] {$b_{40}$};

\draw[very thick, red,postaction={on each segment={mid arrow=red}}] (-4,1) -- (-4,-1);
\draw[red] (-4,0) node[right] {$\tau_0$};

\draw[very thick, red,postaction={on each segment={mid arrow=red}}] (-4,-1) -- (0.65,1);
\draw[red] (-.5,.35) node[right] {$\tau_1$};

\draw[] (-4,-1) -- (-0.5,1);
\draw (-.8,.75) node[right] {$\tau_4$};

\draw[] (1,-1) -- (1.75,1);
\draw (1.15,-.5) node[right] {$\tau_4$};

\draw[very thick, red,postaction={on each segment={mid arrow=red}}] (0.65,1) -- (1,-1);
\draw[red] (.8,.4) node[right] {$\tau_1$};

\draw[very thick, red,postaction={on each segment={mid arrow=red}}] (1,-1) -- (4,0);
\draw[red] (2.3,-.5) node[above] {$\tau_3$};

\draw[] (1.75,1) -- (4,0);
\draw (2.25,.75) node[below] {$\tau_5$};

\draw (-4.5,-.65) node[left] {$b_{01}$};
\draw (3.5,.65) node[right] {$b_{40}$};
\draw (3.5,-.65) node[right] {$b_{34}$};

\end{tikzpicture}
\end{center}
\caption{The walk $\operatorname{trail}(b) = (b_{40}, \tau_5, \tau_1, \tau_1, \tau_3)$ in $\widetilde{S}_\gamma$ for a BCI tuple $b=(\Delta_0$, $A$, $B$, $C$, $D$, $\Delta_3$, $E$, $F)$ from Example \ref{ex:11_bci_tuples}.}
\label{fig:bcitrail_big}
\end{figure}
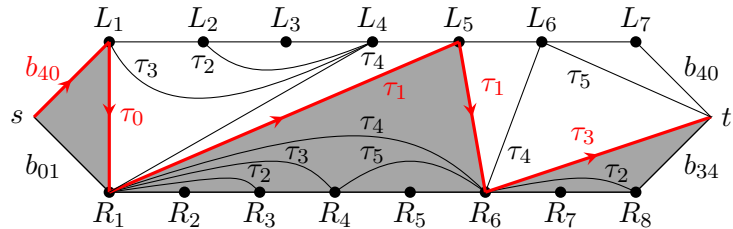

We say that $\operatorname{trail}(b)$ is the \emph{BCI trail} associated to $b$.
By convention, if $\gamma$ is an arc of $T$, the BCI tuple for $\gamma$ is an empty tuple, and $\gamma$ is the BCI trail of length $1$ for itself.
\end{defn}

\begin{rem}\label{rem:bcitrail}
To see that Definition~\ref{defn:bcitrail} is well-defined, we observe that,
if $t_1\neq \Delta_0$, the first step of $\alpha$ must go along the edge from $s$ to $R_1$. Otherwise, the first step goes from $s$ to $L_1$.
Similarly, the last step of $\alpha$ goes from $L_l$ to $t$ if $\Delta_d$ is matched. If $R_r$ is matched to a different triangle, then the last step of $\alpha$ goes from $R_r$ to $t$.
It is clear that no two distinct BCI tuples correspond to the same BCI trail.
\end{rem}

Since we are working on a polygonal cover, we only need to recall the definition of (\emph{reduced}) $T$-paths of type $A$.  
However, $T$-paths can be defined directly on ideal triangulations of a general marked surface \cite{ST09,Sch10} (possibly a once-punctured {\disk}, see \cite{GM15}). 

\begin{defn}[{\cite[Definition 1]{Sch08Tpath}}, reduced $T$-paths]
\label{defn:reduced}
Let $T$ denote the triangulation of $\widetilde{S}_\gamma$, where
$\tau_1,\dots,\tau_d$ are the inner diagonals, and $\tau_{d+1},\dots,\tau_{2d+3}$ are the boundary edges of the polygon cover $\widetilde{S}_\gamma$.
A \emph{reduced $T$-path} for $\gamma = (s,t)$ is a sequence
$$\alpha=(a_0,a_1,\dots,a_{\text{length}(\alpha)}|i_1,i_2,\dots,i_{\text{length}(\alpha)})$$
such that
\begin{enumerate}[(T1)]
\item $s=a_0,a_1,\dots,a_{\text{length}(\alpha)}=t$ are vertices of $\widetilde{S}_\gamma$. 
\item Each $\tau_{i_k}$ connects the vertices $a_{i_{k-1}}$ and $a_{i_k}$
for each $k=1,2,\dots,\text{length}(\alpha)$.
\item No step goes along the same edge twice.
\item The length of a $T$-path is odd.
\item Every even step crosses $\gamma$.
\item If $j<k$ and both $\tau_{i_j}$ and $\tau_{i_k}$ cross $\gamma$, then the crossing point of $\tau_{i_j}$ and $\gamma$ is closer to the vertex $s$ than the crossing point of $\tau_{i_k}$ and $\gamma$.
\end{enumerate}
\end{defn}

We will show the following in Sections \ref{subsec:trail_map} and \ref{subsec:triangles_map}.
\begin{prop}
\label{thm:bci_tpath_bijection}
Let $\gamma$ be a generalized arc.
The BCI trail map defined in Definition \ref{defn:bcitrail} gives a
 bijection between the set of BCI tuples for $\gamma$ and the set of reduced $T$-paths for $\gamma$.
\end{prop}

\subsection{Cluster expansion formula in terms of BCI tuples}

We assign to a BCI tuple $b$ the same weight as the reduced $T$-path corresponding to it. This choice of weight agrees with the weighting given by Caroll and Price \cite{CP03} and Propp \cite[pages 10-11]{Pro05}. An equivalent weighting is used in a recent article \cite{Yur16}. 

\begin{defn}[Laurent monomial]
To any BCI tuple $b$, 
we associate an element $x(b)$ in the cluster algebra $\mathcal{A}(S,M)$ by first considering
$\operatorname{trail}(b) = \alpha = (\alpha_1$, $\dots$, $\alpha_{\text{length}(\alpha)})$
and setting
\[
x(b)=x(\alpha)=\frac{\prod_{i \text{ odd}} x_{\alpha_i}}{\prod_{i \text{ even}} x_{\alpha_i}}.
\]
\end{defn}

Due to Proposition \ref{thm:bci_tpath_bijection}, we can rewrite the reduced $T$-path cluster expansion formula \cite[Thm. 1.2]{Sch08Tpath} in terms of BCI tuples as follows.

\begin{cor}[BCI tuple expansion formula]\label{rem:bci_formula}
Let $\mathcal{A}(S,M)$ be a cluster algebra arising from $(S,M)$, as per Theorem \ref{thm:FST}. 
Let $\gamma$ be a peripheral generalized arc.
The Laurent polynomial corresponding to $\gamma$ with respect to the cluster $x_T$ is
\begin{gather}
\label{eq:expansion_formula}
x_\gamma =\sum_{b} \, x(b)
\end{gather}
where the sum is over all BCI tuples $b$ for $\gamma$.
This formula does not depend on our choice of orientation on $\gamma$.
\end{cor}

\begin{ex}[Example of a Laurent expansion corresponding to a generalized arc]
Consider the ideal triangulation $T$ of a once-punctured {\disk} and a generalized arc in
Fig. 
\ref{fig:example_finite_cover_strip}.
We obtain the $11$ BCI tuples for $\gamma$, listed in Example
\ref{ex:11_bci_tuples}. 
Following (\ref{eq:expansion_formula}), we compute
\[
x_{\gamma} = 
\frac{x_{0} x_{1} x_{4} + 
2 x_{1} x_{3} x_{4} + 
2 x_{0}^{2} + 
4 x_{0} x_{3} + 
2 x_{3}^{2}}
{x_{0} x_{1} x_{4}}
\]
by specializing $x(\tau)=1$ for each boundary edge $\tau$.
\end{ex}

\subsection{Proof that the trail map sends BCI tuples to $T$-paths}
\label{subsec:trail_map}
In this section, we show that the map 
 \begin{gather*}
 \{ \text{BCI tuples for $\gamma$} \} \to 
 \{ \text{reduced $T$-paths for $\gamma$} \}\\
 b \mapsto \operatorname{trail}(b),
 \end{gather*}
where $\operatorname{trail}(b)$ is the BCI trail corresponding to $b$ defined in Definition \ref{defn:bcitrail}, is well-defined.
We mean this in the sense that,
given a BCI tuple $b$ for $\gamma$,
the walk $\operatorname{trail}(b)$ is indeed a reduced $T$-path for $\gamma$.

\begin{lem}\label{lem:far_vertices}
Suppose that the $i$-th right vertex $R_i$ is not adjacent to a triangle which crosses $\gamma$.
Then $R_i$ can only be matched to one triangle, and so the $i$-th entry in every BCI tuple for $\gamma$ is fixed.
\end{lem}

\begin{proof}
Let $b$ be a BCI tuple for $\gamma$.
If such a vertex exists, there must be a vertex $R_j$ which is adjacent to exactly one triangle $\triangle$.
Then $t_j = \triangle$ by (B\ref{bci:1}).
As in \cite[Section 2]{BCI74}, we can remove $\triangle$ from $\widetilde{S}_\gamma$ and get a smaller triangulated polygon.
By induction, we can remove all vertices not adjacent to a triangle crossed by $\gamma$ this way.
\end{proof}

\begin{rem}\label{rem:adjacent_to_2}
Due to Lemma \ref{lem:far_vertices}, we can assume without loss of generality that all vertices of $\widetilde{S}_\gamma$ are adjacent to triangles crossed by $\gamma$.
Note that this gives us the same triangulated polygon cover defined in \cite[Section 7]{MSW11} (see also an exposition in \cite[Section 4.1]{GM15}).
Every right vertex $R_i$ is adjacent to at least $2$ triangles of $\widetilde{S}_\gamma$, with the exception of the starting vertex $s$ and the finishing vertex $t$, which are adjacent to exactly one triangle $\Delta_0$, $\Delta_d$, respectively.
\end{rem}

\begin{figure}
\newcommand{\localscale}{1}
\begin{center}
\begin{tikzpicture}[scale=1.4*\localscale]

\foreach \t in {-2,-1,2}{ 
	\fill (\t, 1) circle (0.075);
}

\fill[violet] (0,1) circle  (0.075); 

\foreach \t in {-2,2}{ 
	\fill (\t, -1) circle (0.075);
}

\draw (-2,-1) node [below] {$R_1=1$};
\draw (-2,1) node [above] {$L_1=4$};
\draw (-1+0.1,1) node [above] {$L_2=1$};
\draw (0.3,1) node [above] {$L_3=P$};
\draw (2,-1) node [below] {$R_{2}=1$};
\draw (2,1) node [above] {$L_{4}=1$};
\draw(-3,0) node [left] {$s=0$};
\draw (3,0) node [right] {$t=4$};

\draw (0,-1) node [below] {$\tau_4$};

\draw (-2,-1) rectangle (2,1);

\draw  (-2,1) --  (-3,0) 
node[left,pos=0.3] {$b_{40}$}-- (-2,-1)
node[left,pos=0.7] {$b_{01}$};

\draw (2,-1) -- (3,0) node[right,pos=0.3] {$\tau_3$} -- (2,1) 
node[right,pos=0.7] {$\tau_0$};

\draw (-2,-1) -- (-1,1); 
\draw (-1.5,0.3) node {$\tau_4$};

\draw (-2,-1) -- (-2,1) ;
\draw (-2.2,0.3) node[] {$\tau_0$};

\draw (-2,-1) -- (0,1) node[pos=0.65,right] {$\tau_1$};
\draw (2,-1) -- (0,1) node[pos=0.65,right] {$\tau_1$};

\draw (2,-1) -- (2,1); 
\draw (2-0.2,0.3) node[] {$\tau_4$};

\draw[ultra thick, red, ->,densely dotted] (-3,0) -- (2.9,0) node[pos=0.5,below] {\Large $\gamma$};
\end{tikzpicture}
\caption{Finite $(5+3)$-gon cover $\widetilde{S}_\gamma$ from Fig.~ 
\ref{fig:example_finite_cover_strip}
(also see Fig.~\ref{fig:lattice}).} 
\label{fig:cover}
\label{fig:8gon_cover}
\end{center}


\begin{center}
\begin{tikzpicture}[scale=\localscale]

\fill [\bcifill] (-2,-1) --  (-3,0) -- (-2,1) -- cycle;
\fill [\bcifill] (-2,-1) --  (0,1) -- (2,-1) -- cycle;

\draw (-2.5,0) node {$\triangle(R_1)$};
\draw (0,-0.3) node {$\triangle(R_2)$};

\foreach \t in {-2,-1,2}{ 
	\fill (\t, 1) circle (0.075);
}

\fill[violet] (0,1) circle  (0.075); 

\foreach \t in {-2,2}{ 
	\fill (\t, -1) circle (0.075);
}

\draw (-2,-1) node [below] {$R_1$};
\draw (2,-1) node [below] {$R_2$};
\draw(-3,0) node [left] {$s$};
\draw (3,0) node [right] {$t$};

\draw (0,-1) node [below] {$\tau_4$};

\draw (-2,-1) rectangle (2,1);

\draw  (-2,1) --  (-3,0) 
node[left,pos=0.3] {$b_{40}$}-- (-2,-1)
node[left,pos=0.7] {$b_{01}$};

\draw (2,-1) -- (3,0) node[right,pos=0.3] {$\tau_3$} -- (2,1) 
node[right,pos=0.7] {$\tau_0$};

\draw (-2,-1) -- (-1,1); 
\draw (-1.1,0.3) node[] {$\tau_4$};

\draw (-2,-1) -- (-2,1); 
\draw (-2+0.2,0.3) node[] {$\tau_0$};

\draw (-2,-1) -- (0,1); 

\draw (-0.3,0.3) node[] {$\tau_1$};

\draw (2,-1) -- (0,1); 
\draw (1,0.3) node[] {$\tau_1$};

\draw (2,-1) -- (2,1); 

\draw (2-0.2,0.3) node[] {$\tau_4$};

\end{tikzpicture}
\begin{tikzpicture}[scale=\localscale]
\foreach \t in {-2,-1,2}{ 
	\fill (\t, 1) circle (0.075);
}

\fill[violet] (0,1) circle  (0.075); 

\foreach \t in {-2,2}{ 
	\fill (\t, -1) circle (0.075);
}

\draw (-2,-1) node [below] {$R_1$};
\draw (2,-1) node [below] {$R_2$};
\draw(-3,0) node [left] {$s$};
\draw (3,0) node [right] {$t$};

\draw (0,-1) node [below] {$\tau_4$};

\draw (-2,-1) rectangle (2,1);

\draw  (-2,1) --  (-3,0) 
node[left,pos=0.3] {$b_{40}$}-- (-2,-1)
node[left,pos=0.7] {$b_{01}$};

\draw (2,-1) -- (3,0) node[right,pos=0.3] {$\tau_3$} -- (2,1) 
node[right,pos=0.7] {$\tau_0$};

\draw (-2,-1) -- (-1,1); 
\draw (-1.5,0.3) node {$\tau_4$};

\draw (-2,-1) -- (-2,1) ;
\draw (-2.2,0.3) node[red] {$\tau_0$};

\draw (-2,-1) -- (0,1) node[pos=0.65,right,red] {$\tau_1$};
\draw (2,-1) -- (0,1) node[pos=0.65,right,red] {$\tau_1$};

\draw (2,-1) -- (2,1); 
\draw (2-0.2,0.3) node[] {$\tau_4$};

\draw[ultra thick, red, postaction={on each segment={mid arrow=red}}] (-3,0) -- (-2,1) -- (-2,-1) -- (0,1) ;
\draw[ultra thick, red, postaction={on each segment={mid arrow=red}}] (0,1) -- (2,-1) -- (3,0);

\end{tikzpicture}
\end{center}
\caption{A BCI tuple $b=(\Delta_0,\Delta_3)$ for $\gamma$ and corresponding BCI trail $(b_{40}$, $\tau_0$, $\tau_1$, $\tau_1$, $\tau_3)$.}
\label{fig:bcitrail_small}
\end{figure}

Fig.~\ref{fig:example_bcituples} illustrates Lemma \ref{lem:far_vertices}. The triangles $A,B,C,D,E$, and $F$ (which do not cross $\gamma$) appear in every BCI tuple. 
Instead of Fig.~\ref{fig:example_finite_cover_strip}, we work with Fig.~\ref{fig:8gon_cover}.
When drawing this smaller triangulated polygon cover, we relabel the vertices of $\widetilde{S}_\gamma$, so that the indices of $R$ and $L$ are consecutively ordered. For example, vertex $R_6$ in Fig.~\ref{fig:example_finite_cover_strip} is now vertex $R_2$ in Fig.~\ref{fig:8gon_cover}. Compare Fig.  \ref{fig:bcitrail_big} with Fig.~\ref{fig:bcitrail_small}.

\begin{lem}\label{lem:bci_tpaths}
Let $b=(t_1$, $\dots$, $t_r)=(\triangle(R_1)$, $\dots$, $\triangle(R_r))$ be a BCI tuple for $\gamma$ and
let $\alpha=(\alpha_1,\dots,\alpha_{\text{length}(\alpha)})$ be its corresponding
 BCI trail.
\begin{enumerate}
\item \label{lem:bci_tpaths:to_L} If $\alpha_{2j}$ goes from a vertex $R$ (to the right of $\gamma$)
to a vertex $L$ (to the left of $\gamma$),
then, in the tuple $b$, $R$ is matched to the triangle which is crossed by $\gamma$ immediately after $\alpha_{2j}$.
\item \label{lem:bci_tpaths:to_R} If $\alpha_{2j}$ goes from 
 a vertex $L$ (to the left of $\gamma$) to
a vertex $R$ (to the right of $\gamma$),
then, in the tuple $b$, $R$ is matched to the triangle which is crossed by $\gamma$ immediately before $\alpha_{2j}$.
\item \label{lem:bci_tpaths:crosses_gamma} Every even step $\alpha_{2j}$ crosses $\gamma$. So we have either situation (\ref{lem:bci_tpaths:to_L}) or (\ref{lem:bci_tpaths:to_R}).
\item \label{lem:bci_tpaths:odd_length} Furthermore, the length of $\alpha$ is odd.
\end{enumerate}
\end{lem}

\begin{proof}
We prove (\ref{lem:bci_tpaths:to_L}), (\ref{lem:bci_tpaths:to_R}), and (\ref{lem:bci_tpaths:crosses_gamma}) by induction on $j$.
The first step must go from the starting point $s$ of $\gamma$ to either $R_1$ (the first vertex to the right of $\gamma$) or $L_1$ (the first vertex to the left of $\gamma)$.


\begin{figure}[h!]

\begin{center}

\newcommand\localscale{0.7}

\begin{tikzpicture}[scale=\localscale]

\foreach \t in {-4,-3,-2,-1,0,1,2,3,4}{
	\fill (\t, -1) circle (0.075);
	\fill (\t, 1) circle (0.075);
}

\draw (-4,-1) node [below] {$R_1$};
\draw (-4,1) node [above] {$L_1$};
\draw (-3,-1) node [below] {$R_2$};
\draw (-3,1) node [above] {$L_2$};
\draw (3,1) node [above] {$L_{l-1}$};
\draw (4,-1) node [below] {$R_{r}$};
\draw (4,1) node [above] {$L_{l}$};
\draw(-5,0) node [left] {$s$};
\draw (5,0) node [right] {$t$};

\draw (-4,-1) rectangle (4,1);
\draw (-5,0)  -- (-4,1) -- (-4,-1)  -- cycle;
\draw (5,0)  -- (4,1) -- (4,-1)  -- cycle;

\draw[ultra thick, red, postaction={on each segment={mid arrow=red}}] (-5,0) -- (-4,-1);
\draw[ultra thick, red, postaction={on each segment={mid arrow=red}}] (-4,-1) -- (1,1);
\draw (-4.7,-1) node[above,red] {$\alpha_1$};
\draw (-2,-.1) node[above,red] {$\alpha_2$};
\draw (1,1) node[above] {$L_k$};
\draw (0,0) node[] {$\triangle \left( R_1 \right)$};
\end{tikzpicture}

\caption{The first step $\alpha_1$ ends at the right vertex $R_1$, and $\alpha_2$ ends at a left vertex $L_k$ for $k\geq 1$ (Lemma \ref{lem:bci_tpaths} Case \ref{lem:bci_tpaths:to_L}).}
\label{fig:alpha_1_ends_at_R1}
\label{fig:to_R1}


\begin{tikzpicture}[scale=\localscale]

\fill [\bcifill] (-5,0) -- (-4,1) -- (0,-1) -- (-4,-1) --cycle;

\foreach \t in {-4,-3,-2,-1,0,1,2,3,4}{
	\fill (\t, -1) circle (0.075);
	\fill (\t, 1) circle (0.075);
}

\draw (-4,-1) node [below] {$R_1$};
\draw (-4,1) node [above] {$L_1$};
\draw (-3,-1) node [below] {$R_2$};
\draw (-3,1) node [above] {$L_2$};
\draw (3,1) node [above] {$L_{l-1}$};
\draw (4,-1) node [below] {$R_{r}$};
\draw (4,1) node [above] {$L_{l}$};
\draw(-5,0) node [left] {$s$};
\draw (5,0) node [right] {$t$};

\draw (0,-1) node [below] {$R_{k}$};

\draw (-4,-1) rectangle (4,1);
\draw (-5,0)  -- (-4,1) -- (-4,-1)  -- cycle;
\draw (5,0)  -- (4,1) -- (4,-1)  -- cycle;

\draw[ultra thick, red, postaction={on each segment={mid arrow=red}}] (-5,0) -- (-4,1);
\draw[ultra thick, red, postaction={on each segment={mid arrow=red}}] (-4,1) -- (0,-1);

\draw (-4.7,1) node[below,red] {$\alpha_1$};
\draw (-2,0) node[above,red] {$\alpha_2$};

\end{tikzpicture}
\caption{The first step $\alpha_1$ ends at the left vertex $L_1$, and $\alpha_2$ ends at a right vertex $R_k$ for $k\geq 1$ (Lemma \ref{lem:bci_tpaths} Case \ref{lem:bci_tpaths:to_R}). Shaded region represents triangles $\Delta_0, \ldots, \Delta_{k-1}$.
}
\label{fig:alpha_1_ends_at_L1}
\label{fig:to_L1}

\end{center}

\end{figure}

First, suppose $R_1$ is not matched to the first triangle $\Delta_0$ crossed by $\gamma$.
See Fig.~\ref{fig:to_R1}.
Since $\Delta_0$ is not adjacent to other vertices to the right of $\gamma$, the BCI tuple $b$ corresponding to $\alpha$ does not contain $\Delta_0$.
Then $\alpha_1$ must go to $R_1$ because an unmatched triangle of $\widetilde{S}_\gamma$ has to be to the left of $\alpha$.
Then $R_1$ is matched to a different triangle $t_1$ which $\gamma$ crosses some time after $\Delta_0$,
and no other triangle (if any) between $\Delta_0$ and $t_1$ is contained in $b$.
Hence, by (\bcitrail \ref{defn:bcitrail:matched}) and (\bcitrail \ref{defn:bcitrail:unmatched}), $t_1$ is immediately to the right of the next step $\alpha_2$.
Hence $\alpha$ crosses $\gamma$ by going from $R_1$ to a left vertex.
This is the base case for (\ref{lem:bci_tpaths:to_L}) and (\ref{lem:bci_tpaths:crosses_gamma}).

Second, suppose $R_1$ is matched to the first triangle $\Delta_0$ crossed by $\gamma$.
See Fig.~\ref{fig:to_L1}.
Then $\Delta_0$ is the closest matched triangle to $s$.
Then $\alpha_1$ must go to $L_1$ because a matched triangle has to be to the right of $\gamma$.
For some $k\geq 1$ (possibly $k=1$), the triangles $t_1=\Delta_0$, $\dots$, $t_k=\Delta_{k-1}$ matched by $R_1$, $R_2$, $\dots$, $R_k$ form a maximal connected fan.
Then $\alpha_2$ crosses $\gamma$ by going from $L_1$ to $R_k$.
Furthermore, the triangle $\Delta_{k-1}$ matched to $R_k$ appears immediately before $\alpha_{2}$.
This is the base case for (\ref{lem:bci_tpaths:to_R}) and (\ref{lem:bci_tpaths:crosses_gamma}).\\

Let $j$ be given and
assume $\alpha_{2j}$ crosses $\gamma$.
First, assume (\ref{lem:bci_tpaths:to_L}): suppose $\alpha_{2j}$ starts from a right vertex $R$ and ends at a vertex left $L_k$,
and suppose that the triangle $\triangle(R)$ matched to $R$ appears immediately after $\alpha_{2j}$.
See Figs. \ref{fig:to_Rc} and \ref{fig:to_Lk1}.


\begin{figure}[h!]
\begin{center}
\begin{tikzpicture}

\fill [\bcifill] (-2,-1) -- (-1,1) -- (0,1) -- (1,-1) --cycle;
\draw (-.5,0) node[] {\large \textbf{$C$}};

\foreach \t in {-4,-3,-2,-1,0,1,2,3,4}{
	\fill (\t, -1) circle (0.075);
	\fill (\t, 1) circle (0.075);
}

\draw (-4,-1) node [below] {$R_1$};
\draw (-4,1) node [above] {$L_1$};
\draw (-3,1) node [above] {$L_2$};
\draw (3,1) node [above] {$L_{l-1}$};
\draw (4,-1) node [below] {$R_{r}$};
\draw (4,1) node [above] {$L_{l}$};
\draw(-5,0) node [left] {$s$};
\draw (5,0) node [right] {$t$};

\draw (-2,-1) node [below] {$R_{C(start)}$};
\draw (-1,1) node [above] {$L_k$};
\draw (0,1) node [above] {$L_{k+1}$};
\draw (1,-1) node [below] {$R_{C(finish)}$};

\draw (-4,-1) rectangle (4,1);
\draw (-5,0)  -- (-4,1) -- (-4,-1)  -- cycle;
\draw (5,0)  -- (4,1) -- (4,-1)  -- cycle;

\draw[ultra thick, red, postaction={on each segment={mid arrow=red}}] (-2,-1) -- (-1,1);
\draw[ultra thick, red, postaction={on each segment={mid arrow=red}}] (-1,1) -- (0,1);
\draw[ultra thick, red, postaction={on each segment={mid arrow=red}}] (0,1) -- (1,-1);

\draw (-2,0) node[red] {\tiny{$\alpha_{2j}$}};
\draw (1.1,0) node[red] {\tiny{$\alpha_{2j+2}$}};
\draw (-.5,1) node[red,below] {\tiny{$\alpha_{2j+1}$}};

\end{tikzpicture}
\caption{The subpath $\alpha_{2j},\alpha_{2j+1},\alpha_{2j+2}$ starts and finishes on the right of $\gamma$. Lemma \ref{lem:bci_tpaths} Case \ref{lem:bci_tpaths:to_L}: When $\alpha_{2j}$ starts at a right vertex and $\alpha_{2j+2}$ ends at a right vertex}
\label{fig:RtoR}
\label{fig:to_Rc}

\end{center}


\begin{center}
\begin{tikzpicture}

\fill [\bcifill] (-2,-1) -- (-1,1) -- (1,-1) -- cycle;
\draw (-.75,-.25) node[] {\large \textbf{$C$}};

\foreach \t in {-4,-3,-2,-1,0,1,2,3,4}{
	\fill (\t, -1) circle (0.075);
	\fill (\t, 1) circle (0.075);
}

\draw (-4,-1) node [below] {$R_1$};
\draw (-4,1) node [above] {$L_1$};
\draw (4,-1) node [below] {$R_{r}$};
\draw (4,1) node [above] {$L_{l}$};
\draw(-5,0) node [left] {$s$};
\draw (5,0) node [right] {$t$};

\draw (-2,-1) node [below] {$R_{C(start)}$};
\draw (-1,1) node [above] {$L_k$};
\draw (1,-1) node [below] {$R_{C(finish)}$};

\draw (-4,-1) rectangle (4,1);
\draw (-5,0)  -- (-4,1) -- (-4,-1)  -- cycle;
\draw (5,0)  -- (4,1) -- (4,-1)  -- cycle;

\draw[ultra thick, red, postaction={on each segment={mid arrow=red}}] (-2,-1) -- (-1,1);
\draw[ultra thick, red, postaction={on each segment={mid arrow=red}}] (-1,1) -- (1,-1);
\draw[ultra thick, red, postaction={on each segment={mid arrow=red}}] (1,-1) -- (3,1);

\draw (-2,0) node[red] {\tiny{$\alpha_{2j}$}};
\draw (0.6,0) node[red] {\tiny{$\alpha_{2j+1}$}};
\draw (2.6,0) node[red] {\tiny{$\alpha_{2j+2}$}};

\draw (2.8,0-0.5) node[] {$\Delta(R_{C(finish)})$};

\end{tikzpicture}
\caption{The subpath $\alpha_{2j}, \alpha_{2j+1},\alpha_{2j+2}$ starts on the right and finishes on the left of $\gamma$.
Lemma \ref{lem:bci_tpaths} Case \ref{lem:bci_tpaths:to_L}: $\alpha_{2j}$ starts to the right of $\gamma$ and $\alpha_{2j+2}$ ends to the left of $\gamma$.
}
\label{fig:RtoL}
\label{fig:to_Lk1} 
\end{center}

\end{figure}

If $k=l$, then $L_k$ is the last vertex to the left of $\gamma$ before the ending point $t$ of $\gamma$. 
Hence the triangles matched to vertices between $R$ and $R_r$ (inclusive) form a maximal connected fan which contains $\Delta_d$.
Therefore $\alpha_{2j+1}$ goes from $L_l$ to $t$, proving (\ref{lem:bci_tpaths:odd_length}).
So suppose $k<l$.

Consider the maximal connected component $C$ of matched triangles which $\triangle(R)$ is a part of.
The edge between $L_k$ and $L_{k+1}$ is an edge of a triangle $\triangle[L_k,L_{k+1}]$.
Consider these two possibilities: either $\triangle[L_k,L_{k+1}]$ is in $C$ or it is not. 
Suppose it is. See Fig.~\ref{fig:to_Rc}.
Then $\alpha_{2j+1}$ must go from $L_k$ to $L_{k+1}$ so that the matched triangles stay to the right of $\alpha$.
Furthermore, the right vertex $R_{C(finish)}$ that bounds $C$ on the side that is closest to the end of $\gamma$ must be matched to a triangle in $C$. (Note that vertex $R_{C(finish)}$ is isomorphic to $R$ if and only if $C$ contains exactly one triangle, $\triangle(R)$).
Also, $C$ cannot include $\triangle[L_{k+1},L_{k+2}]$. 
Then $\alpha_{2j+2}$ must go from $L_{k+1}$ to $R_{C(finish)}$ in order to obey Definition \ref{defn:bcitrail}. This satisfies (\ref{lem:bci_tpaths:crosses_gamma}) for $j+1$. 
Furthermore, the triangle which is matched to $R_{C(finish)}$ appears immediately before $\alpha_{2j+2}$, satisfying (\ref{lem:bci_tpaths:to_L}) for $j+1$.

Consider the other possibility ($\triangle[L_k,L_{k+1}]$ not in $C$) as in Fig.~\ref{fig:to_Lk1}.
If the edge between $L_k$ and $L_{k+1}$ is not an edge of $C$,
then $\alpha_{2j+1}$ must go to a right vertex $R_{C(finish)}$ in order to keep the unmatched triangles to the left of $\alpha$. But in this situation $R_{C(finish)}$ is not matched to a triangle in $C$, so $R_{C(finish)}$ must be matched to a triangle $\triangle(R_{C(finish)})$ outside of $C$. Since by assumption $C$ is a maximal connected component of matched triangles, there must be at least one unmatched triangle immediately after $C$. Then $\alpha_{2j+2}$ needs to go to a left vertex in order to have the unmatched triangles stay to the left of $\alpha$ and the triangle $\triangle(R_{C(finish)})$ be to the right of $\alpha$.

Second, assume (\ref{lem:bci_tpaths:to_R}):
suppose $\alpha_{2j}$ ends at a vertex $R_k$ (to the right of $\gamma$),
and suppose that the triangle $\triangle(R_k)$ matched to $R_k$ appears immediately before $\alpha_{2j}$. See Figs. \ref{fig:LtoL} and \ref{fig:LtoR}.


\begin{figure}[h!]
\begin{center}
\begin{tikzpicture}

\draw (0,0) node[] {\large \textbf{$D$}};

\foreach \t in {-4,-3,-2,-1,0,1,2,3,4}{
	\fill (\t, -1) circle (0.075);
	\fill (\t, 1) circle (0.075);
}

\draw (-4,-1) node [below] {$R_1$};
\draw (-4,1) node [above] {$L_1$};
\draw (4,-1) node [below] {$R_{r}$};
\draw (4,1) node [above] {$L_{l}$};
\draw(-5,0) node [left] {$s$};
\draw (5,0) node [right] {$t$};

\draw (-2,1) node [above] {$L$};
\draw (-1,-1) node [below] {$R_k$};
\draw (0,-1) node [below] {$R_{k+1}$};

\draw (-4,-1) rectangle (4,1);
\draw (-5,0)  -- (-4,1) -- (-4,-1)  -- cycle;
\draw (5,0)  -- (4,1) -- (4,-1)  -- cycle;

\draw[ultra thick, red, postaction={on each segment={mid arrow=red}}] (-2,1) -- (-1,-1);
\draw[ultra thick, red, postaction={on each segment={mid arrow=red}}] (-1,-1) -- (0,-1);
\draw[ultra thick, red, postaction={on each segment={mid arrow=red}}] (0,-1) -- (2,1);

\draw (-2,0-0.5) node[] {$\triangle \left(R_k \right)$};
\draw (1.8,0-0.5) node[] {$\triangle \left(R_{k+1}\right)$};

\draw (-2,0) node[red] {\tiny{$\alpha_{2j}$}};
\draw (-.5,-1) node[red,above] {\tiny{$\alpha_{2j+1}$}};
\draw (1.8,0) node[red] {\tiny{$\alpha_{2j+2}$}};

\end{tikzpicture}
\caption{The subpath $\alpha_{2j},\alpha_{2j+1},\alpha_{2j+2}$ starts and finishes on the left of $\gamma$. Lemma \ref{lem:bci_tpaths} Case \ref{lem:bci_tpaths:to_R}: $\alpha_{2j}$ starts to the left of $\gamma$ and $\alpha_{2j+2}$ ends to the left of $\gamma$.}
\label{fig:LtoL} 
\label{fig:from_Rk1}
\end{center}


\begin{center}
\begin{tikzpicture}

\fill [\bcifill] (-1,-1) -- (1,1) -- (2,-1) -- cycle;
\draw (-1,0.25) node[] {\large \textbf{$D$}};
\draw (.75,-.25) node[] {\large \textbf{$C$}};

\foreach \t in {-4,-3,-2,-1,0,1,2,3,4}{
	\fill (\t, -1) circle (0.075);
	\fill (\t, 1) circle (0.075);
}

\draw (-4,-1) node [below] {$R_1$};
\draw (-4,1) node [above] {$L_1$};
\draw (3,1) node [above] {$L_{l-1}$};
\draw (4,-1) node [below] {$R_{r}$};
\draw (4,1) node [above] {$L_{l}$};
\draw(-5,0) node [left] {$s$};
\draw (5,0) node [right] {$t$};

\draw (-2,1) node [above] {$L$};
\draw (-1,-1) node [below] {$R_k$};
\draw (2,-1) node [below] {$R_{C(finish)}$};

\draw (-4,-1) rectangle (4,1);
\draw (-5,0)  -- (-4,1) -- (-4,-1)  -- cycle;
\draw (5,0)  -- (4,1) -- (4,-1)  -- cycle;

\draw[ultra thick, red, postaction={on each segment={mid arrow=red}}] (-2,1) -- (-1,-1);
\draw[ultra thick, red, postaction={on each segment={mid arrow=red}}] (-1,-1) -- (1,1);
\draw[ultra thick, red, postaction={on each segment={mid arrow=red}}] (1,1) -- (2,-1);

\draw (-2,0) node[red] {\tiny{$\alpha_{2j}$}};
\draw (0,0.6) node[red] {\tiny{$\alpha_{2j+1}$}};
\draw (2.2,0) node[red] {\tiny{$\alpha_{2j+2}$}};

\draw (-2,0-0.5) node[] {$\triangle \left(R_k \right)$};

\end{tikzpicture}
\caption{The subpath $\alpha_{2j},\alpha_{2j+1},\alpha_{2j+2}$ starts on the left and finishes on the right of $\gamma$. Lemma \ref{lem:bci_tpaths} Case \ref{lem:bci_tpaths:to_R}: When $\alpha_{2j}$ starts to the left of $\gamma$ and $\alpha_{2j+2}$ ends to the right of $\gamma$.}
\label{fig:LtoR} 
\label{fig:from_Lc}
\end{center}

\end{figure}

If $k=r$, then $R_k$ is the last vertex before the ending point of $\gamma$. 
Then the rest of the triangles after $\triangle(R_r)$ (if any) are not matched, so they need to be to the left of $\alpha$, 
and so $\alpha_{2j+1}$ goes from $R_r$ to $t$. 
This proves (\ref{lem:bci_tpaths:odd_length}).
So suppose $k<r$.

Let $D$ denote the maximal connected component of the unmatched triangles between  $\triangle(R_k)$ and  $\triangle(R_{(k+1)})$.
Note that $D$ needs to be to the left of $\alpha_{2j+1}$ by (\bcitrail \ref{defn:bcitrail:unmatched}). 
If the edge between $R_k$ and $R_{k+1}$ is an edge of $D$, then $\alpha_{2j+1}$ needs to go from $R_k$ to $R_{k+1}$. 
Then $\alpha_{2j+2}$ must go from $R_{k+1}$ to a left vertex (since $R_{k+1}$ must be matched to the triangle adjacent to $R_{k+1}$ which appears immediately after $D$). See Fig.~\ref{fig:LtoL}.

If the edge between $R_k$ and $R_{k+1}$ is not an edge of $D$, then $R_{k+1}$ must be matched to a triangle adjacent to $R_k$. So $\alpha_{2j+1}$ goes from $R_k$ to a left vertex. See Fig.~\ref{fig:LtoR}. 
Let $C$ be the maximal connected fan of matched triangles which contain $\triangle(R_{(k+1)})$.
Every right vertex in $C$ must be matched to a triangle as close to the beginning of $\gamma$ as possible. Hence this component forms a pyramid shape fan (with the base to the right of $\gamma$), and so $\alpha_{2j+2}$ goes from the left of $\gamma$ to the right of $\gamma$. Furthermore, the endpoint of $\alpha_{2j+2}$ is matched to the triangle immediately before $\alpha_{2j+2}$.
This concludes our induction step for (\ref{lem:bci_tpaths:to_R}) and (\ref{lem:bci_tpaths:crosses_gamma}).
\end{proof}

\begin{prop}\label{prop:bcitrail_tpath}
A BCI trail for $\gamma$ is a reduced $T$-path for $\gamma$.
\end{prop}
\begin{proof}
The definition of a BCI trail satisfies (T1), (T2), (T3), and (T6).
Lemma \ref{lem:bci_tpaths} proves that a BCI trail satisfies (T4) and (T5).
\end{proof}

\subsection{Triangles map}
\label{subsec:triangles_map}

Conversely, given a reduced $T$-path $\alpha$ for $\gamma$, we show that the set of triangles to the right of $\gamma$ forms a BCI tuple for $\gamma$.
We define the map
\begin{gather*}
\{ \text{reduced $T$-paths for $\gamma$}\} \to 
\{\text{BCI tuples for $\gamma$}\}\\
\alpha \mapsto \operatorname{triangles}(\alpha)
\end{gather*}
which associates to $\alpha$ a tuple of triangles of $\widetilde{S}_\gamma$ that are to the right of $\alpha$. 
The following algorithm tells us the well-defined way to match the $R_i$'s to triangles as we go along $\widetilde{S}_\gamma$ from $R_1$ to $R_r$.

If $\alpha_1$ goes from $s$ to the right vertex $R_1$, we assign $\triangle(R_1)$ to be the triangle immediately to the right of $\alpha_2$. See Fig.~\ref{fig:alpha_1_ends_at_R1}.

If $\alpha_1$ goes to the left vertex $L_1$, then $\alpha_2$ goes from $L_1$ to some right vertex $R_k$. 
We match the first $k$ vertices $R_1,\dots,R_k$ to the triangles $\Delta_0,\dots,\Delta_{k-1}$, in order. These triangles are represented by the shaded area in
 Fig.~\ref{fig:alpha_1_ends_at_L1}.

Since $\alpha_{2j}$ must cross $\gamma$, this step must go from the right to the left of $\gamma$ (or vice versa).
First, suppose $\alpha_{2j}$ goes from a right vertex $R_{C(start)}$ to some left vertex $L_k$.
Then $R_{C(start)}$ has not been matched yet to any triangle which is crossed by $\gamma$ prior to $\alpha_{2j}$.

If $\alpha_{2j+1}$ goes from $L_k$ to the next left vertex $L_{k+1}$, then $\alpha_{2j+2}$ goes from $L_{k+1}$ to a right vertex $R_{C(finish)}$ where $C(start) \leq C(finish)$ (possibly $C(start)=C(finish)$). We match the vertices $R_{C(start)}$, $\dots$, $R_{C(finish)}$ to the triangles in the trapezoid bounded by the subpath $\alpha_{2j}$, $\alpha_{2j+1}$, $\alpha_{2j+2}$. See Fig.~\ref{fig:RtoR}.

If $\alpha_{2j+1}$ goes from $L_k$ to a right vertex $R_{C(finish)}$ where $C(start) < C(finish)$, then 
we match the vertices $R_{C(start)}$, $\dots$, $R_{C(finish)-1}$ to the triangles in the fan bounded by the subpath $\alpha_{2j}$, $\alpha_{2j+1}$, $\alpha_{2j+2}$. See Fig.~\ref{fig:RtoL}.

Second, suppose $\alpha_{2j}$ goes from a left vertex $L$ to some right vertex $R_k$.
Then $R_k$ has been matched to the triangle (having $\alpha_{2j}$ as a side) which is crossed by $\gamma$ prior to $\alpha_{2j}$.

If $\alpha_{2j+1}$ goes from $R_k$ to $R_{k+1}$, then $\alpha_{2j+2}$ goes from the right to the left of $\gamma$ (see Fig.~\ref{fig:LtoL}). This situation for $\alpha_{2j+2}$ has been discussed earlier.

If $\alpha_{2j+1}$ goes from $R_k$ to a left vertex and $\alpha_{2j+2}$ goes to a right vertex $R_{C(finish)}$ (where $k<C(finish)$), then we match $R_{k+1},\dots,R_{C(finish)}$ to the triangles in the fan bounded by the subpath $\alpha_{2j+1}$, $\alpha_{2j+2}$.
See Fig.~\ref{fig:LtoR}.

Finally, the final (and hence odd) step of $\alpha$
either starts from $L_l$ or from $R_r$. If the final step goes from $L_l$ to $t$,
then the previous even step finishes at a left vertex (Fig.~\ref{fig:RtoL} or \ref{fig:LtoL}). Hence $R_r$ has not been matched to a triangle yet. We match $R_r$ to the final triangle $\Delta_d$.
If the final step goes from $R_r$ to $t$, then the previous even step finishes at a right vertex (Fig.~\ref{fig:RtoR} or \ref{fig:LtoR}). Hence $R_r$ has been matched to the triangle having this even step as a side.
By construction, the $\operatorname{triangles}$ map is well-defined.

\begin{lem}
The triangles map $\alpha \mapsto \operatorname{triangles}(\alpha)$ is the inverse of the trail map $b \mapsto \operatorname{trail}(b)$. In particular, both maps are bijections.
\end{lem}
\begin{proof}
Let $\alpha=(\alpha_1,\dots,\alpha_{\text{length}(\alpha)})$ be a reduced $T$-path for $\gamma$.
Then $\operatorname{triangles}(\alpha)$ are the triangles to the right of $\alpha$, which form a BCI tuple for $\gamma$.
But $\operatorname{trail}(\operatorname{triangles}(\alpha))$ is the reduced $T$-path to the left of $\operatorname{triangles}(\alpha)$, so $\alpha=\operatorname{trail}(\operatorname{triangles}(\alpha))$.

Conversely, let $b=(\triangle(R_1),\dots,\triangle(R_r))$ be a BCI tuple for $\gamma$.
Then $\operatorname{triangles}(\operatorname{trail}(b))$ is the tuple of all the triangles to the right of $\operatorname{trail}(b)$.
Since $\operatorname{trail}(b)$ is the trail to the left of $b$, we have $b=\operatorname{triangles}(\operatorname{trail}(b))$.
\end{proof}

\subsection{Natural lattice structure of the BCI tuples}

\begin{rem}
It is known that the set of all snake graph perfect matchings has a natural distributive lattice structure ~\cite[Theorem 5.2]{MSW13} (as a consequence of ~\cite[Thm. 2]{Pro02} or ~\cite[Thm. 3]{Fel04}). 
Using a bijection of~\cite[Thm. 4.6]{MS10} between snake graph perfect matchings and complete $T$-paths (see \cite[Def. 2]{Sch10}), it is straight-forward to show
that the BCI tuples have a natural lattice structure which is preserved by the bijection of Proposition \ref{thm:bci_tpath_bijection}.

In the same spirit as~\cite{MS10,Sch10,MSW13}, we define the \emph{minimal} (respectively, \emph{maximal}) BCI tuple to be the tuple where each chosen triangle is as close as possible to (respectively, as far as possible from) the starting point of $\gamma$.  Recall that, due to the convention we use in Definition \ref{defn:signed_adjacency_matrix_ideal}, our convention is equivalent to the convention of~\cite{MS10,Sch10}, so our choice of the minimal BCI tuple corresponds to the minimal snake graph matching and complete $T$-path in said articles.
\end{rem}

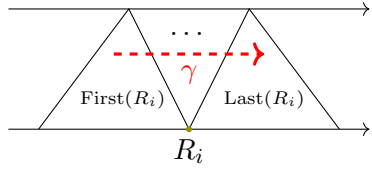
\begin{figure}[h]
\begin{center}
\begin{tikzpicture}[scale=0.4]
\draw[->] (-6,0) to (6,0) ;
\draw[->] (-6,4) to (6,4) ;
\draw (0,0) to (2,4) to (5,0) (0,0) to (-2,4) to (-5,0);
\fill (0,3.2) node {$\dots$};
\fill[olive] (0,0) circle (.1);
\fill (0,0) node [below] {$R_i$};
\fill (-2.2,1) node {\tiny{First($R_i$)}};
\fill (2.5,1) node {\tiny{Last($R_i$)}};
\draw[->,red, dashed,very thick] (-2.5,2.5) -- (2.5,2.5) node[below,pos=0.5] {$\gamma$};
\end{tikzpicture}
\caption{The collection $FAN(R_i)$ of triangles adjacent to $R_i$.}
\label{fig:FANRi}
\end{center}
\end{figure}

\begin{figure}[h!]
\newcommand{\localscale}{1}
\begin{center}
\begin{tikzpicture}[scale=\localscale]

\fill [\bcifill] (-2,-1) --  (-3,0) -- (-2,1) -- cycle;
\fill [\bcifill] (-2,-1) --  (0,1) -- (2,-1) -- cycle;

\draw (-2.5,0) node {$\triangle(R_1)$};
\draw (0,-0.3) node {$\triangle(R_2)$};

\foreach \t in {-2,-1,2}{ 
	\fill (\t, 1) circle (0.075);
}

\fill[violet] (0,1) circle  (0.075); 

\foreach \t in {-2,2}{ 
	\fill (\t, -1) circle (0.075);
}

\draw (-2,-1) node [below] {$R_1$};
\draw (2,-1) node [below] {$R_2$};
\draw(-3,0) node [left] {$s$};
\draw (3,0) node [right] {$t$};

\draw (0,-1) node [below] {$\tau_4$};

\draw (-2,-1) rectangle (2,1);

%
%
%

\draw  (-2,1) --  (-3,0) 
node[left,pos=0.3] {$b_{40}$}-- (-2,-1)
node[left,pos=0.7] {$b_{01}$};

\draw (2,-1) -- (3,0) node[right,pos=0.3] {$\tau_3$} -- (2,1) 
node[right,pos=0.7] {$\tau_0$};

\draw (-2,-1) -- (-1,1); 
\draw (-1.1,0.3) node[] {$\tau_4$};

\draw (-2,-1) -- (-2,1); 
\draw (-2+0.2,0.3) node[] {$\tau_0$};

\draw (-2,-1) -- (0,1); 

\draw (-0.3,0.3) node[] {$\tau_1$};

\draw (2,-1) -- (0,1); 
\draw (1,0.3) node {$\tau_1$};

\draw (2,-1) -- (2,1); 

\draw (2-0.2,0.3) node[] {$\tau_4$};

\end{tikzpicture}
\begin{tikzpicture}[scale=\localscale]

\foreach \t in {-2,-1,2}{ 
	\fill (\t, 1) circle (0.075);
}

\fill[violet] (0,1) circle  (0.075); 

\foreach \t in {-2,2}{ 
	\fill (\t, -1) circle (0.075);
}

\draw (-2,-1) node [below] {$R_1$};
\draw (2,-1) node [below] {$R_2$};
\draw(-3,0) node [left] {$s$};
\draw (3,0) node [right] {$t$};

\draw (0,-1) node [below] {$\tau_4$};

\draw (-2,-1) rectangle (2,1);

\draw  (-2,1) --  (-3,0) 
node[left,pos=0.3,red] {$b_{40}$}-- (-2,-1)
node[left,pos=0.7] {$b_{01}$};

\draw (2,-1) -- (3,0) node[right,pos=0.3,red] {$\tau_3$} -- (2,1) 
node[right,pos=0.7] {$\tau_0$};

\draw (-2,-1) -- (-1,1); 
\draw (-1.5,0.3) node {$\tau_4$};

\draw (-2,-1) -- (-2,1) ;
\draw (-2.2,0.3) node[red] {$\tau_0$};

\draw (-2,-1) -- (0,1) node[pos=0.65,right,red] {$\tau_1$};
\draw (2,-1) -- (0,1) node[pos=0.65,right,red] {$\tau_1$};

\draw (2,-1) -- (2,1); 
\draw (2-0.2,0.3) node[] {$\tau_4$};

\draw[ultra thick, red, postaction={on each segment={mid arrow=red}}] (-3,0) -- (-2,1) -- (-2,-1) -- (0,1) ;
\draw[ultra thick, red, postaction={on each segment={mid arrow=red}}] (0,1) -- (2,-1) -- (3,0);

\end{tikzpicture}
\end{center}
\caption{The (minimal) BCI tuple $(First(R_1),First(R_2))$ for $\gamma$ from Figure \ref{fig:8gon_cover} and its corresponding reduced $T$-path.}
\label{fig:minimalbcituple}
\end{figure}
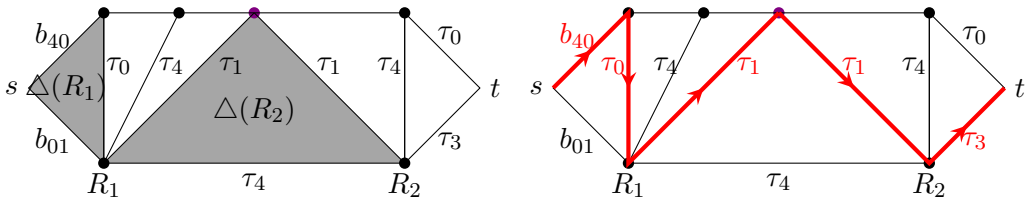

\begin{defn}
[Minimal and maximal BCI tuples]
For a vertex $R_i$ to the right of $\gamma$, denote by FAN$(R_i)$ the fan of all triangles adjacent to $R_i$ which are crossed by $\gamma$. 
Per Remark \ref{rem:adjacent_to_2}, this fan contains at least two triangles. See Fig.~\ref{fig:FANRi}.
The orientation of $\gamma$ determines an ordering of the triangles in this fan. 
Let fan$(R_i)_k$ be the $k$-th triangle in FAN$(R_i)$.
Let $First(R_i):= \text{fan}{(R_i)}_1$ and $Last(R_i):=\text{fan}{(R_i)}_{|\text{FAN}(R_i)|}$ denote the first and last triangles in this fan. 
We say that a BCI tuple is the \emph{minimal} BCI tuple if each vertex $R_i$ is mapped to 
$First(R_i)$ (see Fig.~\ref{fig:minimalbcituple}).
Similarly, the \emph{maximal} BCI tuple is defined to be 
$(Last(R_1)$, $\dots$, $Last(R_r))$.
\end{defn}

The lattice is graded by the sum of distances from the triangles $First(R_i)$ of the minimal BCI tuple.

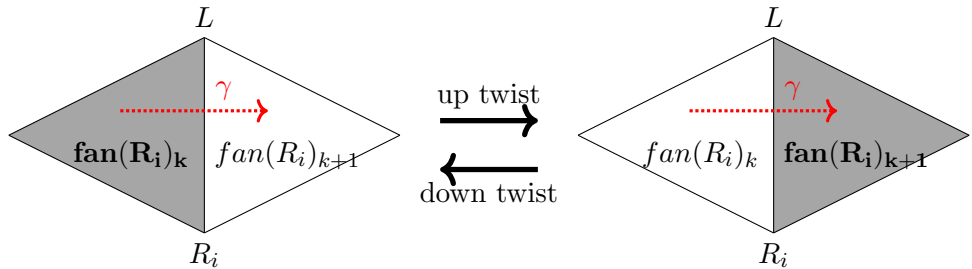
\begin{figure}[h]
\begin{center}
\newcommand{\localscale}{0.65}
\newcommand{\nn}{1.1}
\newcommand{\localwidth}{4}
\begin{tikzpicture}[scale=\localscale]

\draw (0,-2) node[below] {$R_i$};
\draw (0,2) node[above] {$L$};

\fill [\bcifill] (-\localwidth,0) -- (0,-2) -- (0,2) -- cycle;

\draw 
(-\localwidth,0) -- (0,2);
\draw (-\localwidth,0) -- (0,-2);

\draw (0,2) -- (\localwidth,0);
\draw 
(0,-2) -- (\localwidth,0);

\draw
(0,2) -- (0,-2);

\draw (-1-0.5,-0.4) node[] {$\mathbf{fan(R_i)_k}$};
\draw (1.2+0.5,-0.4) node[] {$fan(R_i)_{k+1}$};

\draw[->,very thick, red,densely dotted] (-1.72,0.5) -- (1.3,0.5) node[pos=0.7,above] {$\gamma$};


\end{tikzpicture}
\begin{tikzpicture}[scale=\localscale]
\draw[->,line width=2bp] (0,0.5) -- (2,0.5) node[above, pos=0.5] {\text{up twist}};
\draw[->,line width=2bp] (2,-0.5) -- (0,-0.5) node[below, pos=0.5] {\text{down twist}};
\draw (0,-2.5) node {};
\end{tikzpicture}
\begin{tikzpicture}[scale=\localscale]

\draw (0,-2) node[below] {$R_i$};
\draw (0,2) node[above] {$L$};

\fill [\bcifill] (\localwidth,0) -- (0,-2) -- (0,2) -- cycle;

\draw
(-\localwidth,0) -- (0,-2);
\draw (-\localwidth,0) -- (0,2);

\draw (0,-2) -- (\localwidth,0);
\draw
(0,2) -- (\localwidth,0);

\draw
(0,-2) -- (0,2);

\draw (-1-0.5,-0.4) node[] {$fan(R_i)_{k}$};
\draw (1+0.7,-0.4) node[] {$\mathbf{fan(R_i)_{k+1}}$};

\draw[->, very thick, red,densely dotted] (-1.72,0.5) -- (1.3,0.5) node[pos=0.7,above] {$\gamma$};

\end{tikzpicture}
\end{center}
\caption{Twist of a BCI tuple.}
\label{fig:bcitrail_defn_twist}
\end{figure}

\begin{defn}[twist]
We define an \emph{up twist} to be a local move that affects precisely one triangle $t_i= \text{fan}(R_i)_k$ of $b$,
replacing fan$(R_i)_k$ with fan$(R_i)_{k+1}$.
A \emph{down twist} replaces fan$(R_i)_k$ with fan$(R_i)_{k-1}$.
See Fig.~\ref{fig:bcitrail_defn_twist}.
\end{defn}

\begin{figure}[h]
\begin{center}

\newcommand{\localscale}{0.65}
\newcommand{\nn}{1.2}
\begin{tikzpicture}[scale=\localscale]

\draw (0.7*\nn,0.7*\nn) node[] {$+$};

\draw[line width=3bp] (-1*\nn,-1*\nn) -- (-1*\nn,1*\nn);
\draw[line width=3bp] (1*\nn,-1*\nn) -- (1*\nn,1*\nn);

\draw (-1*\nn,-1*\nn) -- (1*\nn,-1*\nn);
\draw (-1*\nn,1*\nn) -- (1*\nn,1*\nn);

\draw (0,0) node {$i_k$};
\draw (-1.2*\nn,1.2*\nn) node {L};
\draw (1.2*\nn,-1.2*\nn) node {R};
\end{tikzpicture}
\begin{tikzpicture}[scale=\localscale]
\draw[->,ultra thick] (0,0) -- (3,0) node[above, pos=0.5] {\text{up twist}};
\draw (0,-1*\nn) node {};
\end{tikzpicture}
\begin{tikzpicture}[scale=\localscale]

\draw (0.7*\nn,0.7*\nn) node[] {$+$};

\draw (-1*\nn,-1*\nn) -- (-1*\nn,1*\nn);
\draw (1*\nn,-1*\nn) -- (1*\nn,1*\nn);

\draw[line width=3bp] (-1*\nn,-1*\nn) -- (1*\nn,-1*\nn);
\draw[line width=3bp] (-1*\nn,1*\nn) -- (1*\nn,1*\nn);

\draw (0,0) node {$i_k$};
\draw (-1.2*\nn,1.2*\nn) node {L};
\draw (1.2*\nn,-1.2*\nn) node {R};
\end{tikzpicture}

\begin{tikzpicture}[scale=\localscale]

\draw (0.7*\nn,0.7*\nn) node[] {$-$};

\draw (-1*\nn,-1*\nn) -- (-1*\nn,1*\nn);
\draw (1*\nn,-1*\nn) -- (1*\nn,1*\nn);

\draw[line width=3bp] (-1*\nn,-1*\nn) -- (1*\nn,-1*\nn);
\draw[line width=3bp] (-1*\nn,1*\nn) -- (1*\nn,1*\nn);

\draw (0,0) node {$i_k$};
\draw (-1.2*\nn,1.2*\nn) node {R};
\draw (1.2*\nn,-1.2*\nn) node {L};
\end{tikzpicture}
\begin{tikzpicture}[scale=\localscale]
\draw[->,ultra thick] (0,0) -- (3,0) node[above, pos=0.5] {\text{up twist}};
\draw (0,-1*\nn) node {};
\end{tikzpicture}
\begin{tikzpicture}[scale=\localscale]

\draw (0.7*\nn,0.7*\nn) node[] {$-$};

\draw[line width=3bp] (-1*\nn,-1*\nn) -- (-1*\nn,1*\nn);
\draw[line width=3bp] (1*\nn,-1*\nn) -- (1*\nn,1*\nn);

\draw (-1*\nn,-1*\nn) -- (1*\nn,-1*\nn);
\draw (-1*\nn,1*\nn) -- (1*\nn,1*\nn);

\draw (0,0) node {$i_k$};
\draw (-1.2*\nn,1.2*\nn) node {R};
\draw (1.2*\nn,-1.2*\nn) node {L};
\end{tikzpicture}

\begin{tikzpicture}[scale=\localscale]

\draw (0,-2) node[below] {$R$};
\draw (0,2) node[above] {$L$};

\fill [gray!20] (-3,0) -- (0,-2) -- (0,2) -- cycle;

\draw[very thick, draw=black,postaction={on each segment={mid arrow=black}}] (-3,0) -- (0,2);
\draw (-3,0) -- (0,-2);

\draw (0,2) -- (3,0);
\draw[very thick, draw=black,postaction={on each segment={mid arrow=black}}] (0,-2) -- (3,0);

\draw[very thick, draw=black,postaction={on each segment={mid arrow=black}}] 
(0,2) -- (0,-2);

\draw (-1,-0.5) node[] {$\Delta_{k-1}$};
\draw (1.0,-0.5) node[] {$\Delta_{k}$};

\draw[->,ultra thick, red,densely dotted] (-1.72,0.5) -- (1.3,0.5) node[pos=0.7,above] {$\gamma$};


\end{tikzpicture}
\begin{tikzpicture}[scale=\localscale]
\draw[->,ultra thick] (0,0) -- (2,0) node[above, pos=0.5] {\text{up twist}};
\draw (0,-2.2) node {};
\end{tikzpicture}
\begin{tikzpicture}[scale=\localscale]

\draw (0,-2) node[below] {$R$};
\draw (0,2) node[above] {$L$};

\fill [gray!20] (3,0) -- (0,-2) -- (0,2) -- cycle;

\draw[very thick, draw=black,postaction={on each segment={mid arrow=black}}] (-3,0) -- (0,-2);
\draw (-3,0) -- (0,2);

\draw (0,-2) -- (3,0);
\draw[very thick, draw=black,postaction={on each segment={mid arrow=black}}] (0,2) -- (3,0);

\draw[very thick, draw=black,postaction={on each segment={mid arrow=black}}] 
(0,-2) -- (0,2);

\draw (-1,-0.5) node[] {$\Delta_{k-1}$};
\draw (1,-0.5) node[] {$\Delta_{k}$};

\draw[->, ultra thick, red,densely dotted] (-1.72,0.5) -- (1.3,0.5) node[pos=0.7,above] {$\gamma$};
\end{tikzpicture}

\caption{Two upper rows: An up twist on a snake graph tile. Bottom row: An up twist of a sequence of three steps of a $T$-path.}
\label{fig:uptwist_msw}
\end{center}
\end{figure}
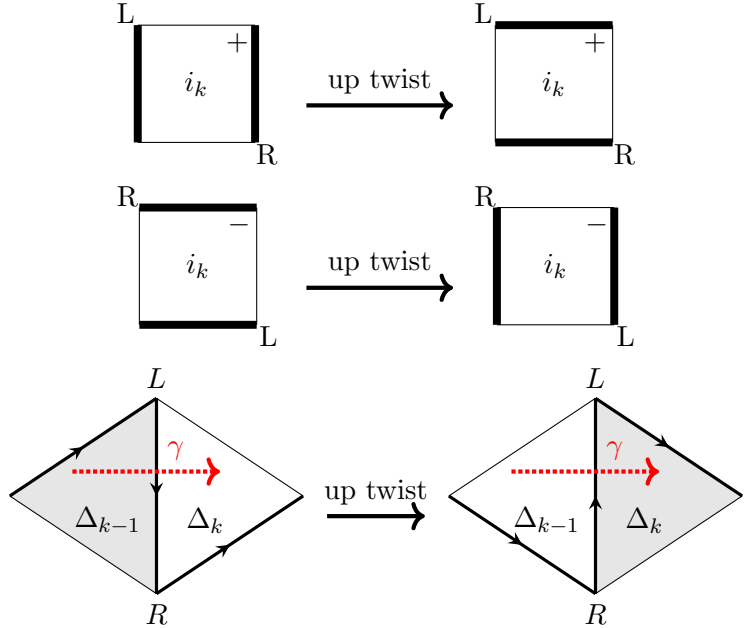

Compare this twist of a triangle in a BCI tuple with
\cite[Thm.~5.4]{MSW13}, which 
explains how to do an up twist on a tile and (consequently, due to~\cite[Thm. 4.6]{MS10}) a three-step subpath of the corresponding $T$-path. See Figure~\ref{fig:uptwist_msw}.

Even though we do not work with principal coefficients, it still makes sense to define a \emph{height function} for each BCI tuple.

\begin{defn}[height function]
Let 
\[
\tau_{R_i,1}, \tau_{R_i,2}, \dots, \tau_{R_i,f}
\]
(where $f=|FAN(R_i)|-1$)
be the arcs crossed by $\gamma$ and adjacent to $R_i$, ordered by the orientation of $\gamma$.
Given an arbitrary BCI tuple $b=(t_1,\dots,t_r)=(\triangle(R_1),\dots,\triangle(R_r))$,
we define its \emph{height function} $h(b)$ by the monomial 
\begin{equation*}
{\left( \prod_k w {\left(\tau_{R_1,k}\right)} \right)} {\left( \prod_k w {\left(\tau_{R_2,k}\right)} \right)} \dots {\left( \prod_k w {\left(\tau_{R_r,k}\right)} \right)}
\end{equation*}
where each inner product is taken over every $k$ such that $\tau_{R_i,k}$ is an edge between triangles $\triangle(R_i)$ and $First(R_i)$.
By convention, if $\triangle(R_i)=First(R_i)$, there is no edge between them, and the product equals $1$.
\end{defn}

\begin{prop}[{Analog of \cite[Theorem 5.2]{MSW13}}]\label{thm:lattice_tuple1}
Construct a graph $L_{BCI}(\gamma)$ whose vertices
are labeled by BCI tuples for $\gamma$, and whose edges connect two vertices
if and only if the two tuples are related by an up or down twist.
This graph is the Hasse diagram of a distributive lattice, whose
minimal element is the minimal BCI tuple of $\gamma$.
The lattice is graded by the degree of each height function.
\end{prop}

We describe how to read off from $T$ and $\ga$ a poset $Q_{\ga}$ whose lattice of order ideals $J(Q_\ga)$ is equal to $L_{BCI}(\gamma)$.
The following definition is equivalent to \cite[Def. 5.3]{MSW13}.

\begin{defn}[poset $Q_\gamma$]
\label{defn:poset_Q_ga}
We associate to $\widetilde{S}_\ga$ and $\ga$ a directed graph $Q_\ga$ whose vertices are labeled by $1$, $\dots$, $d$, and whose directed edges are described as follows.
Put an arrow from $k$ to $k+1$ if and only if $\taui{k+1}$ follows $\taui{k}$ (considered as sides of $\Delta_k$) in the {\ellOpOrientation} order (see Fig.~\ref{fig:lattice}).
Because $\widetilde{S}_\ga$ is a triangulated polygon, the underlying undirected graph of $Q_\ga$ is a Dynkin diagram of type $A$.
By abuse of notation, let $Q_\ga$ denote the poset whose Hasse diagram is $Q_\ga$.
\end{defn}

\def\mylightgray{gray!50}
\begin{figure}[h]
\begin{tikzpicture}[scale = .45]
		\tikzstyle{every node} = [font = \small]
		
		\foreach \x in {0,10}
        {
        	\foreach \y in {-8}
            {
            \PruferShort{\x-6}{\y-3}{10}{}
            }
        }
        
        \PruferShortCurvedLeft{20-6}{-8-3}{10}{}
        
        \draw[] (0+0,-8-3) -- (0+4,-8-3);
        
        \foreach \x in {0,10,20}
		{
			\foreach \y in {-8}
			{
				\draw[dotted] (\x-6,\y-3) -- (\x+4,\y-3);
			
				\foreach \t in {-6,-4,-2,0,2,4}
				{
					\fill (\x+\t,\y-3) circle (.1);
				}

				\fill (\x-6,\y-3) node [below] {$1$};
				\fill (\x-4,\y-3) node [below] {$2$};
				\fill (\x-2,\y-3) node [below] {$3$};
				\fill (\x+0,\y-3) node [below] {$4$};
				
				\fill (\x+2,\y-3) node [below] {$5$};
				\fill (\x+4,\y-3) node [below] {$1$};
				
				\fill (\x-5.7,\y-1) node [left] {\tiny{$\tau_1$}};
				\draw[dotted] (\x-6,\y-3) .. controls (\x-5,\y-2) and (\x-3,\y-2) .. (\x-2,\y-3);
				\fill (\x-2.3,\y-2.9) node [above] {\tiny{$\tau_2$}};
                \draw[dotted] (\x-6,\y-3) .. controls (\x-4.5,\y) and (\x+3,\y) .. (\x+4,\y-3);
                \fill (\x-0,\y-0.5) node [right] {\tiny{$\tau_4$}};
				\draw[dotted] (\x-6,\y-3) .. controls (\x-4.5,\y-1.25) and (\x-1.5,\y-1.25) .. (\x-0,\y-3);
				\fill (\x-3,\y-1.8) node [above] {\tiny{$\tau_3$}};

				\draw[dotted] (\x+0,\y-3) .. controls (\x+1,\y-2) and (\x+3,\y-2) .. (\x+4,\y-3);
				\fill (\x+2,\y-2.4) node [above] {\tiny{$\tau_0$}};	
			}
		}	
		
\draw[] (0-6,-8-3) .. controls (0-4.5,-8-1.25) and (0-1.5,-8-1.25) .. (0-0,-8-3); 

\draw[] (0-6,-8-3) .. controls (0-4.5,-8) and (0+3,-8) .. (0+4,-8-3);	
\draw[] (10-6,-8-3) .. controls (10-4.5,-8) and (10+3,-8) .. (10+4,-8-3);	
\draw[] (20-6,-8-3) .. controls (20-4.5,-8) and (20+3,-8) .. (20+4,-8-3);	

\draw[] (20-6,-8-3) .. controls (20-4.5,-8-1.25) and (20-1.5,-8-1.25) .. (20-0,-8-3); 
	
\draw[] (0+0,-8-3) .. controls (0+1,-8-2) and (0+3,-8-2) .. (0+4,-8-3); 
\draw[] (20+0,-8-3) .. controls (20+1,-8-2) and (20+3,-8-2) .. (20+4,-8-3); 
		
        \PruferShortCurvedLeft{20+4}{-8-3}{10}{}
        
        \fill (20+3.7,-8-1) node [right,black] {\tiny{$\tau_1$}};
         
\fill[violet] (0+10,-8+0.9) circle (.2);
\fill (0+10,-8+0.9) node [above,violet] {$\puncture$};
         
\fill (0+2,-8-3.2) node [above] {$\Delta_0$};
\fill (10+2,-8-3.1) node [above,\mylightgray] {$D$};

\fill (5+1,-8-3.18) node [above,\mylightgray] {$A$};
\fill (15+1,-8-3.18) node [above,\mylightgray] {$E$};

\fill (5+3.5,-8-2.9) node [above,\mylightgray] {$B$};
\fill (15+3.5,-8-2.9) node [above,\mylightgray] {$F$};

\fill (0-2,-8-0.5) node [above] {$\Delta_2$};
\fill (-2+10,-8-0.5) node [above] {$\Delta_3$};
\fill (-2+20,-8-0.5) node [above] {$\Delta_4$};  

\fill (0,-8-2.2) node [above] {$\Delta_1$};
\fill (0+10,-8-2.2) node [above,\mylightgray] {$C$}; 
\fill (0+20,-8-2.2) node [above] {$\Delta_5$};  
         
\foreach \x in {1,2,3,4,5,6,7,8}
{
\fill (2*\x+2.25,-8-4.5) node {$R_\x$};
}
\fill (0+1.8,-8-4.5) node {$s$};
\fill (20,-8-4.5) node {$t$};
\end{tikzpicture} 

\vspace {4pt}

\begin{tikzpicture}[scale=1.1,>=latex,line join=bevel,]
	\tikzstyle{every node} = [font = \small]
\foreach \y in {-10bp}
{
\foreach \x in {-40bp}
{
\node (v1) at (\x+253.0bp,\y+225.0bp) [draw,draw=none] {$1$};
\node (v2) at (\x+305.0bp-26bp,\y+279.0bp-27bp) [draw,draw=none] {$2$};
\node (v3) at (\x+305.0bp,\y+279.0bp) [draw,draw=none] {$3$};
\node (v4) at (\x+305.0bp+26bp,\y+279.0bp-27bp) [draw,draw=none] {$4$};
\node (v5) at (\x+305.0bp+52bp,\y+279.0bp) [draw,draw=none] {$5$};
\draw [black,->] (v1) -- (v2);
\draw [black,->] (v2) -- (v3);
\draw [black,->] (v4) -- (v3);
\draw [black,->] (v4) -- (v5);
}
}

\node (node_9) at 
(43.0bp,225.0bp) [draw,draw=none] {$\left(\Delta_3, A, B, C, D, \Delta_4, E, F\right)$};
  \node (node_8) at 
  (148.0bp,225.0bp) [draw,draw=none] {$\left(\Delta_2, A, B, C, D, \Delta_5, E, F\right)$};
  \node (node_7) at (95.0bp,171.0bp) [draw,draw=none] {$\left(\Delta_2, A, B, C, D, \Delta_4, E, F\right)$};
  \node (node_6) at (43.0bp,117.0bp) [draw,draw=none] {$\left(\Delta_2, A, B, C, D, \Delta_3, E, F\right)$};
  \node (node_5) at (200.0bp,171.0bp) [draw,draw=none] {$\left(\Delta_1, A, B, C, D, \Delta_5, E, F\right)$};
  \node (node_4) at (148.0bp,117.0bp) [draw,draw=none] {$\left(\Delta_1, A, B, C, D, \Delta_4, E, F\right)$};
  \node (node_3) at (95.0bp,63.0bp) [draw,draw=none] {$\left(\Delta_1, A, B, C, D, \Delta_3, E, F\right)$};
  \node (node_2) at (253.0bp,117.0bp) [draw,draw=none] {$\left(\Delta_0, A, B, C, D, \Delta_5, E, F\right)$};
  \node (node_1) at (200.0bp,63.0bp) [draw,draw=none] {$\left(\Delta_0, A, B, C, D, \Delta_4, E, F\right)$};
  \node (node_0) at (147.0bp,9.0bp) [draw,draw=none] {$\left(\Delta_0, A, B, C, D, \Delta_3, E, F\right)$};
  \node (node_10) at 
  (95.0bp,279.0bp) [draw,draw=none] {$\left(\Delta_3, A, B, C, D, \Delta_5, E, F\right)$};
  \draw [red,->] (node_0) -- (node_3) node[black,pos=0.5,left] {$R_1$};
  \draw [red,->] (node_1) -- (node_4) node[black,pos=0.5,left] {$R_1$};
  \draw [blue,->] (node_4) -- (node_5) node[black,pos=0.5,right] {$R_6$};
  \draw [red,->] (node_8) -- (node_10) node[black,pos=0.5,right] {$R_1$};
  \draw [red,->] (node_4) -- (node_7) node[black,pos=0.5,right,left] {$R_1$};
  \draw [red,->] (node_2) -- (node_5) node[black,pos=0.5,right] {$R_1$};
  \draw [blue,->] (node_7) -- (node_8) node[black,pos=0.5,left] {$R_6$};
  \draw [blue,->] (node_1) -- (node_2) node[black,pos=0.5,right] {$R_6$};
  \draw [red,->] (node_5) -- (node_8) node[black,pos=0.5,right] {$R_1$};
  \draw [blue,->] (node_9) -- (node_10) node[black,pos=0.5,left] {$R_6$};
  \draw [blue,->] (node_3) -- (node_4) node[black,pos=0.5,left] {$R_6$};
  \draw [red,->] (node_7) -- (node_9) node[black,pos=0.5,left] {$R_1$};
  \draw [blue,->] (node_0) -- (node_1) node[black,pos=0.5,right] {$R_6$};
  \draw [red,->] (node_3) -- (node_6) node[black,pos=0.5,left] {$R_1$};
  \draw [blue,->] (node_6) -- (node_7) node[black,pos=0.5,left] {$R_6$};
\end{tikzpicture}
\caption{Based on the setup of Figure \ref{fig:example_finite_cover_strip} (redrawn at the top). Left: the lattice $L_{BCI}(\gamma)$ of the BCI tuples for $\ga$. Right: the poset $Q_\ga$.}
\label{fig:lattice}
\end{figure}
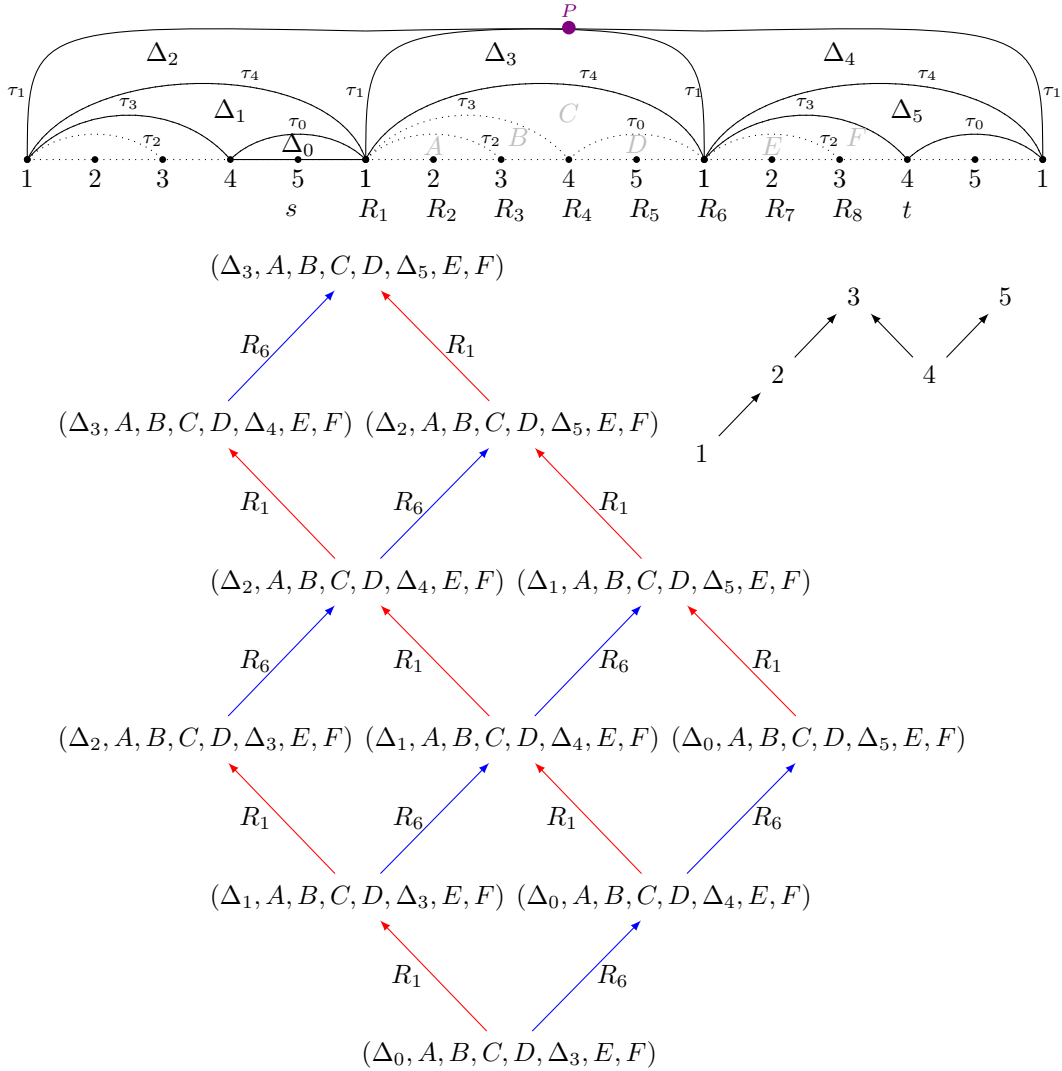

\begin{rem}
As a corollary of \cite[Theorem 5.4]{MSW13}, we have the following facts.
The lattice $L_{BCI}(\gamma)$ from Proposition \ref{thm:lattice_tuple1} is isomorphic to the lattice of order ideals $J(Q_\ga)$ of the poset $Q_\ga$ from Definition \ref{defn:poset_Q_ga};
the support of the height monomial of a BCI tuple $b$ consists precisely of the elements in the corresponding order ideal.
Moreover, an up twist of an entry of the tuple $b$ corresponds to going up in the poset. 
\end{rem}


\bibliography{citations}{} 
\bibliographystyle{alpha} 

\end{document}